\documentclass[12pt]{article}
\usepackage[margin=1in]{geometry}  

\usepackage[pdfencoding=auto]{hyperref}

\usepackage{color}

\usepackage{amsmath,amssymb}
\usepackage{braket}
\usepackage{bbm}
\usepackage{enumerate}

\usepackage{amsthm}
\theoremstyle{theorem}
\newtheorem{thm}{Theorem}
\newtheorem{prop}{Proposition}
\newtheorem{lemm}{Lemma}
\newtheorem{coro}{Corollary}
\theoremstyle{definition}
\newtheorem{defi}{Definition}

\newtheorem{exam}{Example}

\newcommand{\defarrow}{\stackrel{\mathrm{def.}}{\Leftrightarrow}}

\newcommand{\realn}{\mathbb{R}}
\newcommand{\cmplx}{\mathbb{C}}
\newcommand{\natn}{\mathbb{N}}

\newcommand{\wt}[1]{\widetilde{#1}}
\newcommand{\ovl}[1]{\overline{#1}}

\newcommand{\barPsi}{\overline{\Psi}}

\newcommand{\barL}{\overline{\Lambda}}
\newcommand{\barG}{\overline{\Gamma}}

\newcommand{\tPhi}{\widetilde{\Phi}}
\newcommand{\tG}{\widetilde{\Gamma}}
\newcommand{\tL}{\widetilde{\Lambda}}

\newcommand{\vph}{\varphi}
\newcommand{\s}{\mathrm{s}}

\newcommand{\vphx}{\varphi_x}
\newcommand{\vphxin}{(\varphi_x)_{x\in X}}

\newcommand{\A}{\mathcal{A}}
\newcommand{\E}{\mathcal{E}}
\newcommand{\F}{\mathcal{F}}

\newcommand{\Ga}{\Gamma}
\newcommand{\La}{\Lambda}

\newcommand{\DF}{\mathcal{D}(F)}

\newcommand{\Ch}{\mathbf{Ch}}
\newcommand{\Chw}{\mathbf{Ch}_{\mathrm{w}\ast}}

\newcommand{\evm}{\mathbf{EVM}}
\newcommand{\evmsub}{\mathbf{EVM}^{\mathrm{sub}}}

\newcommand{\evmsimL}{\mathbf{EVM}_{\mathfrak{sim}(\mathfrak{L})}}
\newcommand{\evmcomp}{\mathbf{EVM}_\mathrm{comp}}

\newcommand{\oA}{\mathsf{A}}
\newcommand{\oB}{\mathsf{B}}

\newcommand{\oM}{\mathsf{M}}
\newcommand{\oN}{\mathsf{N}}
\newcommand{\oK}{\mathsf{K}}

\newcommand{\GM}{\Gamma^\mathsf{M}}

\newcommand{\ME}{\mathfrak{M}(E)}
\newcommand{\MfinE}{\mathfrak{M}_{\mathrm{fin}}(E)}
\newcommand{\MmaxE}{\mathfrak{M}_{\mathrm{max}}(E)}
\newcommand{\MirrE}{\mathfrak{M}_{\mathrm{irr}}(E)}

\newcommand{\Meas}{\mathbf{Meas}(E)}
\newcommand{\meas}{\mathbf{Meas}}

\newcommand{\Ens}{\mathbf{Ens}}

\newcommand{\Pg}{P_{\mathrm{g}}}

\newcommand{\Runs}{R_{\mathrm{uns}}}
\newcommand{\Rinc}{R_{\mathrm{inc}}}

\newcommand{\conv}{\mathrm{conv}}
\newcommand{\cconv}{\overline{\mathrm{conv}}}
\newcommand{\lin}{\mathrm{lin}}
\newcommand{\cone}{\mathrm{cone}}
\newcommand{\ccone}{\overline{\mathrm{cone}}}

\newcommand{\ssimu}{\mathfrak{sim}_\mathrm{str}}
\newcommand{\simu}{\mathfrak{sim}}

\newcommand{\fL}{\mathfrak{L}}

\newcommand{\ssimuL}{\mathfrak{sim}_\mathrm{str}(\mathfrak{L})}
\newcommand{\simL}{\mathfrak{sim}(\mathfrak{L})}

\newcommand{\Ac}{A_\mathrm{c}}
\newcommand{\Ab}{A_\mathrm{b}}
\newcommand{\Aalg}{A_{\mathrm{alg}}}

\newcommand{\abs}[1]{\left| #1 \right|}

\newcommand{\cstar}{$C^\ast$}
\newcommand{\Wstar}{$W^\ast$}
\newcommand{\wstar}{$\mathrm{w}\ast$}

\newcommand{\sa}{\mathrm{sa}}

\newcommand{\unit}{\mathbbm{1}}
\newcommand{\cH}{\mathcal{H}}
\newcommand{\cK}{\mathcal{K}}
\newcommand{\LH}{\mathcal{L} (\mathcal{H})}

\newcommand{\TcH}{\mathcal{T}(\mathcal{H})}

\newcommand{\tr}{\mathrm{tr}}

\newcommand{\pp}{\preceq_{\mathrm{post}}}
\newcommand{\ppeq}{\sim_{\mathrm{post}}}

\newcommand{\pprime}{{\prime \prime}}
\newcommand{\aast}{{\ast \ast}}

\newcommand{\de}{\partial_\mathrm{e}}

\newcommand{\BX}{\mathcal{B}(X)}
\newcommand{\MX}{\mathbf{M} (X)}

\newcommand{\linf}{\ell^\infty}

\newcommand{\id}{\mathrm{id}}

\newcommand{\condi}{\mathbb{E}}

\newcommand{\iin}{{i \in I}}
\newcommand{\jin}{{j \in J}}
\newcommand{\xin}{{x \in X}}
\newcommand{\yin}{{y \in Y}}

\newcommand{\wto}{\xrightarrow{\mathrm{weakly}}}
\newcommand{\bwto}{\xrightarrow{\mathrm{BW}}}
\newcommand{\wsto}{\xrightarrow{\mathrm{weakly}\ast}}

\newcommand{\la}{\lambda}
\newcommand{\ola}{(1-\lambda)}

\newcommand{\nono}{\neq \varnothing}
\newcommand{\onon}{\varnothing \neq}

\newcommand{\loc}[1]{{\downarrow #1}}

\newcommand{\ikj}{{i(j)}}

\newcommand{\ChL}{\mathbf{Ch}^\mathfrak{L}}
\newcommand{\ChwsimL}{\mathbf{Ch}^{\mathfrak{sim}(\mathfrak{L})}}

\newcommand{\vecE}{\overrightarrow{\mathcal{E}}}

\newcommand{\interi}{\mathrm{int}}

\newcommand{\Mcomp}{\mathfrak{M}_\mathrm{comp}}
\newcommand{\Mincomp}{\mathfrak{M}_\mathrm{incomp}}
\newcommand{\McompXE}{\mathfrak{M}_\mathrm{comp}^X(E)}
\newcommand{\MincompXE}{\mathfrak{M}_\mathrm{incomp}^X(E)}

\newcommand{\Pgcomp}{P_\mathrm{g}^\mathrm{comp}}

\newcommand{\Yxin}{(Y_x)_{x\in X}}
\newcommand{\Yx}{{Y_x}}

\newcommand{\voM}{\overrightarrow{\mathsf{M}}}

\newcommand{\Gxin}{(\Gamma_x)_{x\in X}}
\newcommand{\Lxin}{(\Lambda_x)_{x\in X}}

\newcommand{\FX}{\mathbb{F}(X)}

\newcommand{\Fxin}{(F_x)_{x\in X}}

\newcommand{\Chcomp}{\mathbf{Ch}^\mathrm{comp}}

\newcommand{\ordA}{\preceq_A}
\newcommand{\ordB}{\preceq_B}

\newcommand{\Cprec}{\mathcal{C}_\preceq}

\newcommand{\bfE}{\mathbf{E}}
\newcommand{\bfF}{\mathbf{F}}

\newcommand{\phth}{\varphi_\theta}
\newcommand{\phthin}{(\varphi_\theta)_{\theta \in \Theta}}
\newcommand{\psth}{\psi_\theta}
\newcommand{\psthin}{(\psi_\theta)_{\theta \in \Theta}}
\newcommand{\thin}{{\theta \in \Theta}}

\newcommand{\Exper}{\mathbf{Exper}}

\newcommand{\qsucc}{q_{\mathrm{succ}}}

\newcommand{\Gtriv}{\Gamma_{\mathrm{triv}}}

\begin{document}

\title
{Compact convex structure of measurements and its applications to simulability, incompatibility, and convex resource theory of continuous-outcome measurements}

\author{Yui Kuramochi%
\thanks{Email: kuramochi@qi.t.u-tokyo.ac.jp}
\thanks{This work was supported by
Cross-Ministerial Strategic Innovation Promotion Program (SIP) 
(Council for Science, Technology and Innovation (CSTI)).
}
\\ \small \emph{Photon Science Center, Graduate School of Engineering,}
\\ \small \emph{The University of Tokyo, 7-3-1 Hongo, Bunkyo-ku, Tokyo 113-8656, Japan}
}
\date{}
\maketitle

\begin{abstract}
We introduce the post-processing preorder and equivalence relations 
for general measurements on a possibly infinite-dimensional
general probabilistic theory described by an order unit Banach space $E$
with a Banach predual.
We define the measurement space $\mathfrak{M}(E)$ as the set of 
post-processing equivalence classes of continuous measurements on $E .$
We define the weak topology on $\mathfrak{M} (E)$ as the weakest topology
in which the state discrimination probabilities for any finite-label ensembles 
are continuous and show that $\mathfrak{M}(E)$ equipped with the convex operation
corresponding to the probabilistic mixture of measurements
can be regarded as a compact convex set regularly embedded in
a locally convex Hausdorff space.
We also prove that the measurement space $\mathfrak{M}(E) $ is infinite-dimensional
except when the system is $1$-dimensional
and give a characterization of the post-processing monotone affine functional.
We apply these general results to the problems of simulability 
and incompatibility of measurements.
We show that the robustness measures of unsimulability and incompatibility
coincide with the optimal ratio of the state discrimination probability of measurement(s)
relative to that of simulable or compatible measurements, respectively.
The latter result for incompatible measurements generalizes 
the recent result for finite-dimensional quantum measurements.
Throughout the paper, the fact that any weakly$\ast$ continuous measurement can be arbitrarily approximated
in the weak topology 
by a post-processing increasing net of finite-outcome measurements
is systematically used to reduce the discussions to finite-outcome cases.
\end{abstract}

\vspace{2pc}

\section{Introduction} \label{sec:intro}
The measurement process is one of the indispensable
constituents of the quantum theory, or more generally any kind of operational physical theory,
since it connects the predictions by an abstract mathematical model
to the observed experimental events,
making the theory comparable with the real world. 
In spite of such a general importance, 
little is known for 
the property of the \textit{totality} of measurements of a given system.
One of the reason for this might be its mathematical difficulty,
especially that the class of measurements is a proper class,
i.e.\ a class larger than any set,
because we have no restrictions to the outcome space of a measurement.

A related important problem of the measurement we investigate in this paper is 
how we should consider continuous-outcome  measurements.
In quantum theory and technology,
continuous-outcome measurements, like the homodyne detection of a photon field,
play fundamental roles,
for example in the continuous-variable quantum key distribution~\cite{PhysRevLett.102.180504}.
We cannot however naively think that the continuous  measurement
described by a positive-operator valued measure (POVM)
is exactly realized in a real experiment
because it is impossible for an experimental device to exactly record a continuous variable, e.g.\ a real number, 
which requires infinite bits of information.
One way to reconcile such a contradiction is to think
that the theoretical description of a continuous-outcome
measurement approximates in some sense the real measurement process
which has a finite outcome space.
If we take this standpoint, 
then we have to answer in what sense this \lq\lq{}approximation\rq\rq{} is.

Another related mathematical problem is that the operation of the probabilistic mixture of two (or generally more than two) measurements that does not post-process the measurement outcome is not closed in a certain set, but is defined on the \textit{class} of measurements.
For instance, two general measurements on a quantum system have different outcome spaces $X$ and $Y$ and  the outcome space of the probabilistic mixture of the two measurements is the disjoint union of $X$ and $Y .$
Thus the outcome space becomes larger if we take probabilistic mixture and this operation cannot be closed within some set of measurements.
Presumably because of this kind of difficulty, the probabilistic mixture operation has not been sufficiently studied, 
while in some works it is natural to consider this operation.
For example, as we will see in the main part of this paper, the class of measurements simulable~\cite{doi:10.1063/1.4994303,PhysRevLett.119.190501,PhysRevA.97.062102} by a certain set of measurements and the class of pairs of compatible (i.e.\ jointly measurable) measurements~\cite{1751-8121-49-12-123001,PhysRevA.98.012133,kuramochi2018incomp,PhysRevLett.122.130402} are closed under this operation.
Moreover the state discrimination probability recently considered in the context of convex resource theory of measurements (POVMs)~\cite{PhysRevLett.122.130403,PhysRevLett.122.130404,Oszmaniec2019operational} is affine with respect to this operation.

The purpose of this paper is to study the \textit{measurement space}
$\ME,$ which is the \textit{set} of 
\textit{post-processing equivalence classes}
of measurements on a given (possibly infinite-dimensional)
order unit Banach space $E$ with a predual.
Such an ordered Banach space $E$ corresponds to 
the set of observables on the state space of 
a general probabilistic theory (GPT)~\cite{Gudder1973,hartkamper1974foundations,1751-8121-47-32-323001,PhysRevA.97.062102,PhysRevA.94.042108}.
We also apply this general formulation of measurements to the problems of the simulability and (in)compatibility of measurements.

This paper is organized as follows.
In Section~\ref{sec:prel}, we give preliminary results for order unit Banach spaces 
(GPTs).
We introduce two kinds of formulations of GPT.
The first formulation is based on compact state space
and considers the continuous affine functionals as the observables,
while the second one only requires the norm completeness of the state space
and considers the bounded affine functionals as the observables.
In this paper the former one will appear as the measurement space in the main part,
while we consider the state space of the second type as the physical system. 
This is because ordinary formulation of the quantum theory in infinite dimensions
is described by the second one, but not by the first one
since the set of density operators is not compact in the trace-norm topology in infinite dimensions.

In section~\ref{sec:meas},
we give some basic facts on measurement, which is in this paper 
defined as an abstract GPT-to-classical channel,
and post-processing relations among measurements.
The results in Section~\ref{sec:meas} is essentially the same as those in restricted situations,
for example when the system is quantum or that described by a von Neumann algebra~\cite{kuramochi2018incomp}.

In Section~\ref{sec:ccs}, based on the Blackwell-Sherman-Stein (BSS) theorem 
for measurements (Theorem~\ref{thm:bss}), 
we introduce the measurement space and the weak topology on it.
We show that the measurement space equipped with the weak topology and 
convex combination corresponding to the probabilistic mixture
can be regarded as a compact convex set in a locally convex Hausdorff space
(Theorems \ref{thm:compact} and \ref{thm:cconvME}).
We also prove that
any \wstar-measurement can be approximated by finite-outcome ones
(Theorem~\ref{thm:finapp})
and that the measurement space
$\ME$ is an infinite-dimensional convex set except when $E$ is $1$-dimensional
(Theorem~\ref{thm:infinite}).
The weak topology is known in the area of theory of statistical experiments (statistical decision theory)~\cite{lecam1986asymptotic,torgersen1991comparison},
a branch of mathematical statistics,
and our formalism contain this theory as a special case.
How the theory of statistical experiments is reduced to that of measurements is addressed in Appendix~\ref{app:se}.

In Section~\ref{sec:mono}, we consider more general class of preorders on a compact convex set that is characterized by a set of continuous affine functionals.
By the BSS theorem, the post-processing order on the measurement space, the main subject of this paper, is an example of such an order.  
We give characterizations of post-processing monotone affine functionals (Theorem~\ref{thm:monoGen} and Corollary~\ref{coro:monoME}).
Moreover, by using the condition when the order is a partial or total order (Proposition~\ref{prop:orcon}) and the infinite-dimensionality of the measurement space, we prove that the post-processing order on the measurement space is not total (Corollary~\ref{coro:nontotal}). 
Finally in Theorem~\ref{thm:vnm} we will see that the class of preorders in consideration is characterized by the independence and continuity axioms, which is a result analogous to the von Neumann-Morgenstern utility theorem~\cite{vnm1953,DUBRA2004118}.

The following Sections~\ref{sec:sim}, \ref{sec:irr}, and \ref{sec:incomp}
are devoted to the applications of the general theory of the compact convex structure
to the simulability and incompatibility of measurements.
In Section~\ref{sec:sim}, we introduce the notion of simulability based on 
the weak topology, which is a weaker notion than the previously known
simulability~\cite{doi:10.1063/1.4994303,PhysRevLett.119.190501,PhysRevA.97.062102} which we call in this paper the strong simulability.
We show that the simulability is characterized by the outperformance 
on the state discrimination probability
(Theorem~\ref{thm:simdisc}), which generalizes the finite-dimensional result
\cite{PhysRevLett.122.130403}.
As an application of Theorem~\ref{thm:simdisc}, we show a formula that characterizes the maximal success probability of simulation in terms of the state discrimination probabilities (Theorem~\ref{thm:qsucc}).
We define the robustness of unsimulability of a measurement as 
the minimal noise needed to make the measurement simulable
and prove in Theorem~\ref{thm:RoU}
that the robustness measure is the optimal ratio of the state discrimination probability of the measurement relative to that of simulable ones.

In section~\ref{sec:irr}, we consider 
related classes of extremal, maximal, and simulation irreducible measurements.
Based on the characterization of the extremality (Theorem~\ref{thm:extremal}) and simulation irreducibility (Proposition~\ref{prop:irrexmax}),
we show that any measurements is simulable 
by the simulation irreducible measurements 
(Theorem~\ref{thm:irrsim}), 
which is known in the finite-dimensional quantum systems~\cite{Haapasalo2012} and finite-dimensional GPTs~\cite{PhysRevA.97.062102}.

In Section~\ref{sec:incomp}, we consider incompatibility of measurements
and prove 
that any incompatible measurements outperform the compatible ones 
in the state discrimination task
(Theorem~\ref{thm:incompdisc}) generalizing the result for finite-dimensional quantum systems \cite{PhysRevLett.122.130402}.
We also introduce the quantity called the robustness of incompatibility for a family of measurements
as the minimal noise needed to make the measurements compatible and show 
that this quantity coincides with the optimal ratio of state discrimination probabilities 
with pre- and post-measurement information
(Theorem~\ref{thm:RoI}).
The results in Section~\ref{sec:incomp} generalize the finite-dimensional results in \cite{PhysRevLett.122.130402,PhysRevLett.122.130403,PhysRevLett.122.130404}.

Section~\ref{sec:concl} concludes the paper.

\subsection{Summary of the results in the quantum case}
Before going into the main part, for the reader not acquainted with the GPT, we describe our main results, especially Theorems~\ref{thm:RoU}, in the case of quantum measurements.

Let us fix a separable complex Hilbert space $\cH$ corresponding to the system and denote by $\LH$ and $\mathcal{T}(\cH)$ the sets of bounded and trace-class operators on $\cH ,$ respectively.
A POVM~\cite{davies1976quantum,holevo2011probabilistic,busch2016quantum} is a mapping $\oM \colon \Sigma \to \LH$ 
such that $\Sigma$ is a $\sigma$-algebra on a some set $\Omega ,$
$\oM(\Omega ) = \unit_\cH$ (the identity operator on $\cH$),
$\oM(A) \geq 0$ $(A \in \Sigma),$
and $\oM(\cup_n A_n) = \sum_{n \in \natn} \oM(A_n)$ (in the weak operator topology)
for any disjoint sequence $(A_n)_{n \in \natn} $ in $\Sigma .$
For each trace-class operator $T \in \TcH$ we define a complex measure 
$\mu_T^\oM $ on $(\Omega , \Sigma)$ by $\mu^\oM_T (E) : = \tr (T \oM(E)) .$
If $\rho$ is a density operator (i.e.\ a positive operator with unit trace), $\mu^\oM_\rho$ is the outcome probability distribution of the measurement $\oM$ when the state of the system is prepared to be $\rho .$
In this subsection we assume that all the outcome $\sigma$-algebras of POVMs are standard Borel spaces~\cite{srivastava1998course}.
Since all the results are invariant under the following notion of classical post-processing equivalence \cite{Martens1990,Dorofeev1997349,HEINONEN200577,jencova2008}, we does not lose generality by this simplification~\cite{10.1063/1.4934235}.

Let $\oM_j \colon \Sigma_j \to \LH$ be a POVM with a outcome space $(\Omega_j , \Sigma_j)$ $(j=1,2) .$
$\oM_1$ is said to be \textit{post-processing} of $\oM_2,$ 
written as $\oM_1 \pp \oM_2 ,$
if there exists a mapping $p(\cdot| \cdot) \colon \Sigma_1 \times \Omega_2 \to [0,1]$ 
such that 
\begin{enumerate}[(i)]
\item \label{i:MK1}
$p(\cdot | \omega_2) \colon \Sigma_1 \ni A \mapsto p(A|\omega_2)$
is a probability measure for all $\omega_2 \in \Omega_2 ;$
\item \label{i:MK2}
$p(A|\cdot) \colon \Omega_2 \ni \omega_2 \mapsto  p(A|\omega_2) $
is $\Sigma_2$-measurable for all $A \in \Sigma_1 ; $
\item \label{i:MK3}
$\oM_1 (A) = \int_{\Omega_2} p(A|\omega_2 ) d \oM_2 (\omega_2)$
for all $A \in \Sigma_1 .$
\end{enumerate}
A mapping $p(\cdot | \cdot)$ satisfying the above conditions \eqref{i:MK1} and \eqref{i:MK2} is called a (regular) \textit{Markov kernel}.
The relation $\oM_1 \pp \oM_2$ says that the measurement $\oM_1$ is realized by first performing $\oM_2$ and then post-processing operation corresponding to a Markov kernel.
In this sense $\oM_1$ is less informative than $\oM_2 .$
$\oM_1$ and $\oM_2$ are said to be post-processing equivalent, written as $\oM_1 \ppeq \oM_2 , $
if $\oM_1 \pp \oM_2$ and $\oM_2 \pp \oM_1$ hold.
Post-processing equivalent POVMs bring us essentially the same information on the system.
It can be shown that the class of post-processing equivalence classes of POVMs on $\cH$ forms a set, which we write as $\mathfrak{M}(\LH) $
(Proposition~\ref{prop:small}).
For each POVM $\oM,$ the equivalence class in $\mathfrak{M}(\LH)$  
to which $\oM$ belongs is denoted as $[\oM] .$
Each element $[\oM]$ of $\mathfrak{M}(\LH)$ is called a measurement.
We also define the post-processing partial order on $\mathfrak{M}(\LH)$ by 
\[
	[\oM_1 ] \pp [\oM_2] :\defarrow \oM_1 \pp \oM_2 .
\]

In this paper, almost all the important concepts and results are related to or based on the following quantity of the \textit{state discrimination probability} (or \textit{gain functional}) which is defined as follows.
For a finite set $X,$ an indexed family $\E = (\rho_x)_{x \in X}$ of positive trace-class operators on $\cH$ 
is called an ensemble if the normalization condition $\sum_\xin \tr (\rho_x) =1 $ holds.
For a finite set $X$ and a measurable space $(\Omega, \Sigma) ,$ a \textit{decision rule}
is a mapping $p (\cdot | \cdot) \colon X\times \Omega \to [0,1]$ such that 
\begin{enumerate}[(i)]
\item
$p(x|\cdot) \colon \Omega \to [0,1]$ is $\Sigma$-measurable for all $\xin ;$
\item
$\sum_\xin p(x|\omega)=1$ for all $\omega \in \Omega .$
\end{enumerate}
Let $\E = (\rho_x)_\xin$ be an ensemble and let $\oM \colon \Sigma \to \LH$ be POVM with the outcome space $(\Omega ,\Sigma) .$
We define the state discrimination probability by
\begin{equation}
	\Pg (\E ; \oM)
	:= \sup_{p : \text{ decision rule}}
	\sum_\xin \int_\Omega p(x|\omega) d \mu^\oM_{\rho_x} (\omega) .
	\label{eq:Pgintro}
\end{equation}
The operational meaning of \eqref{eq:Pgintro} is as follows.
Consider that Alice prepares the state of the system as $\tr(\rho_x)^{-1} \rho_x$
with probability $\tr(\rho_x) ,$ Bob performs the measurement $\oM$ on the system,
and then, based on the measurement outcome $\omega \in \Omega,$ Bob guesses which label $x\in X$ is prepared by Alice.
The quantity \eqref{eq:Pgintro} is then the optimal probability of the event that Bob can correctly guess the label $x.$ 
Each decision rule $p$ corresponds to Bob\rq{}s guessing strategy.

The state-discrimination functionals characterize the post-processing relation in the following sense:
for any POVMs $\oM_1$ and $\oM_2 ,$ the post-processing relation $\oM_1 \pp \oM_2$ holds if and only if $\Pg (\E ; \oM_1) \leq \Pg (\E ; \oM_2)$ for any ensemble $\E $ (the Blackwell-Sherman-Stein theorem for POVMs (Theorem~\ref{thm:bss})).
This implies that $\Pg(\E ; [\oM]) := \Pg (\E ; \oM)$ is a well-defined function on the measurement space $\mathfrak{M}(\LH) . $
We define the \textit{weak topology} on $\mathfrak{M}(\LH)$ as the weakest topology in which $\mathfrak{M}(\LH) \ni [\oM] \mapsto \Pg(\E ; [\oM])$ is continuous for all ensemble $\E .$
The weak topology is a compact Hausdorff topology.
We also define the probabilistic mixture (or convex combination) operation on $\mathfrak{M}(\LH)$ by 
\[
	[0,1] \times \mathfrak{M}(\LH) \times \mathfrak{M}(\LH) 
	\ni (\la , [\oM_1] , [\oM_2])
	\mapsto 
	[\la \oM_1 \oplus \ola \oM_2] 
	\in \mathfrak{M}(\LH) ,
\]
where each $\oM_j$ has the outcome space $(\Omega_j , \Sigma_j)$
and $\la \oM_1 \oplus \ola \oM_2$ is the POVM with the outcome space 
\begin{gather*}
	(\coprod_{j=1,2} \Omega_j , \Sigma_1 \oplus \Sigma_2) , \\
	\Sigma_1 \oplus \Sigma_2 :=
	\set{ A_1 \sqcup A_2 | A_1 \in \Sigma_1 , \, A_2 \in \Sigma_2}
\end{gather*}
defined by
\[
	(\la \oM_1 \oplus \ola \oM_2)(A_1 \sqcup A_2  )
	:=
	\la \oM_1 (A_1) + \ola \oM_2 (A_2) .
\]
Here $\sqcup$ denotes the disjoint union of sets.
The POVM $\la \oM_1 \oplus \ola \oM_2$ corresponds to the measurement realized by performing $\oM_1$ with probability $\la$ and $\oM_2$ with probability $1-\la .$
Under this convex operation and the weak topology, the measurement space $\mathfrak{M}(\LH)$ can be regarded as a compact convex set on a locally Hausdorff $V$ so that by this identification we may write as $[\la \oM_1 \oplus \ola \oM_2] = \la [\oM_1] + \ola [\oM_2] .$

The first main result (Theorem~\ref{thm:RoU}) relates the state discrimination functional and the measurement simulability, which is defined as follows.
For a set $\fL \subset \mathfrak{M}(\LH),$ 
a POVM $\oM$ on $\cH$ (or its equivalence class $[\oM]$) is \textit{simulable} by $\fL$ if there exists a measurement $[\oN] \in \cconv(\fL ) $ such that $[\oM] \pp [\oN] ,$ where $\cconv (\cdot)$ denotes the closed convex hull with respect to the weak topology.
This condition says that $\oM$ is realized by classical pre and post-processings 
of the measurement belonging to $\fL .$

We also define the \textit{robustness of unsimulability} as follows:
for a POVM $\oM$ with the outcome space $(\Omega , \Sigma)$ and a set $\fL \subset \mathfrak{M} (\LH)$ of measurements, the robustness of unsimulability is defined by
\begin{equation}
\begin{aligned}
\Runs (\oM ; \fL) := \inf_{r , \oN} \quad &  r
\\
\textrm{subject to} \quad 
& r \in [0,\infty) \\ 
&  \text{$\oN$ is a POVM with the outcome space $(\Omega, \Sigma)$}   \\
  &\text{the POVM }\frac{\oM + r \oN}{1+r} \text{ is simulable by $\fL ,$}
\end{aligned}
\label{eq:RoUpovm}
\end{equation}
where $\Runs (\oM ; \fL) : = \infty$ when the feasible region of \eqref{eq:RoUpovm}
is empty. 
This quantifies how much noise $\oN$ should be mixed to $\oM$ to make $\oM$ simulable by $\fL .$
Note that $\Runs (\oM ; \fL) = 0$ if and only if $\oM$ is simulable by $\fL .$

Theorem~\ref{thm:RoU} for POVMs states that the robustness measure~\eqref{eq:RoUpovm} can be written as 
\begin{equation}
	1+ \Runs (\oM ; \fL)
	= 
	\sup_{\E \colon \mathrm{ensemble}}  \frac{\Pg (\E ; \oM)}{\Pg (\E ; \fL)} ,
	\notag
\end{equation}
where 
\begin{equation}
	\Pg(\E ; \fL)
	:=
	\sup_{[\oN] \in \fL} \Pg (\E ;\oN)
	\notag
\end{equation}
is the optimal state discrimination probability of the ensemble $\E$ when we have ability to perform the measurements belonging to $\fL .$

We also have a similar result for the robustness of incompatibility in Theorem~\ref{thm:RoI}, which generalizes the finite-dimensional result~\cite{PhysRevLett.122.130403,PhysRevLett.122.130404}.
Since the physical significance of this result is sufficiently described in \cite{PhysRevLett.122.130403}, we do not repeat it here.
We still remark that our Theorem~\ref{thm:RoI} generalizes the previous works~\cite{PhysRevLett.122.130403,PhysRevLett.122.130404} in the points that the outcome spaces of measurements can be continuous and that the number of incompatible measurements in consideration can be infinite.

In Theorems~\ref{thm:RoU} and \ref{thm:RoI} and the BSS theorem, it is sufficient to consider finite-outcome ensembles; infinite or continuous ensembles are not necessary.
This simplicity comes from the fact that any POVM can be approximated by a post-processing increasing net of finite-outcome POVMs (Theorem~\ref{thm:finapp}).
In this sense our theory places the known result for the robustness of incompatibility in the more general theory of measurement spaces and the weak topology on it.

\section{Preliminaries} \label{sec:prel}
In this preliminary section, we review basic properties of order unit 
Banach spaces (GPTs), (compact) convex structures, and classical spaces 
as well as fix the notation.
For general references on ordered topological linear spaces, we refer to \cite{jameson1970ordered,alfsen1971compact,%
hartkamper1974foundations,schaefer1999topological}.
For a more complete review of the GPT and ordered vector spaces, 
see \cite{ddd.uab.cat:187745} (Chapter~1).

\subsection{Order unit Banach spaces (with preduals)} \label{subsec:ouB}
In this subsection we introduce the notions of the order unit Banach space
and that with a predual.
In this paper the former appears as the space of continuous affine functionals on the 
measurement space,
while the latter as the the space of observables on a physical state space. 

Throughout the paper linear spaces are assumed to be over the reals $\realn$
unless otherwise stated.
For a normed linear space $E,$
its Banach dual and double dual are denoted as $E^\ast$ and $E^{\ast \ast} ,$
respectively.
The scalar $\psi (x)$ $(x \in E , \, \psi \in E^\ast)$
is occasionally written in the bilinear form as 
$\braket{\psi ,x}$
or 
$\braket{x ,\psi} .$
For a subset $A$ of a normed linear space $E$ and $r \in [0 , \infty ) ,$ we write as
$(A)_r := \set{ x \in A | \| x \| \leq r} .$

Let $(E, F)$ be a pair of dual pair of linear spaces (e.g.\ a Banach space $E$
and its dual $E^\ast$)
separated with the bilinear form $\braket{\cdot , \cdot} \colon E \times F \to \realn .$
For a subset $A \subset E ,$ the polar $A^\circ$ and the bipolar $A^{\circ \circ}$ of $A$ in the pair $(E,F)$ 
are defined by
\[
	A^\circ 
	:= 
	\set{ y \in F | \braket{x,y} \geq -1 \, (\forall x \in A)}
\]
and
\[
	A^{\circ \circ} 
	:= 
	\set{x \in E | \braket{x,y} \geq -1 \, (\forall y \in A^\circ)},
\]
respectively.
According to the bipolar theorem, $A^{\circ \circ}$
is the $\sigma(E,F)$-closed convex hull of $A \cup \{0\} .$

A subset $K$ of a linear space $E$ is called a cone
if it satisfies
\begin{enumerate}[(i)]
\item \label{i:cone1}
$K+K \subset K ,$ 
\item  \label{i:cone2}
$\lambda K \subset K$ $(\forall \lambda \in [0,\infty ) ).$
\end{enumerate}
A subset $K \subset E$ is called a positive cone (or proper cone) 
if $K$ is a cone satisfying
\begin{enumerate}
\item[(iii)] \label{i:cone3}
$K \cap (-K) = \{ 0 \} .$
\end{enumerate}
A positive cone $K$ on $E$ induces a partial order $\leq$ by
$x \leq y$ $: \defarrow$ $y - x \in K$
$(x , y \in E) .$
An order $\leq$ on a linear space induced by a positive cone is called a linear order. 
Conversely any partial order $\leq$ on $E$
induces the positive cone $K = \set{x \in E | x \geq 0 } $ 
and the order induced by $K$ coincides with $\leq$
if
\begin{enumerate}[(a)]
\item
$x \leq y$ $\implies$ $x+z \leq y+z $ $\quad (x,y,z \in E) ,$  
\item
$x \leq y$ $\implies$ $\lambda x \leq \lambda y $ 
$\quad (x,y \in E ; \, \la \in [0 , \infty) ). $
\end{enumerate}
A linear space $E$ equipped with such a positive cone or a linear order is called
an ordered linear space. 
The positive cone of an ordered linear space $E$ is denoted by $E_+$
and 
each element of $E_+$ is called positive.

An ordered linear space $E$ is called Archimedean if 
for any $x \in E ,$
if there exists $y \in E$ such that $nx \leq y$ for all positive integer $n ,$
then $ x \leq 0 .$
A positive element $u \in E_+$ is called an order unit if for any $x \in E$
there exists $\lambda \in [0,\infty)$ such that
$ - \lambda u \leq x \leq \lambda u .$
For an Archimedean ordered linear space $E$ with an order unit $u ,$
we define the order unit norm on $E$ by
$
	\| x \|
	:= 
	\inf
	\set{ \lambda \in [ 0, \infty ) | - \lambda u \leq x \leq \lambda u }
$
$(x \in E) . $
This norm satisfies $ - \| x \| u \leq x \leq \| x \| u$
for any $x \in E.$
We call $(E, u_E)$ an order unit Banach space if $E$ is an Archimedean ordered linear 
space with the order unit $u_E$ and the order unit norm induced by $u_E$ is complete.
Throughout this paper the order unit of an order unit Banach space $E$ is always
written as $u_E .$

Let $E$ be an ordered linear space.
A convex subset $B \subset E_+ $ is called a base of the positive cone $E_+$
if for each positive element $x \in E_+$ there exists a unique 
$\lambda \in [0,\infty )$ and $b \in B$ such that $x = \lambda b .$
For $x \in E_+ - E_+ = \lin (E_+)$
(here $\lin (\cdot)$ denotes the linear span),
we define the base norm 
$\| x \|_B := \inf \set{\alpha + \beta | x = \alpha b_1 - \beta b_2 ;
\, b_1 , b_2 \in B ; \alpha , \beta \in [0,\infty )} .$
The base norm $\| \cdot \|_B$ on $\lin (E_+)$ coincides with
Minkowski functional of $\conv (B \cup (-B)) ,$ 
where $\conv (\cdot)$ denotes the convex hull.
An ordered linear space $E$ is called base-normed if $E_+$ is generating,
i.e.\ $E = \lin (E_+) ,$
and $E_+$ has a base $B .$
If the base $B$ of a base-normed space $E$ induces a complete norm,
then $E$ is called a base-normed Banach space.

For a compact convex set $K$ on a locally convex Hausdorff space $V ,$
we denote by $\Ac (K)$ the set of continuous real affine functionals
on $K .$
Then $(\Ac (K) , 1_K)$ is an order unit Banach space and the order unit norm 
coincides with the supremum norm
$\| f \| = \sup_{x \in K} | f(x)| ,$
where $1_S (\cdot) \equiv 1$ denotes the unit constant function on a set $S .$

Conversely, any order unit Banach space $(E,u_E)$ can be regarded as 
$(\Ac (K) , 1_K)$ for some compact convex set $K$ in the following way.
The dual space $E^\ast$ is an ordered linear space with the dual positive cone
$E_+^\ast := \set{\psi \in E^\ast | \braket{\psi , x} \geq 0 \, (\forall x \in E_+)}$
and a positive linear functional $ \psi \in E_+^\ast$ is called a state (on $E$) if 
$\| \psi \| =1,$
or equivalently $\braket{\psi , u_E} =1 .$
The set of states on $E$ is written as $S(E),$ 
which is a weakly$\ast$ compact convex subset of $E^\ast$
and $(E, u_E)$ is isomorphic to $(\Ac (S(E))  ,   1_{S(E)})$
by the following correspondence:
\begin{gather*}
	E \ni x \mapsto f_x \in \Ac (S(E)) ,
	\\
	f_x (\psi) := \braket{ \psi , x}
	\quad
	(x \in E , \psi \in S(E)) 
\end{gather*}
(\cite{alfsen1971compact}, Theorem~II.1.8).
The dual space $E^\ast$ is a base-normed Banach space with the base $S(E)$
and the base norm on $E^\ast$ coincides with the dual norm
$\| \psi \| = \sup_{x \in (E)_1 } |\braket{\psi , x}| $
(\cite{alfsen1971compact}, Theorem~II.1.15).

A similar base norm property also holds for a Banach predual of 
an order unit Banach space.
A Banach space $E$ is said to have a predual $E_\ast$ if $E$ is isometrically
isomorphic to the Banach dual $(E_\ast )^\ast$ of the normed linear space $E_\ast .$
We can and do take a predual $E_\ast$ as a norm closed linear subspace of 
$E^\ast $ and such $E_\ast$ is called a Banach predual of $E.$
Let $(E,u_E)$ be an order unit Banach space with a Banach predual $E_\ast .$
Then $E_\ast$ is ordered by the predual positive cone
$E_{\ast +} := \set{\psi \in E_\ast | \braket{\psi , x} \geq 0 \, (\forall x \in E_+)} .$
It is known that $E_+$ is weakly$\ast$ closed \cite{Ellis1964duality,Olubummo1999}
and hence by the bipolar theorem $E_+$ is the dual cone of $E_{\ast +} .$
The positive cone $E_{\ast +}$ generates $E_\ast $ and 
has the base $S_\ast (E) := S(E) \cap E_\ast ,$
which is the set of weakly$\ast$ continuous states on $E .$
Furthermore the base norm on $E_\ast$ induced by $S_\ast (E)$ coincides with 
the original norm \cite{Ellis1964duality,Olubummo1999}, 
i.e.\ for $\psi \in E_\ast $
\[
	\sup_{x \in (E)_1 } \abs{  \braket{\psi , x}  } =: \| \psi \|
	=
	\inf
	\set{\alpha + \beta | 
	\psi = \alpha \phi_1 -\beta \phi_2 ; 
	\, \alpha , \beta \in [0,\infty ) ; \, \phi_1 , \phi_2 \in S_\ast (E)} .
\]

An order unit Banach space with a Banach predual
can be represented as the 
set of bounded affine functionals on a convex set as follows.
We denote by $\Ab (C)$ by the set of bounded real affine functionals 
on a convex set $C.$
Then $(\Ab (C) , 1_C)$ is an order unit Banach space.
Furthermore the pointwise convergence topology on $\Ab(C)$ is defined,
which is the weakest topology such that $\Ab(C) \in f \mapsto f(x) \in \realn$
is continuous for any $x\in C.$
Now let $(E,u_E)$ be an order unit Banach space with a Banach predual $E_\ast .$
Then $E$ and $\Ab (S_\ast (E))$ are isomorphic by the correspondence
\begin{gather*}
	E \ni x \mapsto g_x \in \Ab (S_\ast (E)) ,\\
	g_x (\psi) := \braket{\psi , x } \quad (x \in E , \psi \in S_\ast (E)) .
\end{gather*}
Moreover, by this identification the weak$\ast$ topology $\sigma(E , E_\ast)$ on $E$ and the pointwise convergence topology on $\Ab (S_\ast (E))$ coincide.

\begin{exam}[Operator algebraic and quantum theories \cite{takesakivol1}]\label{ex:oa}
Let $\A$ be a \cstar-algebra with a unit $\unit_\A $
and let $\A_\sa$ denote the set of self-adjoint elements of $\A .$
By taking the ordinary positive cone 
$\A_+ := \set{a^\ast a | a \in \A} \subset \A_\sa,$
$(\A_\sa , \unit_\A)$ is an order unit Banach space
and the order unit norm on $\A_\sa$ coincides with the \cstar-norm 
restricted to $\A_\sa .$
If we further assume that $\A$ is a \Wstar-algebra,
which is a \cstar-algebra with a (unique) complex Banach predual $\A_\ast ,$
then the unique Banach predual of $\A_\sa$ is given by the self-adjoint part 
$\A_{\ast \sa}$ of $\A_\ast .$

An important example of this is the ordinary quantum theory.
Let $\cH$ be a complex Hilbert space.
Then the set $\LH$ of bounded linear operators on $\cH$ is a special
kind of \Wstar-algebra and the predual $\LH_\ast$ can be identified with 
the set $\mathcal{T} (\cH )$ of trace-class operators on $\cH$ by the bilinear form
$\braket{T , a} := \mathrm{tr} (Ta)$
$(T \in \mathcal{T} (\cH) , a \in \LH) ,$
where $\mathrm{tr} (\cdot)$ denotes the trace.
By this identification $S_\ast (\LH_\sa)$ corresponds to the set of density operators on 
$\cH .$
Note that if $\cH$ is infinite-dimensional, 
the Banach dual $\LH^\ast$ and the state space $S (\LH_\sa)$ do not coincide with
$\LH_\ast$ and $S_\ast (\LH_\sa) ,$ respectively. \qed
\end{exam}

Let $(E_i , u_{E_i})$ $(i \in I)$ be a (possibly infinite) family of order unit Banach 
spaces.
Then we can define another order unit Banach space $(\wt{E} , u_{\wt{E}} ) ,$ 
called the direct sum space, 
by
\begin{gather*}
	\wt{E}
	:=
	\set{(x_i)_{i \in I} \in \prod_{i \in I}E_i | 
	\sup_{i \in I} \| x_i \| < \infty
	} ,
	\\
	\wt{E}_+
	:=
	\set{
	(x_i)_{i \in I} \in \wt{E} 
	|
	x_i \geq 0 \, (\forall i \in I)
	}
	, \\
	u_{\wt{E}}
	:=
	(u_{E_i})_{i \in I} .
\end{gather*}
The order unit norm on $\wt{E}$ is then given by 
$\| (x_i)_{i \in I} \| = \sup_{i \in I} \| x_i \|$
$((x_i)_{i \in I} \in \wt{E} ) .$
The Banach space $\wt{E}$ is occasionally written as $\bigoplus_{i \in I} E_i .$

Suppose further that each $E_i$ has a Banach predual $E_{i \ast} .$
Then $\wt{E}$ has the predual
$\wt{E}_\ast := \set{(\psi_i)_{i \in I} \in \prod_{i \in I} E_{i \ast}  |
\sum_{i \in I} \| \psi_i \| < \infty
}$
with the bilinear form
\begin{equation}
	\braket{(\psi_i)_{i \in I} ,  (x_i)_{i \in I}}
	:=\sum_{i \in I}
	\braket{\psi_i , x_i}
	\label{eq:pair}
\end{equation}
$((\psi_i)_{i \in I} \in \wt{E}_\ast , (x_i)_{i \in I} \in \wt{E} ) .$
The positive cone $\wt{E}_{\ast+}$ and the base norm (dual norm) of $\wt{E}_\ast$
are respectively given by
\begin{gather}
	\wt{E}_{\ast +}
	=
	\set{(\psi_i)_{i \in I} \in \wt{E}_\ast | 
	\psi_i \in E_{i \ast +} \, (\forall i \in I)} ,
	\notag
	\\
	\| (\psi_i)_{i \in I} \|
	=
	\sum_{i \in I} \| \psi_i \| .
	\notag
\end{gather}
If $I$ is a finite set, the dual space $\wt{E}^\ast$ 
can be identified with $\prod_{i \in I } E_i^\ast$
with the positive cone $\prod_{i \in I } E_{i+}^\ast$
by the bilinear form \eqref{eq:pair}.
Note that this identification of $\wt{E}^\ast$ is not true when $I$ is infinite.

For a finite number of order unit Banach spaces
$(E_1 , u_{E_1}) ,  (E_2 , u_{E_2}) \dots , (E_n , u_{E_n}) ,$
the direct sum space $\wt{E}$ and each element $(x_i)_{i=1}^n \in \wt{E}$
are occasionally written as
$E_1 \oplus E_2 \oplus \dots \oplus E_n$
and 
$x_1 \oplus x_2 \oplus \dots \oplus x_n ,$
respectively.

\subsection{Abstract convex structures} \label{subsec:aconv}
The order unit Banach space and that with a predual 
introduced in Section~\ref{subsec:ouB}
can be regarded as the spaces of observables of physical systems.
Here we conversely derive these notions from 
abstract state spaces based on the line of 
Gudder \cite{Gudder1973,gudder1979stochastic}.

\begin{defi}[Convex structures] \label{def:convex}
\begin{enumerate}[1.]
\item
A set $S$ endowed with a map
\[
	[0,1] \times S \times S \ni (\lambda , s,t)
	\mapsto 
	\braket{\lambda ; s,t} 
	\in S
\]
is called a \textit{convex prestructure} \cite{Gudder1973,gudder1979stochastic} 
and $\braket{\cdot ; \cdot , \cdot}$ is called the convex combination on $S .$
We always assume that any convex subset $C$ of a linear space is equipped with 
the usual convex combination
$ \braket{\lambda ;s,t} = \lambda s + (1-\lambda ) t$
$(\lambda \in [0,1] ; s,t \in C) .$
\item
Let $(S_i , \braket{\cdot ; \cdot , \cdot}_i)$
$(i=1,2)$
be convex prestructures.
A map $\Psi \colon S_1 \to S_2$ is called \textit{affine} if
$\Psi (\braket{\lambda ; s,t}_1) = \braket{\lambda ; \Psi (s) , \Psi (t)}_2$
$(\forall \lambda \in [0,1] ; \forall s,t \in S_1) .$
An affine bijection $\Psi \colon S_1 \to S_2$ is called an affine isomorphism.
Note that if $\Psi$ is an affine isomorphism, its inverse $\Psi^{-1}$ is also affine.
\item
Let $(S,\braket{\cdot ; \cdot , \cdot})$ be a convex prestructure.
An affine map $f \colon S \to \realn$ is called an \textit{affine functional} on $S.$
We denote by $\Ab (S)$ the set of bounded affine functionals on $S .$
$\Ab (S)$ endowed with the supremum norm $\| f \| := \sup_{s \in S} \abs{f(s)}$
is a Banach space.
If $S$ is a topological space, 
we denote by $\Ac (S)$ the set of continuous affine functionals on $S.$
\item
A convex prestructure $(S,\braket{\cdot ; \cdot , \cdot})$
is called a \textit{compact convex structure} if $S$ is a compact Hausdorff topological 
space and $\Ac (S)$ separates points of $S,$
i.e.\
for any $s, t \in S ,$
$f(s) = f(t)$
$(\forall f \in \Ac (S))$
implies $s=t .$
\item
A convex prestructure $(S,\braket{\cdot ; \cdot , \cdot})$
is called a \textit{norm-complete convex structure}
if $\Ab (S)$ separates points of $S$ and the metric $d$ on $S$ defined by
\[
	d (s,t) := \sup \set{\abs{f(s) - f(t)} | f \in \Ab (S) , \| f \| \leq 1 }
	\quad 
	(s,t \in S)
\]
is complete. \qed
\end{enumerate}
\end{defi}
The notion of compact (norm-complete) convex structure 
corresponds to that of order unit Banach space (with a Banach predual)
as in the following proposition

\begin{prop} \label{prop:convst}
\begin{enumerate}[1.]
\item \label{i:convst1}
Let $(S,\braket{\cdot ; \cdot , \cdot})$ be a compact convex structure.
Then $(\Ac (S) , 1_S)$ endowed with the positive cone
$\Ac (S)_+ := \set{f \in \Ac (S) | f (s) \geq 0 \, (\forall s \in S)}$
is an order unit Banach space.
If we define 
$
	\Psi \colon S \ni s \mapsto \Psi (s) \in \Ac (S)^\ast
$
by 
\[
\braket{\Psi (s) , f} : =f (s)
\quad
(s \in S , f \in \Ac (S)),
\]
then the map $\Psi$ is a continuous affine isomorphism between 
$S$ and $S (\Ac (S))  $ so that we can identify 
$S$ with $S (\Ac (S)) . $
\item \label{i:convst2}
Let $(S,\braket{\cdot ; \cdot , \cdot})$ be a norm-complete convex structure.
Then $(\Ab (S) , 1_S)$ endowed with the positive cone 
$\Ab (S)_+ := \set{f \in \Ab (S) | f(s) \geq 0 \, (\forall s \in S)} $
is an order unit Banach space.
The map $\Phi \colon S \to \Ab (S)^\ast$
defined by 
\[
\braket{\Phi (s) , f} := f(s)
\quad
(s \in S , f \in \Ab (S))
\]
is an isometry so that $S$ may be identified with the norm-closed 
convex subset $\Phi (S)$ of $\Ab (S)^\ast .$
The linear subspace $E_\ast := \lin (S) \subset \Ab (S)^\ast$
is a Banach predual of $\Ab (S)$
and $S$ coincides with the base $S_\ast (\Ab (S))$ of $E_\ast .$
\end{enumerate}
\end{prop}

Proposition~\ref{prop:convst}.\ref{i:convst2} is what is called in \cite{ddd.uab.cat:187745} Ludwig's embedding theorem \cite{ludwig1985derivation}
(IV, Theorem~3.7).
The claim~\ref{i:convst1} can be shown similarly as claim~\ref{i:convst2}.
For completeness short proofs are included in Appendix~\ref{app:convst}.

We can rephrase Proposition~\ref{prop:convst}.1
in terms of the regular embedding (\cite{alfsen1971compact}, Section~II.2).
For a compact convex structure $(S , \braket{\cdot ; \cdot , \cdot} ) ,$
a continuous affine injection
$\Psi \colon S \to V$ into a locally convex Hausdorff space $E$
is called a regular embedding
if $E = \lin (\Psi (S)) $ and $0 \notin \mathrm{aff}(\Psi(S)) ,$
where $\mathrm{aff} (\cdot )$ denotes the affine hull.
If such $\Psi$ exists, $S$ is said to be regularly embedded into $E .$

\begin{prop} \label{prop:emb}
Let $(S , \braket{\cdot ; \cdot , \cdot} ) $ be a compact convex structure
and let $\Psi \colon S \to \Ac (S)^\ast$ be the map in Proposition~\ref{prop:convst}.
Then $\Psi $ is a regular embedding into $\Ac (S)$ equipped with the 
weak$\ast$ topology.
Furthermore, such a regular embedding is unique in the following sense:
if $\Phi \colon S \to E$ is another regular embedding into a locally convex Hausdorff
space $E ,$ there exists a continuous linear isomorphism 
$J \colon E \to \Ac (S)^\ast$ such that 
$\Psi = J \circ \Phi .$
\end{prop}
\begin{proof}
The first part of the claim is immediate from Proposition~\ref{prop:convst}
and from that $S (\Ac (S)) = \Psi (S)$ is a base of the positive cone $\Ac (S)_+^\ast .$
It also follows that
$\Ac (S)$ separates points of $S,$
$1_S \in \Ac (S) ,$ 
and $S(\Ac (S)) = S .$
Therefore $(S , \Ac (S))$ is an abstract convex in the sense of \cite{alfsen1971compact}
(Section~II.2) and the rest of the claim follows from Theorem~II.2.4 of 
\cite{alfsen1971compact}.
\end{proof}
As we have seen in Sections~\ref{subsec:ouB} and \ref{subsec:aconv},
the order unit Banach space and the convex state space are dual notions
and we can always translate a general statement on the one side to the other.
In physical terms, these notions correspond to the descriptions of the systems
in the Heisenberg and Schr\"odinger pictures, respectively.
In what follows in this paper we mainly consider order unit Banach spaces
with preduals as the spaces of observables of physical systems,
while the measurement space will be introduced in Section~\ref{sec:ccs} as a 
special kind of compact convex structure.

\subsection{Classical space}
\label{subsec:classical}
An order unit Banach space $E$ is called \textit{classical}
if it satisfies either (all) of the following equivalent conditions
(\cite{alfsen1971compact}, Theorem~II.4.1):
\begin{enumerate}[(i)]
\item
The state set $S(E)$ is a Bauer simplex, i.e.\
the set $\de S(E)$ of extremal points (or pure states) of $S(E)$ is compact and 
any $\phi \in S(E)$ is a barycenter of a unique simplicial boundary measure \cite{alfsen1971compact}.
\item
$(E,u_E)$ is ismorphic to 
$(C(X) , 1_X) $
for some compact Hausdorff space $X$
as an order unit Banach space,
where $C(X)$ denotes the set of real continuous functions on $X$
equipped with the positive cone
$C(X)_+ = \set{f \in C(X) | f(x) \geq 0 \, (\forall x \in X)} .$
\item
The partially ordered set $(E , \leq )$ is a lattice.
\end{enumerate}

Now, by generalizing the finite-dimensional result in 
\cite{barnum2006cloning} (Corollary~1),
we give another characterization of a classical space 
in terms of a well-behaving product operation,
or a universal broadcasting channel as in the following proposition.
See also \cite{torgersen1991comparison} (Corollary~5.7.9)
for the uniqueness part.
\begin{prop}
\label{prop:nb}
An order unit Banach space $E$ is classical if and only if 
there exists a bilinear map $B \colon E \times E \to E$ such that
\begin{enumerate}[(i)]
\item \label{i:nb1}
(broadcasting property)
$B(a , u_E) = B(u_E , a) = a$
$(\forall a \in E) ;$
\item \label{i:nb2}
(bipositivity)
$B(a,b) \geq 0$ $(\forall a, b \in E_+) .$
\end{enumerate}
Furthermore, such a bilinear map $B$ is, if exists, unique
and satisfies the commutativity $B(a,b) = B(b,a)$
and the associativity $B(B(a,b) , c) = B(a, B(b,c))$
$(\forall a, b, c \in E) .$
\end{prop}
The proof of Proposition~\ref{prop:nb} is analogous to the finite-dimensional case
\cite{barnum2006cloning} and to the Gelfand's representation theorem for abelian $C^\ast$-algebras \cite{takesakivol1}.
See Appendix~\ref{app:nb} for detail.

For a classical space $E,$ the unique bilinear map $B(a,b)$ in Proposition~\ref{prop:nb}
is written as $a \cdot b $ $(a, b \in E)$ and called the product on $E.$ 

A possibly infinite direct sum of classical spaces is also classical.
If $E$ is a classical space with a Banach predual $E_\ast ,$
then $E$ is isomorphic to the set of self-adjoint elements of an abelian \Wstar-algebra. 
By the uniqueness of the complex Banach predual of a \Wstar-algebra, 
the Banach predual of a classical space is, if exists, unique.
The double dual $E^\aast$ of a classical space $E$ is also a classical space.

An element $P $ of a classical space $E$ is called a projection if $P \cdot P = P .$
A projection $P$ always satisfies $0 \leq P \leq u_E .$
The following proposition, which is immediate from the general properties from 
\Wstar-algebras (von Neumann algebras), will be used in the main part.
\begin{prop} \label{prop:proj}
Let $E$ be a classical space with a Banach predual $E_\ast .$
\begin{enumerate}[1.]
\item
For each $a \in E$ there exists a sequence $(a_n)_{n \in \natn}$ 
of finite sums of projections on $E$ such that $\| a - a_n \| \to 0 .$
\item
For each positive weakly$\ast$ continuous linear functional
$\vph \in E_{\ast +} $ there exists the smallest projection
$P \in E$ 
such that $\braket{\vph , P} = \| \vph \| .$
Such $P$ is called the support projection of $\vph$ and written as 
$\s (\vph) .$
\end{enumerate}
\end{prop}

\section{Measurements} \label{sec:meas}
In this section we introduce and prove basic facts on the channels, measurements, and the post-processing order and equivalence relations between them.

\subsection{Channels and post-processing relations} \label{subsec:chpp}
Before introducing measurements, we consider more general class of channels 
between order unit Banach spaces.

For simplicity, in what follows in this paper, 
if we say that $E$ is an order unit Banach space, 
we understand that $E $ is endowed with an order unit which is written as $u_E .$

A linear map $\Psi \colon E \to F$ between ordered linear spaces 
$E$ and $F$ is called positive if $\Psi (E_+) \subset F_+ .$
If $E$ and $F$ are order unit Banach spaces,
a positive linear map $\Psi \colon E \to F$ that is unital, i.e.\
$\Psi (u_E) = u_F ,$
is called a channel (in the Heisenberg picture).
The domain $E$ and the codomain $F$ of a channel $\Psi \colon E \to F$
are called the outcome and input spaces of $\Psi  ,$ respectively.
We write the set of channels from $E$ to $F$ as
$\Ch (E\to F) .$

\begin{prop}\label{prop:pos}
Let $E$ and $F$ be order unit Banach spaces.
Then any positive linear map $\Psi \colon E \to F$ is bounded and 
the uniform norm is given by
$\| \Psi \| = \| \Psi (u_E) \| .$
If $\Psi$ is channel, then $\| \Psi \| = 1 .$
\end{prop}
\begin{proof}
For any $a \in E ,$ we have 
$-\| a \| u_E \leq a \leq \| a \| u_E $
and the positivity of $\Psi $ implies
$- \| a \| \Psi (u_E) \leq \Psi (a) \leq  \| a \| \Psi (u_E) $
and hence
$-\| a \| \| \Psi (u_E)\| u_F \leq \Psi (a) \leq \| a \| \| \Psi (u_E)\| u_F .$
Thus $\| \Psi (a) \| \leq \| \Psi (u_E)\| \| a \|$
and we obtain $\| \Psi \| \leq \| \Psi (u_E)\| . $
Since $\| \Psi (u_E) \| \leq \| \Psi \|$ is obvious, the first part of the claim is proved.
If $\Psi $ is a channel, then $\| \Psi \| = \| \Psi (u_E) \| = \| u_F \| = 1.$
\end{proof}
Let $E$ be a Banach space and let $F$ be a Banach space with a Banach predual 
$F_\ast .$
Then the set $\mathcal{L} (E \to F)$
of bounded linear maps from $E$ to $F$ is endowed with a locally convex Hausdorff
topology called the BW-topology (\cite{paulsen_2003}, Chapter~7) in the following way.
The BW-topology is the weakest topology such that
\[
	\mathcal{L} (E \to F) \ni \Psi
	\mapsto 
	\braket{\psi , \Psi (a)} \in \realn 
\]
is continuous for any $a \in E$ and any $\psi \in F_\ast .$
A net $(\Psi)_{i\in I}$ in $\mathcal{L} (E \to F)$ is BW-convergent to 
$\Psi \in \mathcal{L} (E \to F)$ if and only if
$\braket{\psi , \Psi_i (a) } \to \braket{\psi , \Psi  (a) }$
for any $a \in E$ and $\psi \in F_\ast ,$
or equivalently $\Psi_i (a ) \xrightarrow{\text{weakly$\ast$}} \Psi (a)$
for any $a \in E .$
It follows from Tychonoff\rq{}s theorem that 
the closed unit ball $(\mathcal{L} (E \to F)  )_1$ is  
BW-compact.
\begin{prop} \label{prop:BWcompact}
Let $E$ and $F$ be order unit Banach spaces.
Suppose that $F$ has a Banach predual $F_\ast .$
Then $\Ch (E\to F)$ is a BW-compact convex subset of $\mathcal{L} (E \to F) .$
\end{prop}
\begin{proof}
It is easy to show the convexity of $\Ch (E \to F) .$
Since $\Ch (E \to F)$ is a subset of the BW-compact set 
$(\mathcal{L} (E \to F)  )_1 $ by Proposition~\ref{prop:pos},
it suffices to show that $\Ch (E \to F)$ is BW-closed
and this follows from the weak$\ast$ closedness of the positive cone $F_+ .$
\end{proof}
Let $E$ and $F$ be order unit Banach spaces with preduals
$E_\ast $ and $F_\ast ,$ respectively.
Then a weakly$\ast$ continuous channel $\Psi \colon E \to F$
is briefly called a \wstar-channel.
The set of \wstar-channels from $E$ to $F$ is denoted by 
$\Chw (E\to F) .$
For a \wstar-channel $\Psi \colon E \to F ,$
there exists a unique bounded linear map
$\Psi_\ast \colon F_\ast \to E_\ast$
such that
\begin{equation}
\braket{\psi , \Psi (a)} = \braket{\Psi_\ast (\psi) , a}
\label{eq:chst}
\end{equation}
$(a \in E , \psi \in F_\ast ) .$
This map satisfies $\Psi_\ast (S_\ast (F)) \subset S_\ast (E) .$
Conversely for each affine map $\Psi_\ast \colon S_\ast (F) \to S_\ast (E)$
there exists a unique \wstar-channel $\Psi \colon E \to F$
satisfying \eqref{eq:chst} for any $a\in E$ and $\psi \in S_\ast (F) .$
The above map $\Psi_\ast$ is called the predual of $\Psi$ corresponds to the channel in the Schr\"{o}dinger picture.

We now introduce the post-processing relations for channels.
\begin{defi} \label{def:ppch}
Let $\Psi \in \Ch (F \to E)$ and $\Phi \in \Ch (G \to E)$ be channels 
with the same input space $E .$
\begin{enumerate}[1.]
\item
$\Psi$ is said to be a \textit{post-processing} of $\Phi ,$
written as $\Psi \pp \Phi ,$
if there exists $\Lambda \in \Ch (F \to G)$ such that
$\Psi = \Phi \circ \Lambda .$
\item
$\Psi$ is said to be \textit{post-processing equivalent} to $\Phi,$
written as $\Psi \ppeq \Phi ,$ if
$\Psi \pp \Phi$ and $\Phi \pp \Psi $ hold. \qed
\end{enumerate}
\end{defi}
By noting that any composition of channels is again a channel,
we can easily see that the relations $\pp$ and $\ppeq$ are respectively
binary preorder and equivalence relations defined on the class of channels with a fixed input space.

We next introduce the \wstar-extension of a channel.
For this we need the following characterization of the double dual Banach space.
As usual, we regard every normed linear space $E$ as a linear 
subspace of the double dual
Banach space $E^{\ast \ast} .$ 
\begin{prop} \label{lemm:double}
Let $E$ be a Banach space, let $F$ be a Banach space with a Banach predual 
$F_\ast ,$
and let $\Psi \colon E \to F$ be a bounded linear map.
Then $\Psi $ is uniquely extended to a weakly$\ast$ continuous
(i.e.\ $\sigma(E^{\ast \ast} , E^\ast) / \sigma (F, F_\ast )$-continuous)
linear map $\overline{\Psi} \colon E^{\ast \ast} \to F .$
The map $\overline{\Psi}$ is called the \wstar-extension of $\Psi .$
\end{prop}
\begin{proof}
Let $\Psi^\ast \colon F^\ast \to E^\ast$ be the dual map of $\Psi$
and let $\Phi \colon F_\ast \to E^\ast$ be the restriction of $\Psi^\ast$
to $F_\ast (\subset (F_\ast)^{\ast \ast} = F^\ast ) . $
We define $\ovl{\Psi }\colon E^{\ast \ast} \to F (= (F_\ast)^\ast)$
by the dual map of $\Phi .$
Then $\ovl{\Psi}$ is $\sigma(E^{\ast \ast} , E^\ast) / \sigma (F, F_\ast )$-continuous
by definition.
Furthermore for any $a \in E$ and $\psi \in F_\ast $
\[
	\braket{\psi ,\ovl{\Psi} (a) }
	=
	\braket{\Phi(\psi) , a}
	=
	\braket{\Psi^\ast (\psi) , a}
	=
	\braket{\psi , \Psi (a)} ,
\]
which implies $\ovl{\Psi} (a)= \Psi (a)$ $(a\in E).$
Therefore $\ovl{\Psi}$ satisfies the required conditions of the claim.
The uniqueness of $\ovl{\Psi}$
follows from the weak$\ast$ density of $E$ in $E^{\ast \ast} .$
\end{proof}
If $E$ is an order unit Banach space, the double dual space $E^{\ast \ast}$
with the order unit $u_{E^{\ast \ast}} = u_E$ and the double dual positive cone
$E_+^{\ast \ast}:= \set{ a^\pprime \in E^\aast |
\braket{\psi , a^\pprime} \geq 0 \, (\forall \psi \in E^\ast_+)
} $
is an order unit Banach space with the Banach predual $E^\ast .$
Then we have $E_+ = E^\aast_+ \cap E ,$ 
i.e.\ the orders on $E$ and $E^\aast$ are consistent.
Moreover by the bipolar theorem $E_+$ is a weakly$\ast$ dense subset of $E_+^\aast .$

\begin{prop} \label{prop:w*ext}
Let $\Psi \in \Ch (E \to F)$ be a channel.
Suppose that the order unit Banach space $F$ 
has a Banach predual $F_\ast .$
Then the \wstar-extension $\ovl{\Psi} \colon E^\aast \to F$ of $\Psi $
is a \wstar-channel.
\end{prop}
\begin{proof}
The unitality of $\ovl{\Psi}$ follows from $\ovl{\Psi} (u_E) = \Psi (u_E) = u_F .$
To show the positivity, 
take an element $a^\pprime \in E_+^\aast .$
Then there exists a net $(a_i)_{i \in I} $ in $E_+$ weakly$\ast$ converging to
$a^\pprime .$
Then since the positive cone $F_+$ is weakly$\ast$ closed, 
we have $\ovl{\Psi} (a^\pprime) = \lim_{i \in I} \Psi (a_i) \in F_+ ,$
where the limit is with respect to $\sigma (F,F_\ast) .$
Therefore $\Psi$ is a \wstar-channel.
\end{proof}
The following proposition implies that the \wstar-extension of a channel
is the least channel in the post-processing order that 
upper bounds the original channel (cf.\ \cite{kuramochi2018incomp}, Lemma~7).
\begin{prop} \label{prop:w*extch}
Let $\Psi \in \Ch (E\to F)$ and $\ovl{\Psi} \in \Chw (E^\aast \to F)$
be the same as in Proposition~\ref{prop:w*ext}.
Then for any \wstar-channel $\Phi \in \Chw (G \to F),$
where $G$ has a Banach predual $G_\ast ,$ 
$\Psi \pp \Phi $ if and only if
$\barPsi \pp \Phi .$
\end{prop}
\begin{proof}
Assume $\Psi \pp \Phi .$
Then there exists a channel $\Lambda \in \Ch (E \to G)$
such that $\Psi = \Phi \circ \Lambda .$
Let $\ovl{\Lambda} \in \Chw (E^\aast \to G)$ be the \wstar-extension of $\Lambda .$
Then $\Phi \circ \ovl{\Lambda} \in \Chw (E^\aast \to F) $ 
and for any $a\in E$ we have
$\Phi \circ \ovl{\Lambda} (a) = \Phi \circ \Lambda (a) = \Psi (a) .$
Therefore the uniqueness of the \wstar-extension implies 
$\barPsi = \Phi \circ \ovl{\Lambda} \pp \Phi .$
The converse implication follows from $\Psi \pp \barPsi ,$ 
which holds because $\Psi$
is the restriction of $\barPsi$ to $E .$
\end{proof}
Let $\Phi \in \Ch (F \to E)$ and $\Psi \in \Ch (G \to E)$ be channels.
For $\lambda \in [0,1]$ we define the direct convex combination channel 
$\lambda \Phi \oplus (1-\lambda) \Psi \in \Ch (F \oplus G \to E)$
by
\[
	[ \lambda \Phi \oplus (1-\lambda ) \Psi]
	(a \oplus b)
	:=
	\lambda \Phi (a) + (1-\lambda ) \Psi (b) 
	\quad
	(a\oplus b \in F \oplus G ) .
\]
The channel $\lambda \Phi \oplus (1-\lambda ) \Psi$ corresponds 
to performing $\Phi$ and $\Psi$ independently with probabilities $\lambda$
and $1-\lambda ,$ respectively.
If $\Phi$ and $\Psi$ are \wstar-channels, 
so is $\lambda \Phi \oplus (1-\lambda ) \Psi .$
As mentioned in Section~\ref{sec:intro}, this operation is not closed in a set, but 
in this case defined on the class of channels with a fixed input space.

The first claim of the next proposition indicates that the convex operation is consistent with the post-processing order.
\begin{prop} \label{prop:direct}
\begin{enumerate}[1.]
\item \label{i:direct1}
Let $\Phi_i \in \Ch (F_i \to E)$ and $\Psi_i \in \Ch (G_i \to E)$
$(i=1,2)$ be channels.
Then $\Phi_1\pp \Phi_2$ and $\Psi_1 \pp \Psi_2$ imply
$\lambda \Phi_1 \oplus (1-\lambda ) \Psi_1 
\pp \lambda \Phi_2 \oplus (1-\lambda ) \Psi_2$
for any $\lambda \in [0,1] .$
\item \label{i:direct2}
If $\Psi , \Phi \in \Ch (F\to E )$ are channels with the common input and 
outcome spaces, then 
$\lambda \Psi + (1-\lambda ) \Phi \pp \lambda \Psi \oplus (1-\lambda ) \Phi $
for any $\la \in [0,1].$
\end{enumerate}
\end{prop}
\begin{proof}
\begin{enumerate}[1.]
\item
By assumption there exist channels 
$\Theta \in \Ch (F_1 \to F_2)$
and 
$\Xi \in \Ch (G_1 \to G_2)$
such that $\Phi_1 = \Phi_2 \circ \Theta $
and $\Psi_1 = \Psi_2 \circ \Xi .$
We define $\Omega \in \Ch (F_1 \oplus F_2 \to G_1 \oplus G_2)$
by 
$\Omega (a \oplus b ) := \Theta (a) \oplus \Xi (b) $
$(a \oplus b \in F_1 \oplus F_2) .$
Then it readily follows that
\[
\lambda \Phi_1 \oplus (1-\lambda ) \Psi_1  =
[\lambda \Phi_2 \oplus (1-\lambda ) \Psi_2 ] \circ \Omega 
\pp
\lambda \Phi_2 \oplus (1-\lambda ) \Psi_2 . 
\]
\item
Define a channel $\Delta \in \Ch (F \to F \oplus F)$
by
$\Delta (a) := a \oplus a$
$(a \in F) .$
Then for each $a \in F ,$
$[\lambda \Psi \oplus (1-\lambda ) \Phi] \circ \Delta (a)
= \lambda \Psi (a) + (1-\lambda ) \Phi (a) ,
$
which implies 
$\lambda \Psi + (1-\lambda ) \Phi =
[\lambda \Psi \oplus (1-\lambda ) \Phi] \circ \Delta
\pp
\lambda \Psi \oplus (1-\lambda ) \Phi . 
$ \qedhere
\end{enumerate}
\end{proof}

\subsection{Measurements} \label{subsec:meas}
In this paper we consider \wstar-measurement as abstract 
GPT-to-classical channels, generalizing the quantum-to-classical channels.
This kind of formulation, 
rather than the ordinary way of considering 
POVMs or effect-valued measures (EVMs) (e.g.\ \cite{PhysRevA.97.062102}), 
is useful for developing the general theory of measurements
as in the succeeding sections. 

In the rest of this paper, unless otherwise stated, 
\textit{we fix an input order unit Banach space $(E , u_E)$
and its Banach predual $E_\ast .$}

A channel $\Psi \in \Ch (F\to E)$ is said to be a \textit{measurement}
if the outcome space $F$ is classical.
When $F$ is a classical space with a Banach predual,
then a \wstar-channel $\Psi \in \Chw (F\to E)$ is called a
\textit{\wstar-measurement}.
If we say that $\Psi \in \Chw (F \to E)$ is a \wstar-measurement,
we understand that $F$ is a classical space with the Banach predual $F_\ast .$

\begin{prop} \label{prop:postw*}
Let $\Psi \in \Chw (F \to E)$ and $\Phi \in \Chw (G \to E)$ be \wstar-measurements.
Then $\Psi \pp \Phi$ if and only if there exists a \wstar-channel 
$\Gamma \in \Chw (F \to G)$ such that $\Psi = \Phi \circ \Gamma .$
\end{prop}
\begin{proof}
\lq\lq{}If\rq\rq{} part of the claim is obvious.
Assume $\Psi \pp \Phi .$
Then by the proof of Proposition~\ref{prop:w*extch} there exists a \wstar-channel
$\Lambda \in \Chw (F^\aast \to G)$ such that
$\barPsi = \Phi \circ \Lambda ,$
where $\barPsi \in \Chw (F^\aast \to E)$ is the \wstar-extension of $\Psi . $
From \cite{gutajencova2007} (in the proof of Lemma~3.12),
there exists a \wstar-channel $\Xi \in \Chw (F \to F^\aast)$
such that
$\braket{\vph , \Xi (a)} = \braket{\vph , a}$
$(\vph \in F_\ast , a \in F) .$
Then for $a\in F$ and $\psi \in E_\ast $ 
\[
	\braket{\psi , \barPsi \circ \Xi (a)}
	=
	\braket{\Psi^\ast (\psi) , \Xi (a)}
	=
	\braket{\Psi^\ast (\psi) , a}
	=
	\braket{ \psi , \Psi (a)} ,
\]
where we used $\Psi^\ast (\psi) \in F_\ast $ in the second equality.
This implies $\Psi = \barPsi \circ \Xi = \Phi \circ \Lambda \circ \Xi .$
Since $\Lambda \circ \Xi $ is a \wstar-channel, this proves the \lq\lq{}only if\rq\rq{}
part of the claim.
\end{proof}
As we can see from the proof, 
Proposition~\ref{prop:postw*} still holds when the outcome spaces $F$ and $G$ 
are relaxed to the self-adjoint parts of 
arbitrary \Wstar-algebras.

The above definition of \wstar-measurement is related to the more common notion of 
normalized EVM.
A triple $(X , \Sigma , \oM)$ is said to be an EVM on $E$ if
$\Sigma$ is a $\sigma$-algebra on a set $X$ and 
$\oM \colon \Sigma \to E_+$ is a map such that
\begin{enumerate}[(i)]
\item
$\oM (X) = u_E ,$ $\oM(\varnothing) = 0 ,$
\item
for any disjoint and countable family $(A_k)_{k \in \natn}$
$(\natn := \set{1,2, \dots})$
in $\Sigma ,$
$\oM( \bigcup_{k \in \natn} A_k) = \sum_{k \in \natn} \oM(A_k) ,$
where the RHS converges weakly$\ast$.
\end{enumerate}
For $\psi \in E_\ast$
(respectively, $\psi \in S_\ast (E)$)
the function 
$\mu^\oM_\psi \colon \Sigma \ni A \mapsto \braket{\psi , \oM(A)} \in \realn $  
is a signed (respectively, probability) measure.
Conversely for any affine map
\[
	S_\ast (E) \ni \psi \mapsto \nu_{\psi}
\]
that maps each weakly$\ast$ continuous state to a probability measure on a measurable space $(X, \Sigma),$ there exists a unique EVM $(X, \Sigma , \oM)$ such that $\nu_\psi = \mu^\oM_\psi$ $( \psi \in S_\ast (E)) .$

For a measurable space $(X , \Sigma) ,$ we denote by 
$B(X, \Sigma)$ the set of real bounded $\Sigma$-measurable 
functions on $X .$ 
Then the order unit Banach space $(B(X, \Sigma) , 1_X)$ equipped with the positive cone 
\[
B(X,\Sigma)_+ = \set{f \in B(X,\Sigma) | f(x) \geq 0 \, (\forall x \in  X)}
\]
is a classical space.

Let $(X, \Sigma , \oM)$ be an EVM on $E .$ 
For each function $f \in B(X , \Sigma) ,$ the integral
$\int_X f(x) d\oM (x) \in E $ is well-defined by
\[
	\Braket{ \int_X f(x) d\oM (x) , \psi}
	:=
	\int_X f(x) d\mu^\oM_\psi (x) 
	\quad
	(\psi \in E_\ast) .
\]
Then the map
\[
	\gamma^\oM \colon 
	B(X, \Sigma)
	\ni f \mapsto
	\int_X f(x) d\oM (x) \in E
\]
is a measurement and 
called the measurement associated with the EVM $\oM .$
The \wstar-extension $\Gamma^\oM \in \Chw (B(X, \Sigma)^\aast \to E)$
of $\gamma^\oM$ is called the \wstar-measurement associated with $\oM .$
Thus for each EVM $\oM$ there corresponds a natural \wstar-measurement 
$\Gamma^\oM .$
If $E = \LH_\sa$ for a separable Hilbert space
(or more generally $E$ is the self-adjoint part of a $\sigma$-finite \Wstar-algebra),
we can show that for EVMs $(X , \Sigma_1 ,\oM) $ and $(Y , \Sigma_2 ,\oN)$ on $E ,$ 
$\Gamma^\oM \pp \Gamma^\oN$ holds if and only if
there exists a weak Markov kernel $p(\cdot | \cdot)$ such that
$\oM(A) = \int_X p(A | y ) d\oN(y)$
\cite{kuramochi2017minimal,kuramochi2018incomp}.

Conversely, the following proposition indicates that 
any \wstar-measurement can be regarded as the associated \wstar-measurement
of an EVM up to post-processing equivalence.
\begin{prop} \label{prop:EVM}
For any \wstar-measurement $\Gamma \in \Chw (F \to E)$ there 
exists an EVM $(X , \Sigma, \oM)$ on $E$ such that 
$\Gamma \ppeq \Gamma^\oM .$
\end{prop}
Proposition~\ref{prop:EVM} can be shown analogously as in 
\cite{kuramochi2017minimal} (Proposition~3).
In Appendix \ref{app:1} we give another proof using the Riesz-Markov-Kakutani-type 
representation theorem for EVMs.

\subsection{Finite-outcome measurements} \label{subsec:fout}
A special class of EVMs called finite-outcome EVMs 
plays a fundamental role in the later sections of this paper.

Let $F$ be an order unit Banach space.
A family (map) $\oM = (\oM (x))_{x \in X} \in F^X $
is called a subnormalized finite-outcome EVM, or just a subnormalized EVM,
on $F$ if $X$ is a finite set called the outcome set of $\oM,$
$\oM(x) \geq 0$ $(x \in X) ,$
and $\sum_{x \in X} \oM(x) \leq u_F .$
A subnormalized EVM $(\oM (x))_{x \in X}$ is called a normalized finite-outcome EVM,
or just an EVM, 
if $\sum_{x \in X} \oM (x) = u_F .$
We write the sets normalized and subnormalized EVMs on $F$
with the outcome set $X$ 
by
$\evm (X ; E) $
and $\evmsub (X;E) ,$ respectively.
For each EVM $(\oM (x))_{x \in X}$ on $E$ there corresponds the associated 
\wstar-measurement
$\GM \in \Ch (\linf (X) \to E) = \Chw (\linf (X) \to E )$
defined by 
\[
\GM (f) = \sum_{x \in X} f(x) \oM (x) \quad
(f \in \linf (X)) ,
\]
where $\linf (X)$ denotes the classical space of (bounded) real functions on $X$
equipped with the order unit $1_X$ and the positive cone
\[
\linf (X)_+ = \set{f \in \linf (X) | f(x) \geq 0 \, (\forall x \in X)} .
\]
The classical space $\linf (X) = \linf (X)^\aast$ is finite-dimensional 
and conversely any finite-dimensional classical space $F$ is isomorphic 
to $\linf (\mathcal{P}_\mathrm{atom} (F)) ,$
where $\mathcal{P}_\mathrm{atom} (F)$ denotes the set of atomic projections in $F .$
The sets $\evm (X ; E)$ and $\evmsub (X;E)$ are compact convex subsets of 
$E^X$ equipped with the product topology $\sigma (E^X , E^X_\ast)$
of the weak$\ast$ topology $\sigma (E , E_\ast) .$
With respect to this topology on $\evm (X;E) ,$ the map
\[
	\evm (X;E) \ni \oM \mapsto
	\GM \in \Ch (\linf (X) \to E)
\]
is a continuous affine isomorphism, where 
the topology of $\Ch (\linf (X) \to E)$ is the BW-topology.

A \wstar-measurement $\Gamma \in \Ch (\linf (X) \to E)$
for some finite set $X$ is called finite-outcome.

For finite-outcome EVMs, the post-processing relation is characterized as follows.

\begin{prop} \label{prop:fEVM}
\begin{enumerate}[1.]
\item
For any finite-outcome EVM $\oM \in \evm (X;E)$
and a channel $\Lambda \in \Ch (F \to E) ,$
$\GM \pp \Lambda $ if and only if there exists an EVM $\oN \in \evm (X;F)$
such that $\oM (x) = \Lambda (\oN (x))$
$(\forall x \in X) .$
\item
For any finite-outcome EVMs $\oA \in \evm (X;E)$
and $\oB \in \evm (Y;E) ,$ 
$\Ga^\oA \pp \Ga^\oB$
if and only if there exists a stochastic matrix
\[
	p (\cdot | \cdot )
	\in
	\mathrm{Stoch}
	(X,Y)
	:=
	\{ q(\cdot | \cdot ) \in \realn^{X \times Y}
	\, | \,
	q(x|y) \geq 0 ,\, 
	\sum_{x^\prime \in X} q(x^\prime | y) =1
	\, (x \in X , y \in Y)
	\}
\]
such that 
\begin{equation}
	\oA (x) = \sum_{y \in Y} p(x|y)\oB(y) 
	\quad (x \in X) . 
	\label{eq:stoch1}
\end{equation}
\end{enumerate}
\end{prop}
\begin{proof}
The claim~1 is immediate from the isomorphism between $\evm (X;F)$
and $\Ch (\linf (X) \to F) .$
To show the claim~2, assume $\Gamma^\oA \pp \Gamma^\oB$
and take a channel $\Psi \in \Ch (\linf (X) \to \linf (Y))$ such that 
$\Gamma^\oA = \Gamma^\oM \circ \Psi .$
Define a stochastic matrix $p(\cdot | \cdot ) \in \mathrm{Stoch} (X,Y)$
by
\begin{equation}
	 p(x|y) := \Psi (\delta_x) (y)
	\quad (x \in X , y \in Y) ,
	\label{eq:stoch2}
\end{equation}
where $\delta_x \in \linf (X)$ is given by
\[
	\delta_x (x^\prime)
	:=
	\begin{cases}
		1 & \text{if } x=x^\prime;
		\\
		0 & \text{otherwise.}
	\end{cases}
\]
Then we can easily check that $p (\cdot | \cdot)$ satisfies 
\eqref{eq:stoch1}.
Conversely, if \eqref{eq:stoch1} holds for some stochastic matrix
$p (\cdot | \cdot) ,$ 
then the channel $\Psi$ defined by \eqref{eq:stoch2} satisfies 
$\Gamma^\oA = \Gamma^\oB \circ \Psi .$
\end{proof}

An EVM $\oM \in \evm (X;E)$ is called trivial if each element $\oM(x)$
$(\xin )$ is proportional to $u_E .$
The associated \wstar-measurement $\GM$ is then minimal with respect to 
the post-processing order, i.e.\
$\GM \pp \La$ for any measurement (indeed, any channel) $\La .$

\section{Compact convex structure of measurements} \label{sec:ccs}
In this section we define the measurement space and the weak topology on it,
and prove some general properties of them.
Among these results, the most important one is Theorem~\ref{thm:finapp}, which states that any measurement can be approximated by a net of finite-outcome ones and will be used in the later application parts to reduce the discussions to the finite-outcome cases.

The results in this section are generalizations of the known facts in the theory of statistical experiments~\cite{lecam1986asymptotic,torgersen1991comparison}.
See Appendix~\ref{app:se} for how statistical experiments can be regarded as a special class of measurements.

\subsection{Gain functional and the Blackwell-Sherman-Stein (BSS) theorem}
We begin with the notion of gain functional, or state-discrimination probability functional, 
which will play a central role in this paper.
\begin{defi}[Ensemble and gain functional] \label{def:gain}
\begin{enumerate}[1.]
\item
For a finite set $X \neq \varnothing  ,$
a family $\E = (\vph_x)_{x \in X} \in E_\ast^X$ is called a \wstar-family.
The set $X$ is then called the label set of $\E . $
A \wstar-family $\E = (\vph_x)_{x \in X} $ is called an \textit{ensemble} if 
$\vph_x \geq 0$ $(x \in X)$
and the normalization condition $\sum_{x \in X} \braket{\vph_x , u_E} = 1$ holds.
\item
For a \wstar-family $\E = (\vph_x)_{x \in X}$ and a measurement
$\Gamma \in \Ch (F\to E) ,$ we define the gain functional by
\begin{equation}
	\Pg (\E ; \Gamma) :=
	\sup_{ \oM \in \evm (X;F)}
	\sum_{x \in X} \braket{\vph_x , \Gamma (\oM (x)) }.
	\label{eq:Pgdef}
\end{equation}
If $\E$ is an ensemble, the gain functional $\Pg (\E ; \Gamma)$
is occasionally called the state discrimination probability.
\qed
\end{enumerate}
\end{defi}
In the operational language, an ensemble
$\E = (\vph_x)_{x \in X}$ corresponds to 
the situation where system\rq{}s state is prepared to be
$\braket{\vph_x , u_E}^{-1} \vph_x$
with the probability $ \braket{\vph_x , u_E} .$
The value $\Pg (\E ; \Gamma)$ is then the optimal probability 
that we can properly guess the state label $x\in X$ when
we have access to the outcome of the measurement $\Gamma .$
Here each EVM $\oM \in \evm (X;F)$ in \eqref{eq:Pgdef} corresponds to a randomized decision rule of $\xin $ when the measurement outcome of $\Ga$ is given 
(cf. \cite{torgersen1991comparison}, Section~4.5).

If $\Ga$ is a \wstar-measurement in Definition~\ref{def:gain},
for each \wstar-family $\E = \vphxin$
we can always take an EVM $\oM \in \evm (X;F)$ that attains
the optimal value for $\Pg (\E ; \Ga) ,$
i.e.\
\[
	\Pg (\E ; \Ga)
	=
	\sum_{x \in X} \braket{\vph_x , \Gamma (\oM (x)) } .
\]

We remark that we can construct the theory developed in this section based instead on 
the loss functional defined by
\[
	L(\E ;\Gamma) :=
	\inf_{ \oM \in \evm (X;F)}
	\sum_{x \in X} \braket{\vph_x , \Gamma (\oM (x)) }
	=
	- \Pg ((- \vph_x)_\xin ; \Ga  )
\]
(cf.\ \cite{Buscemi2012,kaniowski2013quantum,doi:10.1063/1.5074187}).

Now we prove some elementary properties of the gain functional.
\begin{prop}\label{prop:gaineasy}
Let $\E = \vphxin$ be a \wstar-family and let 
$\Gamma \in \Ch (F \to E)$
and 
$\Lambda \in \Ch (G \to E)$
be measurements.
\begin{enumerate}[1.]
\item \label{i:gaineasy1}
$\Pg (\E ; \lambda \Gamma \oplus (1-\lambda ) \Lambda)
=
\lambda \Pg (\E ;  \Gamma) 
+
(1-\lambda ) \Pg (\E ; \Lambda ) 
$
for any $\lambda \in [0,1] .$
\item \label{i:gaineasy2}
$\Pg (\alpha \E ; \Gamma) = \alpha \Pg (\E ; \Gamma )$
for any $\alpha \in [0,\infty) .$
\item \label{i:gaineasy3}
There exist a positive number $\alpha > 0 ,$ 
a linear functional $\psi \in E_\ast ,$
and an ensemble $\E^\prime = (\vphx^\prime)_{x \in X}$ such that
$\vphx = \alpha \vphx^\prime + \psi .$
Then it also holds that 
$\Pg (\E ; \Gamma ) = \alpha \Pg (\E^\prime ; \Gamma)  + \braket{\psi , u_E} .$
\end{enumerate}
\end{prop}
\begin{proof}
By noting 
\begin{align*}
	&\evm (X ; F \oplus G)
	\\ 
	&=
	\set{
	(\oM (x) \oplus \oN(x))_{x \in X} \in (F \oplus G)^X
	|
	\oM \in \evm (X; F) ,
	\, 
	\oN \in \evm (X ;G)
	}
\end{align*}
we obtain 
\begin{align*}
	&\Pg (\E ; \lambda \Gamma \oplus (1-\lambda ) \Lambda)
	\\
	&=
	\sup_{\oM \in \evm (X; F) ,
	\, 
	\oN \in \evm (X ;G)}
	\sum_{x \in X}
	\braket{\vphx , \lambda \Gamma (\oM (x) )  + (1- \lambda ) \Lambda (\oN (x))}
	\\
	&=
	\lambda \sup_{\oM \in \evm (X; F)} \sum_{x \in X} 
	\braket{\vphx , \Gamma  (\oM (x))} 
	+ 
	(1- \lambda ) 
	\sup_{\oN \in \evm (X ;G)}
	\sum_{x \in X}
	\braket{\vphx , \Lambda (\oN(x))}
	\\
	&=
	\lambda \Pg (\E ;  \Gamma) 
	+
	(1-\lambda ) \Pg (\E ; \Lambda ) ,
\end{align*}
which proves the claim~\ref{i:gaineasy1}.
The claim~\ref{i:gaineasy2} is evident from the definition.

We now show the claim~\ref{i:gaineasy3}.
Since $E_{\ast +}$ is generating, we have a decomposition
$\vphx = \vphx^+ - \vphx^-$
$( \vphx^\pm \in E_{\ast+})$
for each $x \in X.$
By adding a common non-zero functional $\vph \in E_{\ast +}$
to $\vphx^\pm$ if necessary, we may assume $\vphx^\pm \neq 0$
for all $x \in X .$
Define $\psi : = - \sum_{x \in X} \vphx^- .$
Then $\vph_x - \psi$ is positive and non-zero for all $x\in X . $
Therefore we may write $\vph_x = \alpha \vphx^\prime + \psi$
$(x\in X)$
and $\E^\prime = (\vphx^\prime)_{x \in X}$ is an ensemble,
where $\alpha := \sum_{x \in X} \braket{\vph_x - \psi , u_E} > 0$
and $\vphx^\prime := \alpha^{-1} (\vphx - \psi) .$ 
The rest of the claim follows from the definition.
\end{proof}
By Proposition~\ref{prop:gaineasy}.\ref{i:gaineasy3}, any gain functional coincides with a state discrimination probability functional up to a positive factor and a constant functional.

The following BSS theorem states that the family of the gain functionals 
completely characterizes the post-processing order relation for \wstar-measurements.
While we can prove the following theorem using the corresponding result for statistical experiments~\cite{lecam1986asymptotic,torgersen1991comparison} and Proposition~\ref{prop:msmap} in Appendix~\ref{app:se},
here we give a direct proof based on the line of \cite{doi:10.1063/1.5074187}.
The finite division of classical space used in the proof are also of great importance in the later development of the theory.

\begin{thm}[BSS theorem for \wstar-measurements] \label{thm:bss}
Let $\Gamma \in \Chw (F \to E)$ and 
$\Lambda \in \Chw (G \to E)$ be \wstar-measurements.
Then the following conditions are equivalent.
\begin{enumerate}[(i)]
\item \label{i:bss1}
$\Gamma \pp \Lambda .$
\item \label{i:bss2}
$\Pg (\E ; \Gamma) \leq \Pg (\E ; \Lambda)$
for any \wstar-family $\E .$
\item \label{i:bss3}
$\Pg (\E ; \Gamma) \leq \Pg (\E ; \Lambda)$
for any ensemble $\E .$
\item \label{i:bss4}
For each EVM $\oM \in \evm (X;F)$ there exists an EVM
$\oN \in \evm (X;G)$ such that 
$\Gamma (\oM(x)) = \Lambda (\oN(x))$
$(x \in X) .$
\end{enumerate}
\end{thm}
For the proof of Theorem~\ref{thm:bss} 
we introduce here the concept of finite division.
Let $F$ be a classical space with a Banach predual $F_\ast .$
A finite subset $\Delta \subset F$ is said to be a \textit{finite division} of $F$ if 
each element $Q \in \Delta $ is a non-zero projection 
and $\sum_{Q \in \Delta} Q = u_F .$
The set of finite divisions on $F$ is denoted by $\DF .$
For each $\Delta \in \DF$ we write as $F_\Delta : = \lin (\Delta) ,$
which is the finite-dimensional subalgebra of $F$ generated by $\Delta .$
For finite divisions $\Delta , \Delta^\prime \in \DF ,$
$\Delta^\prime$ is said to be finer than $\Delta ,$ written as 
$\Delta \leq \Delta^\prime ,$
if $Q = \sum_{R \in \Delta^\prime \colon R \leq Q} R$
for all $Q \in \Delta .$
Then $\leq$ is a directed partial order on $\DF .$
An element of the subalgebra $\bigcup_{\Delta \in \DF} F_\Delta \subset F$
is said to be a simple element.
Note that the set of simple elements is norm dense in $F $
by Proposition~\ref{prop:proj}. 
If $F$ is the real $L^\infty$-space of a some ($\sigma$-)finite measure,
then the set of simple elements is exactly the set of 
measurable simple functions. 

\begin{lemm}[cf.\ \cite{kuramochi2018directedv1}, Lemma~5]
\label{lemm:supm}
Let $\Gamma \in \Chw (F \to E)$ and 
$\Lambda \in \Chw (G \to E)$ be \wstar-measurements and 
let $\Gamma_\Delta \in \Chw (F_\Delta \to E)$ denote 
the restriction of $\Gamma$ to $F_\Delta .$
Then $\Gamma_\Delta \pp \Lambda$ 
$(\forall \Delta \in \DF)$ implies
$\Gamma \pp \Lambda .$
\end{lemm}
\begin{proof}
By assumption for each $\Delta \in \DF $
there exists $\Psi_\Delta \in \Ch (F_\Delta \to G)$
such that $\Gamma_\Delta = \Lambda \circ \Psi_\Delta .$
Let $F_0 := \bigcup_{\Delta \in \DF} F_\Delta $
and define a map $\wt{\Psi }_\Delta \colon F_0 \to G$
by
\[
	\wt{\Psi}_\Delta (a)
	:=
	\begin{cases}
		\Psi_\Delta (a) & \text{if $a\in F_\Delta;$} \\
		0 & \text{otherwise.}
	\end{cases}
\]
Then since $\| \wt{\Psi}_\Delta(a) \| \leq \| a \|$
$(\Delta \in \DF , a \in F_0) ,$
Tychonoff\rq{}s theorem implies that 
there exist a subnet $(\wt{\Psi}_{\Delta(i)})_{i \in I}$
and a map $\Psi_0 \colon F_0 \to G$ such that
$\wt{\Psi}_{\Delta(i)} (a) \wsto \Psi_0 (a) \in (G)_{\| a \| } $
for each $a \in F_0 .$
Then $\Psi_0$ is a unital bounded linear map that maps a positive element in $F_0$
to a positive one in $G. $
Therefore, since $F_0$ is norm dense in $F ,$
$\Psi_0$ is uniquely extended to a channel $\Psi \in \Ch (F \to G) .$
Then for every $a \in F_0 ,$
we have 
\[
	\Lambda \circ \Psi (a) 
	=
	\lim_{i \in I}
	\Lambda \circ \wt{\Psi}_{\Delta (i)} (a)
	=\Gamma (a),
\]
where we used the weak$\ast$ continuity of $\Lambda$ in the first equality.
By the norm density of $F_0$ in  $F ,$
this implies $ \Gamma  = \Lambda \circ \Psi \pp \La  ,$ which proves the claim.
\end{proof}
\noindent
\textit{Proof of Theorem~\ref{thm:bss}.}
\eqref{i:bss1}$\implies$\eqref{i:bss2}.
Assume \eqref{i:bss1} and take a channel $\Psi \in \Ch (F \to G)$
such that $\Gamma = \Lambda \circ \Psi . $
Let $\E = (\vphx)_{x\in X}$ be a \wstar-family.
Then 
\begin{align*}
	\Pg (\E ; \Gamma )
	&=
	\sup_{\oM \in \evm (X;F)}
	\sum_{x \in X} \braket{\vphx , \Gamma (\oM (x) )}
	\\
	&=
	\sup_{\oM \in \evm (X;F)}
	\sum_{x \in X} \braket{\vphx , \Lambda \circ \Psi (\oM (x) )}
	\\
	& \leq 
	\sup_{\oN \in \evm (X ; G)}
	\sum_{x \in X} \braket{\vphx , \Lambda (\oN (x) )}
	\\
	&=
	\Pg (\E ; \Lambda) ,
\end{align*}
where the inequality follows from 
$\set{(\Psi (\oM(x)))_{x \in X} | \oM \in \evm (X ; F)  }
\subset \evm (X;G) .$

\eqref{i:bss2}$\implies$\eqref{i:bss3} is obvious.

\eqref{i:bss3}$\implies$\eqref{i:bss2} follows from Proposition~\ref{prop:gaineasy}.\ref{i:gaineasy3}.

\eqref{i:bss2}$\implies$\eqref{i:bss4} can be shown similarly as in 
\cite{doi:10.1063/1.5074187} (Proposition~2)
by applying the Hahn-Banach separation theorem to 
$\set{(\Lambda(\oN (x)))_{x \in X} | \oN \in \evm (X; G)} .$

\eqref{i:bss4}$\implies$\eqref{i:bss1}.
Assume \eqref{i:bss4}.
Since $(Q )_{Q\in \Delta} \in \evm (\Delta ;F)$
for any $\Delta \in \DF ,$ the assumption \eqref{i:bss4} implies
that there exists an EVM $\oN_\Delta \in \evm (\Delta ; G)$
such that $\Gamma (Q )  = \Lambda (\oN_\Delta (Q))$
$(Q \in \Delta ) .$
Define a channel $\Phi_\Delta \in \Ch (F_\Delta \to G) $ by
$\Phi_\Delta (Q) := \oN_\Delta (Q)  $
$(Q \in \Delta ) .$
Then, for any $\Delta \in \DF ,$ we have 
$\Gamma_\Delta = \Lambda \circ \Phi_\Delta \pp \Lambda ,$
where $\Gamma_\Delta$ is the restriction of $\Gamma $ to $F_\Delta .$
Therefore Lemma~\ref{lemm:supm} implies $\Gamma \pp \Lambda .$
\qed

For a general pair of measurements which are not necessarily weakly$\ast$ continuous, a theorem corresponding to Theorem~\ref{thm:bss} will be

\begin{thm} \label{thm:prebss}
Let $\Gamma \in \Ch (F \to E)$ and $\Lambda \in \Ch (G \to E)$
be measurements and let 
$\barG \in \Chw (F^\aast \to  E)$
and 
$\barL \in \Chw (G^\aast \to  E)$
be the \wstar-extensions of $\Gamma $ and $\Lambda ,$
respectively.
Then the following conditions are equivalent.
\begin{enumerate}[(i)]
\item \label{i:pbss1}
$\Gamma \pp \barL .$
\item \label{i:pbss2}
$\barG \pp \barL .$
\item \label{i:pbss3}
$\Pg (\E ; \Gamma ) \leq \Pg (\E ; \Lambda )$
for any ensemble $\E .$
\end{enumerate}
\end{thm}
For the proof of Theorem~\ref{thm:prebss} we need the following lemmas.

\begin{lemm} \label{lemm:dense}
Let $E_1$ be an order unit Banach space and let $X \neq \varnothing$
be a finite set.
Then $\evm (X; E_1)$ is dense in $\evm (X ; E_1^\aast)$
with respect to the weak$\ast$ topology $\sigma ((E_1^{\aast })^X , (E_1^{\ast })^X) .$
\end{lemm}
\begin{proof}
The claim is trivial when $\abs {X} =1 ,$
where $\abs{\cdot}$ denotes the cardinality.
If $\abs{X} >1 ,$
fix an element $x_0 \in X$ and define $X^\prime := X \setminus \{ x_0\} .$
Then we have a one-to-one affine correspondence 
\begin{equation}
	\evm (X;E_1) \ni 
	(\oM(x))_\xin
	\mapsto 
	(\oM^\prime (x))_{x \in X^\prime}
	\in \evmsub (X^\prime ; E_1) .
	\notag 
\end{equation}
We can similarly identify $\evm (X;E_1^\aast)$ with 
$\evmsub (X^\prime ; E_1^\aast) .$
Therefore the claim will follow if we can show the weak$\ast$ density of 
$\evmsub (X; E_1)$ in $\evmsub(X;E_1^\aast)$
for any finite set $X .$
Define 
\[
	\cK :=
	\set{
	(\phi_x - \psi)_{x \in X} \in  ( E_1^{\ast})^X | 
	(\phi_x)_{x \in X} \in (E_{1+}^{\ast })^X , 
	\, \psi \in S (E_1) 
	} .
\]
Then $\cK$ is a convex subset of $( E_1^{\ast})^X$ containing $0.$
Moreover, for $\oM \in E_1^X $ we have
\begin{align*}
	& \oM \in \evmsub (X ;E_1)
	\\
	\iff & 
	\sum_{x \in X} \braket{\phi_x , \oM (x)}
	+ 
	\braket{\psi , u_E - \sum_{x \in X} \oM(x)} \geq 0
	\quad (\forall (\phi_x)_{x \in X} \in (E_{1+}^{\ast })^X , \forall \psi \in S(E_1))
	\\
	\iff & 
	\sum_{x \in X} \braket{\phi_x - \psi  , \oM (x)}
	\geq -1
	\quad (\forall (\phi_x)_{x \in X} \in (E_{1+}^{\ast })^X , \forall \psi \in S(E_1)),
\end{align*}
which implies that $\evmsub (X ; E_1)$ is the polar of $\cK$ 
in the pair $(E_1^X , (E_1^{\ast })^X) .$
A similar reasoning yields that $\evmsub (X; E_1^\aast ) $
is the polar of $\cK$ in the pair 
$((E_1^{\aast})^X , (E_1^{\ast })^X) .$
Therefore
if we can show that $\cK$ is $\sigma( (E_1^{\ast })^X , E_1^X)$-closed,
the claim follows from the bipolar theorem.
By the Krein-\v{S}mulian theorem,
it is sufficient to prove the $\sigma( (E_1^{\ast })^X , E_1^X)$-closedness 
of $(\cK)_r$ for any $r \in (0 , \infty) .$
Let $(\phi_x^i - \psi^i)_{x \in X}$
$(i \in I)$ be a net in $(\cK )_r $ 
weakly$\ast$ converging to $(\xi_x)_{x \in X} \in (E_1^{\ast })^X , $
where $\phi_x^i \in E_{1+}^\ast$ and 
$\psi^i \in S(E_1) $
$(x \in X , i \in I) .$
Then $\| \psi^i \| =1$ and hence 
\[
	\| \phi_x^i \|
	\leq \| \psi^i \|  + \| \phi_x^i - \psi^i \| 
	\leq
	1 
	+ \sum_{x^\prime \in X}
	\| \phi_{x^\prime}^i - \psi^i \| 
	\leq 1+r .
\]
Therefore by the Banach-Alaoglu theorem there exist 
a subnet $((\phi_x^{i(j)})_\xin , \psi^{i(j)}) $
$(j \in J) $
and $((\phi_x)_{x \in X} , \psi) \in (E_{1+}^{\ast})^X \times S(E_1)$ such that
$\phi_x^{i(j)} \xrightarrow{\mathrm{weakly}\ast} \phi_x$
and 
$\psi^{i(j)} \xrightarrow{\mathrm{weakly}\ast} \psi .$
Then $(\xi_x)_{x \in X} = (\phi_x - \psi )_{x \in X} \in \cK$
and hence $\cK $ is $\sigma( (E_1^{\ast })^X , E_1^X)$-closed.
\end{proof}

The following lemma states that a measurement is equivalent 
to its \wstar-extension if we concern only 
the state discrimination probabilities 
of ensembles.
\begin{lemm} \label{lemm:exgain}
Let $\Gamma \in \Ch (F \to E)$ be a measurement and let 
$\barG \in \Chw (F^\aast \to E)$ be the \wstar-extension of $\Gamma .$
Then $\Pg (\E ; \Gamma ) = \Pg (\E ; \barG)$ for any \wstar-family
$\E .$
\end{lemm}
\begin{proof}
Let $\E = \vphxin$ be an arbitrary \wstar-family.
From $\Gamma \pp \barG ,$ 
we can show 
$\Pg (\E ; \Gamma ) \leq \Pg (\E ; \barG)$
similarly as in Theorem~\ref{thm:bss}.
Take an arbitrary $\oM^\pprime \in \evm (X ; F^\aast) .$
Then by Lemma~\ref{lemm:dense}
there exists a net $(\oM_i)_{i \in I}$ in $\evm (X; F)$
weakly$\ast$ converging to $\oM^\pprime .$
Then  by the weak$\ast$ continuity of $\barG ,$
\[
	\sum_{x \in X}
	\braket{\vphx , \barG (\oM^\pprime (x))}
	=
	\lim_{i \in I}
	\sum_{x \in X}
	\braket{\vphx , \barG (\oM_i (x))}
	=
	\lim_{i \in I}
	\sum_{x \in X}
	\braket{\vphx , \Gamma (\oM_i (x))} 
	\leq \Pg (\E ; \Gamma ) .
\]
By taking the supremum of $\oM^\pprime , $
we obtain $\Pg (\E ; \barG) \leq \Pg (\E ; \Gamma) .$
\end{proof}
\noindent
\textit{Proof of Theorem~\ref{thm:prebss}.}
The equivalence \eqref{i:pbss1}$\iff$\eqref{i:pbss2} is immediate
from Proposition~\ref{prop:w*extch}.
The equivalence \eqref{i:pbss2}$\iff$\eqref{i:pbss3}
follows from Theorem~\ref{thm:bss} and Lemma~\ref{lemm:exgain}.
\qed

\subsection{Measurement space} \label{subsec:msp}

Let us denote the class of \wstar-measurements on $E$ by $\Meas ,$
which is a proper class since the class of classical spaces with preduals is proper.
Here a proper class is a class that is not a set.
Based on the BSS theorem, we can construct the \textit{set} of post-processing
equivalence classes of \wstar-measurements as follows.

\begin{prop} \label{prop:small}
There exist a set $\ME$ and a class-to-set surjection
\[
	\Meas \ni \Gamma \mapsto [\Gamma ] \in \ME
\]
such that 
$\Gamma \ppeq \Lambda$ if and only if $[\Gamma ] = [\Lambda]$
for any $\Gamma , \Lambda \in \Meas .$
\end{prop}
\begin{proof}
Let $\Ens_k (E) $ be the set of ensembles with the label set
$\natn_k := \{ 1,2, \dots , k \}$
$(k \in \natn)$
and let $\Ens (E) := \bigcup_{k \in \natn} \Ens_k (E) .$
We define a map
\begin{gather}
	\Meas \ni \Gamma \mapsto [\Gamma ] 
	:=
	(\Pg (\E ; \Gamma))_{\E \in \Ens (E)} 
	\in \realn^{\Ens (E)}
	\label{eq:ecl}
\end{gather}	
and a set $\ME \subset \realn^{\Ens (E)}$ by the image of the map
\eqref{eq:ecl}.
Then by Theorem~\ref{thm:bss}, $\ME$ and $[\cdot]$ satisfy the required condition of the statement.
\end{proof}
In what follows in this paper, we fix a set $\ME$ and a map
$[\cdot ]$
satisfying the conditions of Proposition~\ref{prop:small}.
The set $\ME$ is called the \textit{measurement space}
of $E .$
Each element of $\ME$ is also called a measurement,
or an equivalence class of measurements if the distinction is necessary.
For each $\omega \in \ME  ,$ a \wstar-measurement
$\Gamma \in \Meas$ with $[\Gamma] = \omega$
is called a \textit{representative} of $\omega .$

We define the post-processing partial order $\pp$ on $\ME$ by
$[\Gamma] \pp [\La] $ $:\defarrow$ 
$\Ga \pp \La $ $([\Ga] , [\La] \in \ME) .$
For any trivial EVM $\oM_0 ,$ the measurement $[\Ga^{\oM_0}]$ 
is the minimum element of $\ME$ in $\pp .$
We symbolically write as $[u_E]:= [\Ga^{\oM_0}]$
and call it the trivial measurement. 

By Theorem~\ref{thm:bss},
for each \wstar-family $\E $
the gain functional $\Pg (\E ; \cdot)$ on $\ME$ is well-defined by 
\begin{equation}
	\Pg (\E ; [\Ga])
	:=
	\Pg (\E ; \Ga)
	\quad
	([\Ga] \in \ME) .
	\label{eq:Pgfore}
\end{equation}
We define the convex combination (probabilistic mixture) map
\[
	\braket{\cdot ; \cdot , \cdot }
	\colon 
	[0,1] \times \ME \times \ME
	\to \ME
\]
by 
\begin{equation}
	\braket{\lambda ; [\Ga] , [\La ] }
	:=
	[ \lambda \Ga \oplus (1-\lambda ) \La]
	\quad
	(\lambda \in [0,1] ; [\Ga] , [\La ] \in \ME) ,
	\label{eq:convec}
\end{equation}
which is well-defined by Proposition~\ref{prop:direct}.
The gain functional $\Pg (\E ; \cdot )$
for a \wstar-family $\E$ is an affine functional on the 
convex prestructure $(\ME , \braket{\cdot;\cdot , \cdot })$
by Proposition~\ref{prop:gaineasy}.

For each measurement $\Ga ,$ we also denote by $[\Ga]$
the equivalence class of the \wstar-extension of $\Ga .$
Note that for any measurements 
$\Ga $ and $\La ,$
\eqref{eq:Pgfore} is also well-defined by Lemma~\ref{lemm:exgain},
as well as \eqref{eq:convec} is well-defined by Proposition~\ref{prop:direct}.

\subsection{Weak topology on the measurement space} \label{subsec:wtop}
Now we are in a position to define the weak topology on $\ME .$

\begin{defi}[Weak topology] \label{def:wtop}
We define the \textit{weak topology} on $\ME$ as the weakest topology on $\ME$
such that the gain functional $\Pg (\E ; \cdot)$ on $\ME$ is continuous 
for any \wstar-family $\E .$ \qed
\end{defi}
In terms of net, 
the weak topology is characterized as follows:
a net $(\omega_i)_{i \in I}$ in $\ME$ weakly converges to $\omega \in \ME$
if and only if 
$\Pg (\E ; \omega_i) \to \Pg (\E ; \omega) $
for any \wstar-family, or ensemble, $\E .$

The following theorem is a basic result for the weak topology.

\begin{thm} \label{thm:compact}
The weak topology on $\ME$ is a compact Hausdorff topology.
\end{thm}
\begin{proof}
The Hausdorff property of the weak topology follows from that 
the gain functionals separate points of $\ME$
by Theorem~\ref{thm:bss}.

To show the compactness, take an arbitrary net $([\Ga_i])_{i \in I}$ 
in $\ME$ and let $F_i$ be the classical outcome space of the representative
$\Ga_i .$
Let $(\wt{F} , u_{\wt{F}})$ be the direct sum of $(F_i , u_{F_i})_\iin$
and for each $\iin $ define $\tG_i \in \Ch  (\wt{F} \to E)$ by 
\[
	\tG_i ((a_{i^\prime})_{i^\prime \in I})
	:=
	\Ga_i (a_i)
	\quad
	((a_{i^\prime})_{i^\prime \in I} \in \wt{F}) .
\]
By the BW-compactness of $\Ch (\wt{F} \to E)$ (Proposition~\ref{prop:BWcompact}),
there exists a subnet $(\tG_{i(j)})_{j \in J}$
BW-convergent to a channel $\tG_0 \in \Ch (\wt{F} \to E) .$
We show $[\Ga_{i(j)}] \xrightarrow{\mathrm{weakly}} [\tG_0] ,$
from which the compactness of $\ME$ follows.
Then it suffices to prove
$\Pg (\E ;\tG_{i(j)}) \to \Pg (\E ; \tG_0) $ for any 
\wstar-family $\E =\vphxin .$
For each $\iin$ we take an EVM $\oM_i \in \evm (X; F_i)$ such that
\[
	\sum_\xin \braket{\vphx , \Ga_i (\oM_i (x))}
	= \Pg (\E ; \Ga_i) .
\]
We define $\wt{\oM} \in \evm (X ; \wt{F})$ by 
$\wt{\oM} (x) := (\oM_i (x))_\iin $ $(\xin ).$
Take an arbitrary EVM $\wt{\oN} \in \evm (X ;\wt{F})$
with $\wt{\oN} (x) = (\oN_i (x))_{i \in I}$
$(x \in X) .$
Then we have
\begin{align*}
	\sum_\xin \braket{ \vphx , \tG_0 (\wt{\oN} (x))}
	&=
	\lim_{j \in J}
	\sum_\xin
	\braket{\vphx , \Gamma_{i(j)} (\oN_{i(j)} (x))}
	\\
	&\leq
	\liminf_{j \in J}
	\Pg (\E ; \Gamma_{i(j)})
	\\
	&\leq
	\limsup_{j \in J}
	\Pg (\E ; \Gamma_{i(j)})
	\\
	&=
	\limsup_{j \in J} 
	\sum_\xin \braket{\vphx , \Ga_{i(j)}  (\oM_{i(j)}  (x))}
	\\
	&=
	\limsup_{j \in J} 
	\sum_\xin \braket{\vphx , \wt{\Ga}_{i(j)}  (\wt{\oM}  (x))}
	\\
	&=
	\sum_{x \in X}
	\braket{\vphx , \tG_0 (\wt{\oM} (x))}
	\\
	&\leq 
	\Pg (\E ; \tG_0) .
\end{align*}
By taking the supremum of $\wt{\oN } $ in the above (in)equalities,
we obtain 
\[
	\Pg (\E ; \tG_0) 
	=
	\sum_{x \in X}
	\braket{\vphx , \tG_0 (\wt{\oM} (x))}
	=
	\lim_{j \in J} \Pg (\E ; \Ga_{i(j)}) ,
\]
which proves $[\Ga_{i(j)}] \xrightarrow{\mathrm{weakly}} [\tG_0]  .$
\end{proof}

We next show that the post-processing order $\pp$ on $\ME$
is compatible with the weak topology in the following sense.

\begin{defi}[\cite{gierz2003continuous}, Chapter~VI] \label{def:posp}
Let $X$ be a topological space.
A preorder $\leq$ on $X$ is said to be \textit{closed} if the graph
$
	\set{(x,y) \in X \times X | x \leq y}
$
is closed in the product topology on $X \times X .$
If $\leq$ is a closed partial order, then 
the poset $(X , \leq)$ is called a \textit{pospace}. \qed
\end{defi}
In terms of net, the above condition says that the order and the limit commute in the following sense:
for any nets 
$(x_i)_{i \in I} , (y_i)_\iin $ in $X ,$
if
$x_i \leq y_i $ $(\forall \iin) ,$
$x_i \to x \in X ,$
and 
$y_i \to y \in X$
hold, then $x\leq y .$

\begin{prop} \label{prop:posp}
The poset $(\ME , \pp)$ equipped with the weak topology 
is a pospace.
\end{prop}
\begin{proof}
Let $(\omega_i)_\iin$ and $(\nu_i)_\iin$ be nets in $\ME$
satisfying $\omega_i \pp \nu_i$
$(\forall \iin) ,$
$\omega_i \wto \omega \in \ME ,$
and
$\nu_i \wto \nu \in \ME .$
Then for any \wstar-family $\E$ we have
$\Pg (\E ; \omega_i) \leq \Pg (\E ; \nu_i)$
$(\forall \iin) $
by Theorem~\ref{thm:bss}.
By taking the limit we obtain 
$\Pg (\E ; \omega) \leq \Pg (\E ; \nu) $
Since $\E$ is arbitrary, this implies $\omega \pp \nu$ by Theorem~\ref{thm:bss}.
\end{proof}

A net $(x_i)_\iin$ in a poset $(X , \leq)$ is called increasing
if $i \leq j $ implies $x_i \leq x_j $
$(i,j \in I) .$
For a net $(x_i)_\iin$ in a poset $(X , \leq) ,$
its supremum $\sup_\iin x_i$ is
the supremum (the least upper bound)
of the image $\set{x_i | \iin}$ in $X ,$ if it exists.

\begin{lemm}[\cite{gierz2003continuous}, Proposition~VI-1.3] \label{lemm:cposet}
Let $(X , \leq)$ be a compact pospace.
Then any increasing net $(x_i)_\iin$ in $X$ has a supremum
$\sup_\iin x_i \in X$ to which $(x_i)_\iin$ converges topologically.
\end{lemm}

A measurement $\omega \in \ME$ is said to be finite-outcome 
if there exists a finite-outcome EVM $\oM$ such that $\omega = [\GM] ,$
or equivalently if $\omega$ has a representative with a finite-dimensional 
outcome space.
We denote by $\MfinE$ the set of finite-outcome measurements in $\ME .$
The following theorem
states that any measurement can be approximated by an increasing net of finite-outcome measurements
and will be used in Sections~\ref{sec:sim} and \ref{sec:incomp} to reduce the discussions to the finite-outcome cases.

\begin{thm}[Approximation by finite-outcome measurements] \label{thm:finapp}
Let $\Gamma \in \Chw (F \to E)$ be a \wstar-measurement
and let $\Ga_\Delta$ be the restriction of $\Ga$ to $F_\Delta$
for each finite division $\Delta \in \DF .$
Then the net $([\Ga_\Delta])_{\Delta \in \DF}$
in $\MfinE$ is an increasing net and weakly converges to 
$[\Gamma] = \sup_{\Delta \in \DF} [\Ga_\Delta] .$
Furthermore, there exists a net $(\La_\Delta)_{\Delta \in \DF}$
in $\Ch (F \to E)$ such that $\La_\Delta \ppeq \Ga_\Delta$
$(\forall \Delta \in \DF)$
and 
$\| \La_\Delta (a) - \Ga (a) \| \to 0$
$(\forall a \in F) .$
\end{thm}
\begin{proof}
If $\Delta \leq  \Delta^\prime $
$(\Delta , \Delta^\prime \in \DF),$
the \wstar-measurement $\Ga_\Delta$ is the restriction of $\Ga_{\Delta^\prime}$
to the subalgebra $F_\Delta$ of $F_{\Delta^\prime} ,$
and so $\Ga_\Delta \pp \Ga_{\Delta^\prime} .$
Thus the net $([\Ga_\Delta ])_{\Delta \in \DF}$
in $\MfinE$ is increasing.
Therefore by Theorem~\ref{thm:compact}, 
Proposition~\ref{prop:posp}, and Lemma~\ref{lemm:cposet},
there exists a supremum $\sup_{\Delta \in \DF} [\Ga_\Delta] \in \ME$
to which $([\Ga_\Delta])_{\Delta \in \DF}$ weakly converges.
Furthermore Lemma~\ref{lemm:supm} implies  
$[\Ga] = \sup_{\Delta \in \DF} [\Ga_\Delta] ,$
which proves the first part of the claim.

To show the latter part, for each non-zero projection 
$P \in F$ take a
weakly$\ast$ continuous state $\vph_P \in S_\ast  (F)$
such that $\s (\vph_P) \leq P  .$
For each finite division $\Delta \in \DF$ we define a linear map
$\condi_\Delta \colon F \to F_\Delta$ by
\[
	\condi_\Delta (a)
	:=
	\sum_{Q \in \Delta}
	\vph_Q (a) Q .
\]
By noting $\braket{\vph_Q , Q^\prime }= 0$ for $Q, Q^\prime \in \Delta$
with $Q \neq Q^\prime ,$
we can see that $\condi_\Delta$ is a conditional expectation onto $F_\Delta ,$
i.e.\ $\condi_\Delta $ satisfies 
\[
	\condi_\Delta (a) = a \quad (a \in F_\Delta) ,
	\quad
	\| \condi_\Delta (b) \| \leq \| b \| \quad (b \in F),
\]
from which it follows that $\condi_\Delta \in \Ch (F \to F_\Delta) .$
Now let $\La_\Delta := \Ga \circ \condi_\Delta$
$(\Delta \in \DF) ,$ which is a channel in $\Ch (F \to E) .$
Then since $\La_\Delta = \Ga_\Delta \circ \condi_\Delta $
and $\Ga_\Delta$ coincides with the restriction of $\La_\Delta $
to the subalgebra $F_\Delta$ of $F ,$ we have
$\Ga_\Delta \ppeq \La_\Delta .$
Take an element $a \in F .$
Then since $a$ is approximated in norm by a sequence of simple elements
in $F ,$
for every $\epsilon >0$ there exists a simple element $a_\epsilon \in F$ 
such that $\| a - a_\epsilon \| < \epsilon /2 .$
Let $\Delta_\epsilon \in \DF$ be a finite division satisfying
$a_\epsilon \in F_{\Delta_\epsilon} .$
Then for any $\Delta \in \DF$ with $\Delta_\epsilon \leq \Delta ,$
we have $\condi_\Delta (a_\epsilon) = a_\epsilon$
and hence
\[
	\| a - \condi_\Delta (a) \| 
	\leq 
	\| a - a_\epsilon \|
	+ \| \condi_\Delta (a_\epsilon  -a ) \| 
	\leq 
	2\| a - a_\epsilon \|
	< \epsilon .
\]
Therefore for any $\Delta \geq \Delta_\epsilon ,$
\[
	\| \Ga (a) - \La_\Delta (a) \|
	=
	\| \Ga ( a - \condi_\Delta (a) ) \| 
	\leq 
	\| a - \condi_\Delta (a) \| 
	< \epsilon ,
\]
which proves the latter part of the claim.
\end{proof}

We now establish the compatibility of the weak topology and the convex structure 
on $\ME . $

\begin{thm}\label{thm:cconvME}
The convex prestructure $(\ME , \braket{\cdot ; \cdot , \cdot })$
equipped with the weak topology is a compact convex structure.
Furthermore, $\ME$ is regularly embedded into the locally convex Hausdorff
space $\Ac (\ME  )^\ast ,$
which is unique up to continuous linear isomorphism. 
\end{thm}
\begin{proof}
By Proposition~\ref{prop:emb} the latter part of the claim follows from the first one.
In Theorem~\ref{thm:compact} we have established that the  weak topology is a compact Hausdorff topology.
Moreover, since any gain functional $\Pg (\E ; \cdot )$ is a weakly continuous affine functional 
on $\ME ,$
by Theorem~\ref{thm:bss} $\Ac (\ME) $ separates points of $\ME .$
Therefore $(\ME , \braket{\cdot; \cdot , \cdot})$ is a compact convex structure.
\end{proof}
From now on we identify $\ME$ with the weakly$\ast$ compact set
$S(\Ac (\ME))$ on $\Ac (\ME)^\ast .$
In this identification the convex combination 
$\braket{\lambda ; [\Ga] , [\La]}
= [\lambda \Ga \oplus (1-\lambda) \La]
$
becomes the ordinary convex combination
$\lambda [\Ga] + (1-\lambda ) [\La]$
$(\lambda \in [0,1] ; [\Gamma ] , [\La] \in \ME) .$

\subsection{Infinite-dimensionality of $\ME$} \label{subsec:inf}
We now prove that the measurement space is infinite-dimensional except 
in the trivial case $\dim E =1 .$

\begin{thm}[Infinite-dimensionality of $\ME$] \label{thm:infinite}
If $\dim E > 1 ,$ then the measurement space $\ME$
is an infinite-dimensional convex set, i.e.\
for any convex set $K$ in a finite-dimensional Euclidean space there exists no
affine bijection $\Psi \colon \ME \to K .$
\end{thm}
Theorem~\ref{thm:infinite} indicates that the measurement space $\ME$ has 
sufficiently many \wstar-measurements and also that considerations on a proper topology,
as we have done in this section,
is indeed necessary.

We first prove the following lemma.
\begin{lemm} \label{lemm:finpg}
Let $\oM \in \evm (X;E)$ be a finite-outcome EVM and let
$\E = (\vph_y)_{y\in Y}$ be a \wstar-family. Then
\[
	\Pg (\E ; \GM)
	=\sum_{x \in X}
	\max_{y \in Y} \braket{\vph_y , \oM(x)} 
\]
holds.
\end{lemm}
\begin{proof}
The gain functional $\Pg (\E ; \GM)$ is evaluated to be
\begin{align*}
	\Pg (\E ; \GM)
	&=
	\sup_{p(\cdot | \cdot) \in \mathrm{Stoch}(Y,X)}
	\sum_\xin \sum_\yin 
	p(y|x) \braket{\vph_y , \oM(x)}
	\\
	&\leq
	\sup_{p(\cdot | \cdot) \in \mathrm{Stoch}(Y,X)}
	\sum_\xin \sum_\yin p(y|x) 
	\max_{y^\prime \in Y} \braket{\vph_{y^\prime} , \oM(x)}
	\\ &
	=
	\sum_{x \in X}
	\max_{y \in Y} \braket{\vph_y , \oM(x)} .
\end{align*}
The equality of the above inequality is attained by putting 
$p(y|x) = \delta_{y, \wt{y}(x) } ,$
where for each $\xin,$ we take $\wt{y}(x) \in Y$ such that
$\braket{\vph_{\wt{y}(x)} , \oM(x)} = \max_{y \in Y} \braket{\vph_y , \oM(x) } .$
\end{proof}

\noindent
\textit{Proof of Theorem~\ref{thm:infinite}.}
By the assumption $\dim E >1 ,$
there exists an element $a \in E$ such that $0 \leq a \leq u_E$
and $(a , u_E)$ is linearly independent.
Therefore if we put $a^\prime := u_E -a ,$
$(a,a^\prime) $ is also linearly independent.
Thus by the Hahn-Banach separation theorem, there exist
linear functionals $\psi_1 , \psi_2 \in E_\ast$ such that
\[
	\braket{\psi_1 , a} = \braket{\psi_2 , a^\prime}
	=1 
	> 0
	=\braket{\psi_1 , a^\prime} = \braket{\psi_2, a} .
\]
For each $p \in [0,1] $ and each $q \in (0,1) ,$
define a \wstar-family $\E_p$ and an EVM $\oM_q \in \evm (\natn_2 ;E)$
by
\begin{gather*}
	\E_p := (0, (1-p)\psi_2 - p\psi_1) ,
	\\
	\oM_q(1) := (1-q)a ,\quad \oM_q(2) := qa + a^\prime .
\end{gather*}
Then by Lemma~\ref{lemm:finpg} we have
\begin{align}
	\Pg (\E_p ; [\Ga^{\oM_q}] )
	&= \max (0, \braket{ (1-p)\psi_2 - p\psi_1 , \oM_q(1)})
	+ \max(0,  \braket{ (1-p)\psi_2 - p\psi_1 , \oM_q(2)})
	\notag
	\\
	&=\max (0 , -p(1-q)) + 
	\max(0 , -pq+1-p)
	\notag
	\\
	&=
	\max(0, 1-(q+1)p) := f_q (p) .
	\label{eq:pgep}
\end{align}
Now suppose that $\ME$ is finite-dimensional.
Then since the map
\begin{equation}
	\ME \ni \omega \mapsto (\Pg (\E_p ; \omega))_{p\in [0,1]} \in \realn^{[0,1]}
	\label{eq:epimg}
\end{equation}
is affine, the image $A$ of the map \eqref{eq:epimg} contains 
finite number of linearly independent elements in $\realn^{[0,1]} .$
On the other hand, by \eqref{eq:pgep}, $A$ contains the functions
$\{  f_q | q \in (0,1) \}$
and it is easy to see that any finite subset of 
$\{  f_q | q \in (0,1) \}$ is linearly independent,
which is a contradiction.
Therefore $\ME$ is infinite-dimensional.
\qed

\section{Order characterized by a set of continuous affine functionals} \label{sec:mono}
In Section~\ref{sec:ccs} we have seen that the post-processing order
on the measurement space is uniquely characterized by the set of gain functionals.
In this section, as a generalization of the post-processing order,
we consider a preorder on a compact convex structure that is characterized by a set of continuous affine functionals.

Throughout this section, if we call a set $S$ a compact convex structure, 
we understand that $S$ is identified with the state space $S (\Ac (S))$ and 
the convex combination on $S$ is the ordinary one $\la \omega + \ola \nu$
on the linear space $\Ac (S)^\ast .$

The main subject of this section is the order of the following kind.

\begin{defi} \label{def:ordA}
Let $S$ be a compact convex structure
and let $A \subset \Ac (S)$ be a set of continuous affine functionals.
We define a preorder $\preceq_A$ on $S$ by
\begin{equation}
	\omega \preceq_A \nu 
	\, : \defarrow \,
	[f (\omega) \leq f(\nu) 
	\quad
	(\forall f \in A)] 
	\notag
\end{equation}
for any $\omega , \nu \in S . $ \qed 
\end{defi}
For example, when $S = \ME,$ the post-processing order $\pp$ can be written as
$\ordA$ where $A$ is the set of gain functionals.
Later in Theorem~\ref{thm:vnm} we will give an axiomatic characterization of this kind of 
order.

We first consider monotonically increasing affine functionals for this kind of order.
We remind the reader that for a set $X$ equipped with a preorder $\leq, $ 
a function $f \colon X \to \realn$ is \textit{monotonically increasing (in $\leq$)} if 
\[
	x \leq y
	\implies 
	f(x ) \leq f(y)
	\quad
	(x,y \in X).
\]
The following theorem characterizes the monotonically increasing affine functional
in $\ordA .$

\begin{thm} \label{thm:monoGen}
Let $S$ be a compact convex structure,
let $f \colon S \to \realn$ be an affine functional,
let $A \subset \Ac (S)$ be a set of continuous affine functionals,
and let $\mathcal{U}_A$ denote the set of affine functionals on $S$ that can be written as
\begin{equation}
	\alpha 1_S + 
	\sum_{j=1}^n \beta_j g_j 
	\quad (n \in \natn ; \alpha \in \realn;  \beta_1 , \dots , \beta_n \in \realn_+ ;
	g_1 ,\dots , g_n \in A
	) ,
	\notag
\end{equation}
i.e.\ $\mathcal{U}_A := \mathrm{cone}(A \cup \{ \pm 1_{S} \})$
where $\cone (\cdot)$ denotes the conic hull.
Then the following assertions hold.
\begin{enumerate}[1.]
\item \label{i:mg1}
$f$ is monotonically increasing in $\preceq_A$ and continuous 
if and only if $f$ is a uniform limit of a sequence in $\mathcal{U}_A .$
\item \label{i:mg2}
$f$ is monotonically increasing in $\preceq_A$ and bounded 
(i.e.\ $\sup_{\omega \in S}  |f(\omega)| < \infty$)
if and only if $f$ is a pointwise limit of a uniformly bounded net in $\mathcal{U}_A.$
\item \label{i:mg3}
$f$ is monotonically increasing in $\preceq_A$ if and only if
$f$ is a pointwise limit of a net in $\mathcal{U}_A.$
\end{enumerate}
\end{thm}

If we apply Theorem~\ref{thm:monoGen} to the measurement space $\ME ,$
we readily obtain
\begin{coro} \label{coro:monoME}
Let $f \colon \ME \to \realn$ be an affine functional
and let $\mathcal{U}$ denote the set of affine functionals on $\ME$ 
that can be written as
\begin{equation}
	\alpha 1_{\ME} + \sum_{i=1}^n \beta_i \Pg(\E_i ; \cdot)
	\quad
	(n \in \natn ; \alpha \in \realn ;  \beta_1, \dots , \beta_n \in \realn_+ ;
	\E_1 , \dots , \E_n \in \Ens (E)) .
	\label{eq:ccombi}
\end{equation}
Then the following assertions hold.
\begin{enumerate}[1.]
\item
$f$ is monotonically increasing in $\pp$ and weakly continuous
if and only if $f$ is a uniform (i.e.\ norm) limit of a sequence in $\mathcal{U} .$
\item
$f$ is monotonically increasing in $\pp$ and bounded 
if and only if $f$ is a pointwise limit of a uniformly bounded net in $\mathcal{U}.$
\item
$f$ is monotonically increasing in $\pp$ if and only if
$f$ is a pointwise limit of a net in $\mathcal{U}.$
\end{enumerate}
\end{coro}
We remark the affine functional \eqref{eq:ccombi} can be written as 
\[
	\alpha^\prime 1_{\ME}  + \beta^\prime \Pg (\vecE ; \cdot)
\]
for some $\alpha^\prime \in \realn ,$ $\beta^\prime \in \realn_+,$ and a partitioned ensemble $\vecE ,$ 
which will be defined later in Definition~\ref{def:pens}.
This indicates that up to constant factors elements of $\mathcal{U}$ can be regarded as the state discrimination probability with some pre-measurement information.

Now we prove Theorem~\ref{thm:monoGen}.

%




\noindent \textit{Proof of Theorem~\ref{thm:monoGen}.}
\begin{enumerate}[1.]
\item
Let $\mathcal{U}_{\preceq_A}$ denote the set of continuous affine functionals
that are monotonically increasing in $\preceq_A .$
Then to show the claim, we have only to prove $\overline{\mathcal{U}_A} = \mathcal{U}_{\preceq_A},$
where the closure is with respect to the norm topology.
Let 
$\mathcal{U}_A^\ast := \set{\psi \in \Ac(S)^\ast  | \braket{\psi , g} \geq 0 
\, (\forall g \in \mathcal{U}_A)}$ 
be the dual cone of $\mathcal{U}_A .$
We show 
\begin{equation}
	\mathcal{U}_A^\ast
	=
	\set{r (\nu -\omega) | r \in (0,\infty) ; \omega , \nu \in S ; \omega \preceq_A \nu } .
	\label{eq:inclusion}
\end{equation}
The inclusion $(\mathrm{LHS}) \supset (\mathrm{RHS})$ is immediate from 
the definition.
To prove the converse inclusion, take arbitrary $\psi \in \mathcal{U}_A^\ast .$
If $\psi = 0 ,$ then $\psi = \omega - \omega \in (\mathrm{RHS})$
for any $\omega \in S .$
Assume $\psi \neq 0 .$
Since $S$ generates $\Ac (S)^\ast ,$
we can write as $\psi = r_1 \nu - r_2 \omega$ for some $r_1 ,r_2 \in \realn_+$
and some $\omega , \nu \in S .$
Since $\pm 1_S \in \mathcal{U}_A ,$
we have
$
	0 =
	\braket{\psi , 1_S}
	=
	r_1 -r_2 .
$
Hence $\psi = r_1 (\nu - \omega) $
and $r_1 \neq 0$ from $\psi \neq 0 .$
Then from $\psi \in \mathcal{U}_A$ we have
\begin{equation*}
	g(\nu) - g(\omega ) = r_1^{-1} \braket{\psi ,g } \geq 0
	\quad
	(\forall g \in A)  ,
\end{equation*}
which implies $\omega \preceq_A \nu .$
Therefore $\psi$ is in the RHS of \eqref{eq:inclusion} and
we have proved \eqref{eq:inclusion}.
Now let $\mathcal{U}_A^\aast$ be the double dual cone of $\mathcal{U}_A  $
in the pair $(\Ac (S)  , \Ac (S)^\ast) .$
Then from \eqref{eq:inclusion}
\[
	\mathcal{U}_A^\aast
	=
	\set{f \in \Ac (S) | f(\omega) \leq f (\nu) \, 
	\text{for any $\omega, \nu \in S$ with $\omega \preceq_A \nu$}
	 } 
	 = \mathcal{U}_{\preceq_A} .
\]
On the other hand, by the bipolar theorem
$\mathcal{U}_A^\aast $ is the weak closure 
(i.e.\ $\sigma (\Ac (S) , \Ac (S)^\ast)$-closure)
of $\mathcal{U}_A .$
Since the weak and the norm closures coincide for a convex set on a 
Banach space (e.g.\ \cite{schaefer1999topological}, Section~9.2),
we have $\mathcal{U}_A^\aast =\ovl{\mathcal{U}_A} .$
Thus we obtain $\overline{\mathcal{U}_A} =\mathcal{U}_{\preceq_A} .$
\item
We first note that the Banach dual space $\Ac (S)^\aast$ can be 
identified with $\Ab (S) $ and the weak$\ast$ topology
$\sigma(\Ac(S)^\aast , \Ac (S)^\ast )$ on $\Ac(S)^\aast$ 
is, by this identification, the pointwise convergence topology on $\Ab (S) .$
Let us define 
\[
	\mathcal{K}
	:= \mathcal{U}_A^\ast + (\Ac(S)^\ast )_1
	= \set{\psi + \phi \in \Ac(S)^\ast | \psi \in \mathcal{U}_A^\ast ;
	 \phi \in  (\Ac(S)^\ast )_1 } ,
\]
which is a convex set containing $0.$
Since $\mathcal{U}_A^\ast$ is weakly$\ast$ closed and 
$(\Ac(S)^\ast )_1$ is weakly$\ast$ compact, $\mathcal{K}$ is weakly$\ast$ closed.
Then for any $g \in \Ab (S)$ we have
\begin{align*}
	& \braket{g , r (\nu - \omega) + \phi} \geq -1
	\quad
	(r \in \realn_+ ; \omega , \nu \in S ; \omega \preceq_A \nu ; 
	\phi \in (\Ac(S)^\ast)_1)
	\\
	\iff &
	\braket{g , \nu-\omega} \geq 0  
	\quad (\omega , \nu \in S ; \omega \preceq_A \nu )
	\text{ and }
	\braket{g,\phi} \geq -1 \quad (\phi \in (\Ac (S)^\ast)_1)
	\\
	\iff &
	\text{$g$ is monotonically increasing in $\ordA$ and $\| g \| \leq 1 ,$}
\end{align*}
which implies that the polar $\mathcal{K}^\circ$ of $\mathcal{K}$ in the pair 
$(\Ac (S)^\ast , \Ab(S))$ is the set of bounded affine functionals in the unit ball 
$(\Ab(S))_1$ 
that are monotonically increasing in $\preceq_A .$
Similarly, the polar $\mathcal{K}_\circ$ of $\mathcal{K}$ in the pair
$(\Ac (S)^\ast , \Ac(S))$ is the set of continuous affine functionals in the unit ball 
$(\Ac(S))_1$ 
that are monotonically increasing in $\preceq_A ,$ 
i.e.\ $\mathcal{K}_\circ = (\mathcal{U}_{\preceq_A})_1 .$
Since $\mathcal{K}$ is a weakly$\ast$ closed convex set containing $0,$ 
the bipolar theorem implies 
that $\mathcal{K}$ is the polar of $\mathcal{K}_\circ$ in the pair 
$(\Ac (S)^\ast , \Ac(S)) .$
Thus again by the bipolar theorem, $\mathcal{K}^\circ$ is the closure of 
$\mathcal{K}_\circ$ in the pointwise convergence topology on $\Ab (S) .$

Now assume that $f$ is monotonically increasing and bounded.
Then $f \in \| f \| \mathcal{K}^\circ$ and hence by the above result there exists a net $(f_i)_\iin$ in 
$ (\mathcal{U}_{\preceq_A})_{\| f \|} $ converging pointwise to $f.$
Thus, from the claim~\ref{i:mg1}, for each $\iin $ there exists a sequence
$(f_{i,n})_{n \in \natn}$ in $\mathcal{U}_A$ uniformly converging to $f_i .$ 
We may assume that $\| f_{i,n} \| \leq \| f \| $ for all $\iin$ and all $n \in \natn .$
Then $(f_{i,n})_{\iin , n \in \natn}$ is a uniformly bounded net in $\mathcal{U}_A$
and converges to $f$ for each point in $S ,$
which proves the \lq\lq{}only if\rq\rq{} part of the claim.
The \lq\lq{}if\rq\rq{} part of the claim is obvious.

\item
Let $\Aalg (S)$ denote the set of affine functionals on $S .$
In a similar manner as the case of $\Ab (S)$ and $\Ac (S)^\aast ,$
we can prove that $\Aalg (S)$ can be identified with the algebraic dual
$(\Ac(S)^\ast)^\prime$ of $\Ac(S)^\ast$
(i.e.\ $(\Ac(S)^\ast)^\prime$ is the set of real linear functionals on $\Ac (S)^\ast$).
Moreover the bilinear form 
$\braket{\cdot , \cdot}$ on $\Aalg (S) \times \Ac(S)^\ast$ such that
\[
	\braket{g , \omega} = g(\omega )
	\quad
	(g \in \Aalg (S) ; \omega \in S)
\]
is well-defined and separating.
By this correspondence, the topology $\sigma (\Aalg (S) , \Ac (S)^\ast)$ 
is the pointwise convergence topology
on $\Aalg (S) .$
Then as in the proof of the claim~\ref{i:mg1}, we can show that
the set of affine functionals in $\Aalg (S)$ that are monotonically increasing
in $\preceq_A$ is the closure of $\mathcal{U}_A$ in 
the pointwise convergence topology on $\Aalg (S) ,$
from which the claim immediately follows. \qed
\end{enumerate}

We next give conditions of a set $A$ of continuous affine functionals 
when the order $\preceq_A$ is a partial or a total order.
\begin{prop} \label{prop:orcon}
Let $S$ be a compact convex structure
and let $A \subset \Ac (S)$ be a set of continuous affine functionals.
Then the following assertions hold.
\begin{enumerate}[1.]
\item
$\ordA$ is a partial order, i.e.\ $\omega \preceq_A \nu$ and $\nu \ordA \omega $
imply $\omega = \nu$ for any $\omega , \nu \in S ,$
if and only if the linear span $\lin (A \cup \{ 1_S \})$ is norm dense in $\Ac (S) .$
\item
$\ordA$ is a total order, i.e.\ either $\omega \ordA \nu$ or $\nu \ordA \omega$
holds for any $\omega , \nu \in S ,$
if and only if there exists an element $f_0 \in \Ac (S)$ such that 
$A \subset \cone (\{ f_0 , \pm 1_S \}) .$
\end{enumerate}
\end{prop}
\begin{proof}
\begin{enumerate}[1.]
\item
Assume that $\lin (A \cup \{ 1_S \})$ is norm dense in $\Ac (S) .$
Take arbitrary $\omega , \nu \in S$ such that 
$\omega \ordA \nu $ and $\nu \ordA \omega .$
Then from the definition of $\ordA $ we have
\begin{align*}
	& \braket{\nu - \omega , f} = 0 \quad (\forall f \in A)
	\\
	\implies &
	\braket{\nu -\omega , f} =0  \quad (\forall f \in \lin (A \cup \{ 1_S  \})) .
\end{align*}
Since $\lin (A \cup \{ 1_S \})$ is norm dense in $\Ac (S) ,$ 
this implies $\nu - \omega =0 .$
Thus $\ordA$ is a partial order.

Conversely, assume that $\lin (A \cup \{ 1_S \})$ is not norm dense in $\Ac (S) .$
Then by the Hahn-Banach theorem, there exists a non-zero linear functional
$\psi \in \Ac (S)^\ast$ such that 
\begin{equation}
	\braket{\psi , f} = \braket{\psi , 1_S} = 0
	\quad
	(\forall f \in A) .
	\label{eq:vanishA}
\end{equation}
Since $S$ generates $\Ac (S)^\ast ,$ by noting $\braket{\psi , 1_S} =0$ 
we may write as $\psi = r (\nu - \omega )$ for some $r \in (0,\infty) $
and some $\omega , \nu \in S .$
Then from \eqref{eq:vanishA} we have 
\[
	f(\omega) = f (\nu) 
	\quad (\forall f \in A)
\]
and hence $\omega \ordA \nu$ and $\nu \ordA \omega $ hold.
Since $\omega \neq \nu$ by $\psi \neq 0 ,$
this proves that $\ordA$ is not a partial order.

\item
Assume that $A \subset \cone (\{ f_0 , \pm 1_S \})$ for some $f_0 \in \Ac (S) .$
Then we have 
\[
	\omega \preceq_{\{ f_0\}} \nu \implies 
	\omega \preceq_A \nu 
	\quad
	(\omega , \nu \in S) .
\]
Moreover we can easily see that $\preceq_{\{ f_0\}}$ is total from the definition.
Thus from the above implication it follows that $\preceq_A$ is also total,
which proves the \lq\lq{}only if\rq\rq{} part of the claim.

Conversely assume that $A$ is not included in $\cone (\{ f_0 , \pm 1_S \})$
for any $f_0 \in \Ac (S) .$
Then we can take $f_1 \in A $ which is not a constant functional.
Moreover since $A$ is not included in $\cone (\{ f_1 , \pm 1_S \}),$
we can take an element $f_2 \in A \setminus \cone (\{ f_1 , \pm 1_S \}) ,$
which implies
\[
	f_1 \notin \cone (\{ f_2 , \pm 1_S \})
	\quad \text{and} \quad
	f_2 \notin \cone (\{ f_1 , \pm 1_S \}) .
\]
Since $ \cone (\{ f_2 , \pm 1_S \})$ is a finitely generated cone and hence
is closed in the norm topology 
(\cite{bonnans2000perturbation}, Proposition~2.41),
the Hahn-Banach separation theorem implies that there exists a linear functional
$\psi_1 \in \Ac (S)^\ast$ such that 
\begin{equation}
	\braket{\psi_1 , f_1} > \braket{\psi_1 , 1_S} = 0 \geq \braket{\psi_1 , f_2 } .
	\label{eq:psi1}
\end{equation}
Then from $\braket{\psi_1 , 1_S} = 0$ and $\psi_1 \neq 0$ we may write as
$\psi_1 = r_1 (\nu_1 - \omega_1 )$ for some
$r_1 \in (0,\infty)$ and some $\omega_1 , \nu_1 \in S .$
Then from \eqref{eq:psi1} we obtain
\begin{equation}
	f_1(\omega_1) < f_1 (\nu_1),
	\quad
	f_2 (\omega_1) \geq f_2 (\nu_1) .
	\label{eq:on1}
\end{equation}
By interchanging $f_1$ and $f_2$ in the above discussion, 
we can take elements $\omega_2 , \nu_2 \in S$ such that
\begin{equation}
	f_1(\omega_2) \geq f_1 (\nu_2),
	\quad
	f_2 (\omega_2) < f_2 (\nu_2) .
	\label{eq:on2}
\end{equation}
Now if $f_2 (\omega_1) > f_2 (\nu_1) ,$ 
\eqref{eq:on1} implies that $\omega_1$ and $\nu_1$ are incomparable in $\ordA .$
Similarly if $f_1(\omega_2) > f_1 (\nu_2) ,$
\eqref{eq:on2} implies that 
$\omega_2$ and $\nu_2$ are incomparable in $\ordA .$
Now assume $f_2 (\omega_1) = f_2 (\nu_1)$ and $f_1(\omega_2) = f_1 (\nu_2) .$
Then from \eqref{eq:on1} and \eqref{eq:on2} we have
\[
	f_1\left( \frac{\omega_1 + \nu_2}{2} \right) 
	< 
	f_1 \left( \frac{\omega_2 + \nu_1}{2} \right) ,
	\quad
	f_2\left( \frac{\omega_1 + \nu_2}{2} \right) 
	>
	f_2 \left( \frac{\omega_2 + \nu_1}{2} \right) ,
\]
which implies the incomparability of $\frac{\omega_1 + \nu_2}{2} $
and $\frac{\omega_2 + \nu_1}{2} $ in $\ordA .$
Therefore $\ordA$ is not total.
\qedhere
\end{enumerate}
\end{proof}

By using this result we prove the non-totality of the measurement space 
$\ME$ except when $\dim E =1 .$
\begin{coro} \label{coro:nontotal}
Assume $\dim E > 1 .$
Then the post-processing order $\pp$ on $\ME$ is not total.
\end{coro}
\begin{proof}
Assume that $\pp$ is a total order
and let $A$ denote the set of gain functionals on $\ME .$
Since $\pp$ coincides with $\ordA,$ 
Proposition~\ref{prop:orcon} implies that 
$A$ is included in $\cone (\{ f_0 , \pm 1_{\ME}\} )$
for some $f_0 \in \Ac (\ME)$
and hence $\lin (A \cup \{ 1_{\ME} \})$ is finite-dimensional.
Furthermore, since $\pp$ is a partial order on $\ME , $ 
again by Proposition~\ref{prop:orcon} 
the linear span $\lin (A \cup \{ 1_{\ME} \})  ,$
which is norm closed by the finite dimensionality,
coincides with $\Ac (\ME).$
Hence $\Ac (\ME) $ and  $\Ac (\ME)^\ast$ are
finite-dimensional,
which contradicts the infinite-dimensionality of $\ME (\subset \Ac(\ME)^\ast)$
established in Theorem~\ref{thm:infinite}.
\end{proof}

Throughout this section, we have considered the class of orders of the form $\ordA ,$
which contains the post-processing order as a special case,
and proved general statements under this general setup.
As finishing this section we give an axiomatization of this kind of order
analogous to that of the preference relation characterized by 
the utility~\cite{vnm1953,DUBRA2004118}
or of the adiabatic accessibility relation characterized by the thermodynamic 
entropy~\cite{giles1964mathematical,LIEB19991}.
\begin{thm}[von Neumann-Morgenstern utility theorem without the totality 
(completeness) axiom]
\label{thm:vnm}
Let $S$ be a compact convex structure and 
let $\preceq$ be a preorder on $S .$ Then the following conditions are equivalent.
\begin{enumerate}[(i)]
\item \label{i:ordA}
There exists a subset $A \subset \Ac (S)$ such that
$\preceq$ conincides with $\ordA .$
\item \label{i:vnm}
The order $\preceq$ satisfies both of the following conditions.
\begin{enumerate}[(a)]
\item \label{i:ind} (Independence axiom).
For any $\omega , \nu , \mu \in S$ and any $\la \in (0,1) ,$
\[
	\omega \preceq \nu \implies
	\la \omega + \ola \mu \preceq \la \nu + \ola \mu .
\]
\item \label{i:conti} (Continuity axiom).
The preorder $\preceq$ is closed in the sense of Definition~\ref{def:posp}.
\end{enumerate}
\end{enumerate}
Moreover, for any non-empty subsets $A , B \subset \Ac (S) ,$
the orders $\preceq_A$ and $\preceq_B$ coincide
if and only if 
$\ovl{\cone} (A \cup \{ \pm 1_S\}) = \ovl{\cone} (B \cup \{ \pm 1_S\}) ,$
where $\ovl{\cone}(\cdot )$ denotes the closed conic hull with respect to
the norm topology.
\end{thm}
Theorem~\ref{thm:vnm} was proved in \cite{DUBRA2004118}
when $S$ is a Bauer simplex $S (C(X))$ for a 
compact metric space $X .$
The proof in \cite{DUBRA2004118} can be straightforwardly generalized 
to the more general case of Theorem~\ref{thm:vnm} with slight modifications.
See Appendix~\ref{app:vnm} for the proof.

\section{Simulability and robustness of unsimulability} \label{sec:sim}
In this section, we consider simulability
\cite{doi:10.1063/1.4994303,PhysRevLett.119.190501,PhysRevA.97.062102} 
of a  measurement 
relative to a given set of measurements.
The main result in this section is Theorem~\ref{thm:RoU} that characterizes the operational meaning of the robustness measure of unsimulability.

\subsection{Simulability: definition and basic properties}

\begin{defi}[Simulability]
\label{def:sim}
Let $\onon \fL  \subset \ME $ be a set of measurements. 
A measurement $\omega \in \ME$
is said to be \textit{simulable} (respectively, \textit{strongly simulable}) by
$\fL$ if there exists 
$\nu \in \cconv (\fL)$
(respectively, $\nu \in \conv (\fL)$)
such that
$\omega \pp \nu ,$
where $\cconv (\cdot )$ denotes the closed convex hull
in the weak topology.
We also say that a measurement $\Gamma$ is 
(strongly) simulable by $\fL$
if the equivalence class $[\Ga]$ is (strongly)
simulable by $\fL .$
The sets of measurements in $\ME$
(strongly) simulable by $\fL$ is written as 
$\simu (\fL)$ ($\ssimu (\fL)$).
If $\fL$ is finite, $\simL$ and $\ssimuL$ coincide since $\conv (\fL) = \cconv (\fL).$
\qed
\end{defi}
The operational meaning of the strong simulability is as follows.
Suppose that an experimenter is able to perform only restricted
measurements belonging to $\fL \subset \ME .$
Then a measurement strongly simulable by $\fL$ is also realized by the experimenter
by classical pre- and post-processing a finite measurements belonging to $\fL .$
Here each element of $\conv (\fL)$ corresponds to take a classical pre-processing.

The following order theoretic terminology 
is useful in representing the set of simulable measurements. 
\begin{defi}\label{def:lower}
Let $(X , \leq)$ be a poset and let $Y \subset X .$
We define the lower closure of $Y$ by
\[
	\downarrow Y 
	:=
	\set{x \in X| \exists y \in Y ,\, x \leq y}.
\] 
If $Y ={ \downarrow Y}  ,$ $Y$ is said to be a lower set.
The lower closure $\downarrow Y$ is the smallest lower set containing 
$Y.$
\qed
\end{defi}
By using this notation,
the sets $\ssimu (\fL)$ and $\simu (\fL)$ 
in Definition~\ref{def:sim}
can be written as
$\ssimu (\fL) = {\downarrow \conv (\fL)}$
and 
$\simu (\fL) = {\downarrow \cconv (\fL)} ,$
respectively.

As for the operational meaning of the simulability, 
the following proposition suggests that 
the simulable measurements are exactly the measurements 
that are arbitrary approximated by strongly simulable ones in the weak topology.

\begin{prop}\label{prop:wsim}
Let $\onon \fL \subset \ME .$
Then $\simu (\fL) = \ovl{ \ssimuL  } .$
\end{prop}
For the proof we need some lemmas.

\begin{lemm} \label{lemm:clower}
Let $(X , \leq)$ be a compact pospace and let 
$Y $ be a compact subset of $X .$
Then the lower closure $\loc{Y}$ is also compact. 
\end{lemm}
\begin{proof}
Let $(x_i)_\iin$ be a net in $\loc{Y} .$
Then for each $\iin$ we take $y_i \in Y$ satisfying $x_i \leq y_i .$
By the compactness of $X$ and $Y ,$
there exist subnets $(x_{i(j)})_{j\in J}$ and $(y_{i(j)})_{j\in J}$
satisfying 
$x_{i(j)} \to x \in X$ and $y_{i(j)} \to y \in Y.$
Then by the pospace condition of $X$ we have $x \leq y,$
which implies $x \in \loc{Y} .$
Therefore $\loc{Y}$ is compact
\end{proof}

\begin{lemm}\label{lemm:wsimc}
Let $\onon \fL \subset \ME .$
Then $\simL$ is a compact convex subset of $\ME .$
\end{lemm}
\begin{proof}
Since $\cconv (\fL)$ is weakly compact, the compactness of 
$\simL = \loc{ \cconv(\fL)  }$ follows from Lemma~\ref{lemm:clower}.
To prove the convexity, take measurements $\omega_1 , \omega_2 \in \simL$
and $\nu_1 , \nu_2 \in \cconv (\fL)$
satisfying $\omega_j \pp \nu_j $
$(j=1,2) .$
Then by Proposition~\ref{prop:direct}, for each $\la \in [0,1]$ we have
\[
	\la \omega_1 + \ola \omega_2 
	\pp
	\la \nu_1 + \ola \nu_2
	\in \cconv (\fL) , 
\]
which implies $\la \omega_1 + \ola \omega_2  \in \simL .$
Therefore $\simL$ is convex.
\end{proof}

For a subset $\onon \fL \subset \ME$
and a finite set $X$ we define
\[
	\evmsimL (X;E) := \set{\oM \in \evm (X;E) | [\GM] \in \simL},
\]
which is the set of EVMs simulable by $\fL$ with the outcome set $X.$

\begin{lemm} \label{lemm:gainconti}
Let $X$ be a finite set.
\begin{enumerate}[1.]
\item \label{i:gainconti1}
The map 
\begin{equation}
	\evm (X;E) \ni \oM
	\mapsto [\GM] \in \ME
	\label{eq:EVMM}
\end{equation}
is continuous with respect to the weak$\ast$ topology on $\evm (X;E)$
and the weak topology on $\ME .$
\item \label{i:gainconti2}
For any subset $\onon \fL \subset \ME ,$ 
$\evmsimL (X;E)$ is a weakly$\ast$ compact convex subset of 
$\evm (X;E) .$
\end{enumerate}
\end{lemm}
\begin{proof}
\begin{enumerate}[1.]
\item
If a net $(\oM_i)_\iin$ in $\evm (X;E)$ weakly$\ast$ converges to 
$\oM \in \evm (X;E) ,$
then by Lemma~\ref{lemm:finpg}
for any \wstar-family $\E = (\vph_y)_\yin$ we have
\[
	\Pg(\E ; \Ga^{\oM_i})
	=
	\sum_\xin \max_\yin \braket{\vph_y , \oM_i (x)}
	\to 
	\sum_\xin \max_\yin \braket{\vph_y , \oM (x)}
	= 
	\Pg (\E ; \GM),
\]
which proves the continuity of \eqref{eq:EVMM}.
\item
Since $\simL$ is weakly closed by Lemma~\ref{lemm:wsimc},
the compactness of $\evmsimL (X;E)$ follows from the claim~\ref{i:gainconti1}.
To show the convexity, take EVMs 
$\oM_1 , \oM_2 \in \evmsimL (X;E) $
and $\la \in [0,1].$
Then by Proposition~\ref{prop:direct}.\ref{i:direct2} and
Lemma~\ref{lemm:wsimc} we have
\[
	[\Ga^{\la \oM_1 + \ola \oM_2}]
	=
	[\la \Ga^{\oM_1} +\ola \Ga^{\oM_2}]
	\pp
	\la [\Ga^{\oM_1} ]+\ola [ \Ga^{\oM_2}]
	\in \simL .
\]
Since $\simL$ is a lower set, this implies 
$\la \oM_1 + \ola \oM_2 \in \evmsimL (X;E) .$
\qedhere
\end{enumerate}
\end{proof}
\noindent 
\textit{Proof of Proposition~\ref{prop:wsim}.}
The inclusion $\ssimuL \subset \simL $ is obvious.
Since $\simL$ is weakly compact by Lemma~\ref{lemm:wsimc},
this implies
$\overline{\ssimuL} \subset \simL .$
To show the converse inclusion, we take $\omega \in \simL$ and prove
$\omega \in  \overline{\ssimuL}.$

We first assume that $\omega $ is finite-outcome.
Then $\omega = [\GM]$ for some finite-outcome EVM
$\oM \in \evm (X;E) .$
By the definition of $\simL ,$
there exist a measurement $\nu \in \cconv (\fL)$
and a net $(\nu_i )_\iin$ in $\conv (\fL)$ such that
$\omega \pp \nu$
and 
$\nu_i \wto \nu .$
Let $\Ga_i \in \Chw (F_i \to E)$
$(i\in I)$ be a representative of $\nu_i$
and let $\wt{F} = \bigoplus_\iin F_i ,$
$\tG_i \in \Chw (\wt{F} \to E) ,$
$(\tG_\ikj)_\jin ,$
and
$\tG_0 \in \Ch (\wt{F} \to E)$
be the same as in the proof of Theorem~\ref{thm:compact}.
Then from the proof of Theorem~\ref{thm:compact} we have
\[
	[\tG_0] 
	=
	\lim_\jin [\Ga_\ikj] 
	=
	\lim_\jin \nu_\ikj
	= \nu .
\]
Let $\tG \in \Chw (\wt{F}^\aast \to E)$ be the \wstar-extension of $\tG_0 .$
Since $[\GM] \pp \nu = [\tG_0] = [\tG],$
there exists an EVM $\oN^\pprime \in \evm (X; \wt{F}^\aast)$ such that
$\oM(x) = \tG (\oN^\pprime(x))$ 
$(\xin) .$
By Lemma~\ref{lemm:dense} there exists a net $(\oN_k)_{k \in K}$
in $\evm (X ; \wt{F})$
such that $\oN_k \wsto \oN^\pprime .$
Let $\oN_k (x) = (\oN_{k,i}(x))_\iin$
($k \in K ,$ $ \xin$) and define
$\oM_{k,j}, \oM_k  \in \evm (X;E)$ by
\begin{gather}
	\oM_{k,j} (x) := 
	\tG_\ikj (\oN_k (x) )
	=
	\Ga_\ikj (\oN_{k, i(j)} (x)),
	\label{eq:Mkj} \\
	\oM_k (x)
	:=
	\tG (\oN_k (x))
	=
	\tG_0 (\oN_k (x))
	=
	\lim_\jin 
	\oM_{k,j} (x) .
	\notag
\end{gather}
Then by the weak$\ast$ continuity of $\tG ,$
we have
$\oM_k \wsto \oM .$
Since $[\Ga^{\oM_{k,j}} ] \in \ssimuL$ by \eqref{eq:Mkj},
Lemma~\ref{lemm:gainconti}.\ref{i:gainconti2} implies $[\Ga^{\oM_k}] \in \overline{\ssimuL}$
and hence again by Lemma~\ref{lemm:gainconti}.\ref{i:gainconti2} we have 
$ \omega = [\Ga^{\oM}] = \lim_{k \in K} [\Ga^{\oM_k}] \in \overline{\ssimuL} .$

For general $\omega \in \simL ,$ 
Theorem~\ref{thm:finapp} implies that there exists a post-processing
increasing net $(\omega_\alpha)_{\alpha \in A}$ in $\MfinE$
weakly converging to $\omega = \sup_{\alpha \in A} \omega_\alpha .$
Since $\simL$ is a lower set, we have $\omega_\alpha \in \simL$
$(\alpha \in A)$
and hence $\omega_\alpha \in \ovl{\ssimuL}$
from what we have shown above.
Therefore $\omega = \lim_{\alpha \in A} \omega_\alpha$
is also in $\ovl{\ssimuL} ,$
which completes the proof.
\qed

One might expect from Lemma~\ref{lemm:gainconti}.\ref{i:gainconti1}
that for an infinite-dimensional classical space $F$ the map
\begin{equation}
	\Ch(F \to E) 
	\ni \Ga \to [\Ga] \in \ME
	\label{eq:uho}
\end{equation}
is also continuous with respect to the BW-topology on $\Ch(F \to E) $
and the weak topology on $\ME .$
This is in fact not true.
We have still a result analogous to Lemma~\ref{lemm:gainconti}.\ref{i:gainconti2}.
Let us show a slightly more general result.

For a subset $\fL \subset \ME$ and a classical space $F$ we define
\[
	\ChL (F \to E)
	:=
	\set{\Gamma \in \Ch (F\to E) | [\Ga] \in \fL},
\]
which is the inverse image of $\fL$ under the map \eqref{eq:uho}.

\begin{prop} \label{prop:cclower}
Let $\fL \subset \ME$ be a lower subset with respect to the post-processing order.
\begin{enumerate}[1.]
\item
$\fL$ is weakly compact if and only if $\ChL (F\to E)$ 
is BW-compact for any classical space $F.$
\item
$\fL$ is convex if and only if 
$\ChL (F \to E)$ is convex for any classical space $F .$
\end{enumerate}
\end{prop}
\begin{proof}
\begin{enumerate}[1.]
\item
Suppose that $\fL$ is weakly compact.
Let $F$ be a classical space and take a net $(\Ga_i)_\iin$
in $\ChL (F\to E).$
By the compactness of $\fL$ and $\Ch (F \to E) ,$
there exist a subnet $([\Ga_\ikj])_\jin ,$
a measurement $\omega \in \fL ,$
and a measurement $\Ga \in \Ch (F \to E)$
such that 
$[\Ga_\ikj] \wto \omega$
and 
$\Ga_\ikj \bwto \Ga .$
Then for any ensemble $\E = \vphxin$ we have 
$\Pg (\E ; \Ga_\ikj) \to \Pg (\E ; \omega) .$
Thus for any $\oM \in \evm (X;F),$ 
\begin{align*}
	\sum_\xin \braket{\vphx , \Ga (\oM(x))}
	&=
	\lim_\jin \sum_\xin \braket{\vphx , \Ga_\ikj (\oM(x))}
	\\
	&\leq
	\lim_\jin \Pg (\E; \Ga_\ikj)
	\\
	&=
	\Pg (\E ; \omega),
\end{align*}
which implies $\Pg (\E ; \Ga) \leq \Pg (\E ; \omega ) .$
Therefore by Theorem~\ref{thm:prebss} we obtain $[\Ga] \pp \omega .$
Since $\fL$ is a lower set, this implies $[\Ga] \in \fL$
and hence $\Ga \in \ChL (F \to E) ,$ which proves the compactness of
$\ChL (F\to E) .$

Conversely suppose that $\ChL (F\to E)$ is BW-compact for any classical space 
$F .$
Let $([\Ga_i])_\iin$ be a net in $\fL$ with the representatives
$\Ga_i \in \Chw (F_i \to E)$
$(\iin ) .$
We take $\wt{F} = \bigoplus_\iin F_i ,$
$\tG_i \in \Ch (\wt{F} \to E) ,$
$(\tG_\ikj )_\jin ,$
and
$\tG_0 \in \Ch  (\wt{F} \to E) $
in the same way as in Theorem~\ref{thm:compact}.
Then $[\Ga_\ikj] \wto [\tG_0] .$ 
We can also easily see that $\Ga_i \ppeq \tG_i .$
Thus by assumption the BW-limit $\tG_0$ of $(\tG_\ikj)_\jin$
is in $\ChL (\wt{F} \to E) ,$
which implies $[\tG_0] \in \fL .$
Therefore $\fL $ is compact.
\item
Suppose that $\fL$ is convex and take a classical space $F,$
measurements $\Ga_1 , \Ga_2 \in \ChL (F \to E) ,$
and $\la \in [0,1] .$
Then by Proposition~\ref{prop:direct} and the convexity of $\fL,$
\[
	[\la \Ga_1 + \ola \Ga_2]
	\pp
	[\la \Ga_1 \oplus \ola \Ga_2 ]
	= 
	\la [\Ga_1] + \ola [\Ga_2]
	\in \fL .
\]
Since $\fL$ is a lower set, this implies $\la \Ga_1 + \ola \Ga_2 \in \ChL (F\to E),$
which proves the convexity of $\ChL (F\to E) .$

Conversely assume that $\ChL (F\to E)$ is convex for any classical space 
$F .$
Let $\omega_1 , \omega_2 \in \fL$
and let $\La_j \in \Chw (F_j \to E) $ 
be a representative of $\omega_j$ $(j=1,2) .$ 
Define \wstar-measurements $\tL_j \in \Chw (F_1 \oplus F_2 \to E)$
$(j=1,2)$ by
\[
	\tL_1 (a\oplus b) := \La_1 (a) , \quad
	\tL_2 (a\oplus b) := \La_2 (b) 
	\quad
	(a \in F_1 , b \in F_2) .
\]
Then it is easy to show $\La_j \ppeq \tL_j$ $(j=1,2).$
Hence $\tL_j \in \ChL (F_1 \oplus F_2 \to E)$
and the assumption implies 
\[
	\la \La_1 \oplus \ola \La_2
	=
	\la \tL_1 +\ola \tL_2
	\in 
	\ChL (F_1 \oplus F_2 \to E).
\]
Thus $\la \omega_1 + \ola \omega_2 = [\la \La_1 \oplus \ola \La_2] \in \fL ,$
which proves the convexity of $\fL .$
\qedhere
\end{enumerate}
\end{proof}

\begin{coro} \label{coro:chlcomp}
Let $\fL \subset \ME .$
Then for any classical space $F,$
$\ChwsimL (F\to E)$ is a BW-compact convex subset of 
$\Ch (F \to E) .$
\end{coro}

\subsection{Simulability and outperformance in the state discrimination task} 
\label{subsec:sd}
We introduce the gain functional relative to a set of measurements
based on the following proposition.
\begin{prop} \label{prop:Pgdef}
Let $\E = \vphxin$ be a \wstar-family and let $\onon \fL \subset \ME .$
Then the following equalities hold:
\begin{align}
	\sup_{\omega \in \simL}
	\Pg (\E ; \omega )
	&=
	\sup_{\omega \in \cconv (\fL)} \Pg (\E ;\omega)
	\notag \\ &=
	\sup_{\omega \in \fL}
	\Pg (\E ;\omega)
	\notag \\ &=
	\sup_{\oM \in \evmsimL (X;E)} 
	\sum_\xin 
	\braket{\vphx, \oM (x)} .
	\label{eq:PgL}
\end{align}
\end{prop}
\begin{proof}
The first two equalities follow from the monotonicity in $\pp ,$ the affinity, and the weak continuity of $\Pg (\E ; \cdot) .$
By the compactness of $\evmsimL (X;E) $ (Lemma~\ref{lemm:gainconti}), the maximal value of the RHS of \eqref{eq:PgL} is attained by som $\oM_0 \in \evmsimL (X;E).$
Then from $[\Ga^{\oM_0}] \in \simL$ we have
\[
	(\text{RHS of \eqref{eq:PgL}})
	= \sum_{\xin} \braket{\vphx, \oM_0 (x)}
	\leq
	\Pg (\E ; [\Ga^{\oM_0}])
	\leq
	\sup_{\omega \in \simL}
	\Pg (\E ; \omega ) .
\]
On the other hand, by the compactness of $\simL ,$ we can take $\omega_0 \in \simL$ such that
$\sup_{\omega \in \simL}
	\Pg (\E ; \omega ) = \Pg (\E ; \omega_0) .$
Let $\La_0 \in \Chw (F \to E)$ be a representative of $\omega_0 .$
Then we can take $\oN_0 \in \evm (X;F)$ such that 
\[
	\Pg (\E ; \omega_0) = \sum_\xin \braket{\vphx , \La_0 (\oN_0 (x))} .
\]
Since $(\La_0 (\oN_0 (x)))_\xin \in \evmsimL (X;E) ,$ this implies
\[
	\sup_{\omega \in \simL}
	\Pg (\E ; \omega ) 
	\leq  (\text{RHS of \eqref{eq:PgL}}) ,
\]
which completes the proof.
\end{proof}
We write the quantity \eqref{eq:PgL} as $\Pg(\E ; \fL) .$
If $\E$ is an ensemble, $\Pg (\E ; \fL)$ is the optimal probability 
that we correctly guess the original state 
when we can perform measurements in $\simL$ (or $\fL$).

Now, as a generalization of the finite-dimensional result
\cite{PhysRevLett.122.130403} (Eq.~(14)),
we prove that the outperformance in the state discrimination task characterizes 
the simulability.

\begin{thm} \label{thm:simdisc}
Let $\onon \fL \subset \ME$
and let $\omega \in \ME  $ be a measurement.
Then $\omega$ is simulable by $\fL$ if and only if
\begin{equation}
	\Pg (\E ; \omega)
	\leq 
	\Pg (\E ; \fL)
	\label{eq:simdisc}
\end{equation}
holds for any ensemble $\E .$
\end{thm}
\begin{proof}
Assume that $\omega $ is simulable by $\fL .$
Then by the definition of simulability and Proposition~\ref{prop:Pgdef},
we can readily see that \eqref{eq:simdisc} holds.
To show the converse implication, 
we assume $\omega \notin \simL $ and find an ensemble $\E$
that does not satisfy \eqref{eq:simdisc}.

We first consider the case when $\omega$ is finite-outcome.
Take an EVM $\oM \in \evm (X;E)$ such that $\omega = [\GM] .$
Then $\oM \notin \evmsimL (X ; E) .$
Since $\evmsimL (X ; E)$ is a weakly$\ast$ compact convex set
by Lemma~\ref{lemm:gainconti},
the Hahn-Banach separation theorem implies that
there exists a non-zero \wstar-family 
$\E = \vphxin \in E_\ast^X$ such that
\begin{equation}
	\sum_\xin \braket{\vphx , \oM(x)}
	> 
	\sup_{\oN \in \evmsimL (X;E)}
	\sum_\xin \braket{\vphx , \oN (x)} .
	\label{eq:sd1}
\end{equation}
By Proposition~\ref{prop:gaineasy}.\ref{i:gaineasy3} we can take $\E$ as an ensemble.
Then \eqref{eq:sd1} implies 
\[
	\Pg(\E ; \omega) \geq  \sum_\xin \braket{\vphx , \oM(x)} 
	>
	\Pg (\E ; \fL) .
\]
Therefore $\E$ violates \eqref{eq:simdisc}.

For general $\omega ,$ by Theorem~\ref{thm:finapp} there exists an increasing 
net $(\omega_i)_\iin$ of finite-outcome measurements weakly converging to
$\omega = \sup_\iin \omega_i .$
Since $\ME \setminus \simL$ is weakly open by Lemma~\ref{lemm:wsimc},
there exists some $i \in I$ satisfying $\omega_i \notin \simL .$
Then from what we have shown in the last paragraph, 
there exists an ensemble $\E$ satisfying 
$\Pg (\E ; \omega_i) > \Pg (\E; \fL) .$
Therefore by the monotonicity of $\Pg(\E; \cdot) $
and $\omega_i \pp \omega,$
we obtain $\Pg (\E ; \omega) > \Pg (\E;\fL) ,$
which completes the proof.
\end{proof}

From Theorem~\ref{thm:simdisc} and Lemma~\ref{lemm:exgain}
we immediately obtain
\begin{coro}\label{coro:simdisc}
Let $\onon \fL \subset \ME$
and let $\Gamma$ be a measurement.
Then $\Gamma$ is simulable by $\fL$ if and only if
\[	\Pg (\E ;\Ga)
	\leq 
	\Pg (\E ; \fL)
\]
holds for any ensemble $\E .$
\end{coro}

\subsection{Maximal success probability of simulation}
Now we introduce the first robustness measure of simulability, 
the success probability of simulation.

We introduce the standard the standard trivial measurement 
\begin{equation}
	\Gtriv \colon \cmplx \ni \alpha \mapsto \alpha u_E \in E ,
\notag
\end{equation}
whose equivalence class $[\Gtriv]$ coincides with $[u_E].$

\begin{defi}[maximal success probability of simulation] \label{def:qsucc}
Let $\Ga 
$ be a measurement and let $\onon \fL \subset \ME .$
The \textit{maximal success probability of simulation} of $\Ga$ by $\fL$ is defined by
\begin{equation}
\begin{aligned}
\qsucc (\Ga ; \fL) := \sup_{q} \quad &  q
\\
\textrm{subject to} \quad 
& q \in [0,1] \\ 
& \text{$q \Ga \oplus (1-q) \Gtriv$ is simulable by $\fL$.}
\end{aligned}
\label{eq:qsuccdef}
\end{equation}
Note that $q=0$ is in the feasible region of \eqref{eq:qsuccdef} and hence $\qsucc (\Ga ; \fL)$ always takes on a finite value in $[0,1].$
It can be readily seen that $\qsucc ([\Ga^\prime] ; \fL) := \qsucc (\Ga^\prime ;\fL) $ is well-defined for any equivalence class $[\Ga^\prime]\in \ME .$
The operational meaning of \eqref{eq:qsuccdef} is the maximal success probability of simulation of $\Ga$ when we can perform the measurements in $\fL ,$ where the event corresponding to the term $(1-q)\Gtriv$ is the failure event of the simulation.
(See \cite{PhysRevLett.119.190501,PhysRevA.100.012351} for this kind of probabilistic simulation of measurements by projection-valued measurements.)
The maximal success probability $\qsucc (\Ga ; \fL)$ quantifies the degree of simulability of $\Ga $ by $\fL$;
if $\qsucc (\Ga ; \fL) =1,$ $\Ga$ is simulable by $\fL$ and if $\qsucc (\Ga ; \fL) =0,$ $\Ga$ is not simulable with any finite success probability. 
\qed
\end{defi}

Now we show that the maximal probability of simulation is related to the state discrimination probabilities as follows:
\begin{thm} \label{thm:qsucc}
In the setting of Definition~\ref{def:qsucc}, the equality 
\begin{equation}
	\qsucc (\Ga ; \fL)
	=
	\left( \inf_\E
	\frac{\Pg (\E ; \fL) - \Pg (\E ; [u_E])}{\Pg (\E ; \Ga) - \Pg (\E ; [u_E])}
	\right) \wedge 1 
	\label{eq:thmqsucc}
\end{equation}
holds,
where the infimum of $\E$ is taken over the ensembles such that $\Pg (\E ; \Ga) > \Pg (\E ; [u_E]) ,$ $\inf \varnothing := \infty ,$ and $a \wedge b := \min (a,b) .$
\end{thm}
\begin{proof}
Let us denote by $\Ens (E)$ the set of ensembles in Proposition~\ref{prop:small}.
From Corollary~\ref{coro:simdisc}, $q \in [0,1]$ is in the feasible region of \eqref{eq:qsuccdef} if and only if 
\begin{align*}
	&
	\Pg (\E ; q \Ga \oplus (1-q) \Gtriv)
	\leq \Pg (\E ;\fL)
	\quad (\forall \E \in \Ens (E))
	\\ \iff &
	q \Pg (\E ; \Ga) + (1-q) \Pg (\E ; [u_E]) 
	\leq \Pg (\E ;\fL)
	\quad (\forall \E \in \Ens (E))
	\\
	\iff & 
	q \leq \frac{\Pg (\E ; \fL) - \Pg (\E ; [u_E])}{\Pg (\E ; \Ga) - \Pg (\E ; [u_E])}
	\quad (\forall \E \in \Ens (E) \text{ with }\Pg (\E ; \Ga) > \Pg (\E ; [u_E])).
\end{align*}
(Note that $\Pg (\E ; \Ga) \geq \Pg (\E ; [u_E])$ holds for any measurement $\Ga$ and ensemble $\E.$)
From this equivalence, the claim \eqref{eq:thmqsucc} immediately follows
\end{proof}

\subsection{Robustness of unsimulability} \label{subsec:RoU}
Now we introduce the second robustness measure,
the robustness of unsimulability 
relative to a set of measurements.

\begin{defi}\label{def:RoU}
Let $\onon \fL \subset \ME $ and let $\Ga \in \Ch (F \to E) $
be a measurement.
We define the \textit{robustness of unsimulability of
$\Ga$ relative to $\fL$}
by 
\begin{equation}
\begin{aligned}
\Runs (\Ga ; \fL) := \inf_{r , \La} \quad &  r
\\
\textrm{subject to} \quad 
& r \in [0,\infty) \\ 
&  \Lambda \in \Ch (F \to E)   \\
  &\frac{\Ga + r \La}{1+r} \in \ChwsimL  (F \to E) ,\\
\end{aligned}
\label{eq:RoUdef}
\end{equation}
where $\Runs (\Ga ; \fL) : = \infty$ when the feasible region of \eqref{eq:RoUdef}
is empty. 
The optimization problem \eqref{eq:RoUdef} can be written as 
\begin{equation}
\begin{aligned}
\Runs (\Ga ; \fL) = \inf_{r , \Psi} \quad &  r
\\
\textrm{subject to} \quad 
& r \in [0,\infty) \\ 
&  \Psi \in \Ch^{\simL} (F \to E)   \\
  &\Ga \leq (1+ r)\Psi ,\\
\end{aligned}
\notag
\end{equation}
where the order $\leq$ on the set of linear operators between the ordered linear spaces   $G , H$ is defined by
\[
	\Phi \leq \Xi \, : \defarrow \, [\Phi (a) \leq \Xi (a) \quad (\forall a \in G_+)]
\]
for linear maps $\Phi  , \Xi \colon G \to H .$
\qed
\end{defi}
The meaning of $\Runs (\Ga ; \fL)$ is the minimal amount of noise
that should be added to make the measurement $\Ga$ simulable by $\fL .$
In the resource theoretic perspective,
measurements in $\simL$ are considered to be free and 
the ability to perform an unsimulable measurement is considered to be 
resourceful.
In this viewpoint $\Runs (\Ga ; \fL)$ quantifies how resourceful 
$\Ga$ is relative to the free measurements in $\fL $ or $\simL .$

Motivated by recent results on robustness measures, we prove the following theorem,
the main result of this section.

\begin{thm} \label{thm:RoU}
Let $\onon \fL \subset \ME $ and let $\Ga \in \Ch (F \to E)$
be a measurement.
Then the equality
\begin{equation}
	1+ \Runs (\Ga ; \fL)
	= 
	\sup_{\E \colon \mathrm{ensemble}}  \frac{\Pg (\E ; \Ga)}{\Pg (\E ; \fL)}
	\label{eq:RoU}
\end{equation}
holds, where the supremum is taken over all the ensembles.
\end{thm}
We remark that if we put $\fL =\{ [u_E]\} ,$
the singleton consisting of the trivial measurement,
then the robustness measure $\Runs (\Ga ; \fL)$
is the one called the \lq\lq{}robustness of measurement\rq\rq{} in 
\cite{PhysRevLett.122.140403}
and Theorem~\ref{thm:RoU} in this case is the 
infinite-dimensional version of Eq.~(13) in \cite{PhysRevLett.122.140403}.

For the first step of the proof, we show some elementary properties of 
$\Runs (\cdot ; \cdot) .$

\begin{lemm} \label{lemm:RoUe1}
Let $\onon \fL \subset \ME  .$ 
Then for each measurement $\Ga \in \Ch (F \to E)$
with $r : = \Runs (\Ga ; \fL) < \infty ,$
there exists a measurement $\Psi \in \Ch^{\simL} (F \to E)$ such that
$\Ga \leq (1+r) \Psi .$
\end{lemm}
\begin{proof}
By the definition of $\Runs (\Ga ; \fL) ,$
there exists a sequence
$(r_n , \Psi_n)$ $(n \in \natn)$ in $[r, \infty ) \times \Ch^{\simL}(F \to E)$
such that
\[
	\Ga \leq (1+r_n) \Psi_n ,
	\quad
	r_n \downarrow r .
\]
Then by the BW-compactness of $\Ch^{\simL} (F\to E)$ there exists a subnet
$(\Psi_{n(i)})_\iin$ BW-convergent to a simulable measurement
$\Psi \in \Ch (F\to E) .$
By the weak$\ast$ closedness of $E_+$ this implies
$\Ga \leq (1+r) \Psi .$
\end{proof}

\begin{lemm} \label{lemm:RoUe2}
Let $\onon \fL \subset \ME $
and let $\Ga_j \in \Ch (F_j \to E)$
$(j=1,2)$ be measurements. Then 
$\Ga_1 \pp \Ga_2$ implies 
$\Runs (\Ga_1 ; \fL) \leq \Runs (\Ga_2 ; \fL) ,$
i.e.\ the robustness of unsimulability is monotonically increasing in the post-processing order.
\end{lemm}
\begin{proof}
We may assume $r_2 := \Runs (\Ga_2 ; \fL) < \infty .$
Then by Lemma~\ref{lemm:RoUe1} there exists $\Psi_2
\in \Ch^{\simL} (F_2 \to E)$ satisfying 
$\Ga_2 \leq (1+r_2)\Psi_2 .$
By assumption there exists $\Psi \in \Ch (F_1 \to F_2)$ such that
$\Ga_1 = \Ga_2 \circ \Psi .$
Then we have
$\Ga_1 \leq (1+r_2) \Psi_2 \circ \Psi .$
Since $\Psi_2 \circ \Psi$ is simulable by $\fL , $ this implies $\Runs(\Ga_1 ; \fL) \leq r_2 = \Runs (\Ga_2; \fL) .$
\end{proof}

\begin{lemm} \label{lemm:RoUe3}
Let $\onon \fL \subset \ME ,$ 
let $\Ga \in \Ch (F \to E)$ be a measurement, 
and let $\barG \in \Chw (F^\aast \to E)$ be the \wstar-extension of $\Ga .$
Then $\Runs (\Ga ; \fL) = \Runs (\barG ; \fL) .$
\end{lemm}
\begin{proof}
From $\Ga \pp \barG ,$ we have $\Runs (\Ga ; \fL) \leq \Runs (\barG ; \fL)$
by Lemma~\ref{lemm:RoUe2}.
Thus without loss of generality we may assume $r:= \Runs (\Ga ; \fL) < \infty .$
Then by Lemma~\ref{lemm:RoUe1} there exists a measurement
$\Psi \in \Ch^{\simL} (F\to E)$ such that 
$\Ga \leq (1+r) \Psi .$
Let $\barPsi \in \Chw (F^\aast \to E)$ be the \wstar-extension of $\Psi .$ 
Since $[\barPsi] = [\Psi] ,$ $\barPsi$ is simulable by $\fL .$
Moreover, by the weak$\ast$ density of $F_+$ in $F_+^\aast ,$ we have 
$\barG \leq (1+r) \barPsi .$
This implies $\Runs (\barG ; \fL) \leq r ,$
which proves the claim.
\end{proof}
We now prove
\begin{lemm}\label{lemm:RoUineq}
In the setting of Theorem~\ref{thm:RoU}, the inequality
\begin{equation}
	1+ \Runs (\Ga ; \fL)
	\geq
	\sup_{\E \colon \mathrm{ensemble}}  \frac{\Pg (\E ; \Ga)}{\Pg (\E ; \fL)}
	\label{eq:RoUineq}
\end{equation}
holds.
\end{lemm}
\begin{proof}
We may assume $r:= \Runs (\Ga ; \fL) < \infty .$
Then by Lemma~\ref{lemm:RoUe1} there exists 
$\Psi \in \Ch^{\simL} (F \to E)$ such that $\Ga \leq (1+r)\Psi .$
Then for any ensemble $\E = \vphxin$ and $\oM \in \evm (X;F)$
we have
\begin{align*}
	\sum_\xin \braket{\vphx , \Ga (\oM(x))}
	&\leq
	(1+r) \sum_\xin \braket{\vphx , \Psi (\oM(x))} \\
	&\leq 
	(1+r) \Pg (\E ; \Psi) \\
	&\leq (1+r ) \Pg (\E ; \fL ) .
\end{align*}
By taking the supremum of $\oM ,$ we obtain 
$\Pg (\E; \Ga) \leq (1+r) \Pg (\E ; \fL) ,$ which implies \eqref{eq:RoUineq}.
\end{proof}
To establish the converse inequality, we first consider the case when 
$\Gamma$ is finite-outcome.

\begin{lemm} \label{lemm:RoUfin}
The statement of Theorem~\ref{thm:RoU} is true when 
$\Ga = \GM$ for some finite-outcome EVM
$\oM \in \evm (X;E) .$
\end{lemm}
\begin{proof}
By the identification between $\Ch (\linf (X) \to E)$ and $\evm (X;E)$
(cf.\ Section~\ref{subsec:fout}),
the robustness of unsimulability can be written as
\begin{equation*}
\begin{aligned}
\Runs (\GM ; \fL) 
= \inf_{r , \oK} \quad &  r
\\
\textrm{subject to} \quad 
& r \in [0,\infty) \\ 
& \oK \in \evmsimL (X;E)  \\
& (1+r)\oK (x) \geq \oM (x) \, (x \in X)  .
\end{aligned}
\notag
\end{equation*}
Thus if we define the convex cone 
\[
\cK_\fL (X;E)
:=
\set{\la \oM \in E^X | \la \in [0,\infty) , \, \oM\in \evmsimL (X;E)}
\]
generated by $\evmsimL (X;E)$ we have
\begin{equation*}
\begin{aligned}
1+ \Runs (\GM ; \fL) 
= \inf_{s , \oK} \quad &  s
\\
\textrm{subject to} \quad 
& s \in \realn, \quad \oK \in \cK_\fL (X;E)  \\
& \sum_\xin \oK(x)  = s u_E \\
& \oK (x) \geq \oM (x) \, (x \in X) ,
\end{aligned}
\notag
\end{equation*}
which is equal to
\begin{equation}
\begin{aligned}
1+ \Runs (\GM ; \fL) 
= \inf_{s , \oK} \quad &  s
\\
\textrm{subject to} \quad 
& s \in \realn, \quad \oK \in \cK_\fL (X;E)  \\
& \sum_\xin \oK(x)  \leq s u_E \\
& \oK (x) \geq \oM (x) \, (x \in X) .
\end{aligned}
\label{eq:RoUpr}
\end{equation}
The optimization problem \eqref{eq:RoUpr} can be written in the standard
form of the conic programming~\cite{Shapiro2001,bonnans2000perturbation}
\begin{equation}
	\inf_v \braket{c^\ast , v}
	\quad
	\textrm{subject to} \quad 
	v\in C , \, A(v) + b \in K .
	\label{eq:coneprgrm}
\end{equation}
where $C$ and $K$ are respectively 
closed convex cones on Banach spaces $V$ and $U ,$ 
$A\colon V \to U$ is a bounded linear map,
$c^\ast \in V^\ast ,$ and $b \in U .$ 
Indeed \eqref{eq:RoUpr} coincides with \eqref{eq:coneprgrm}
if we put
\begin{gather*}
	V:= E^X \times \realn ,
	\quad
	C:= \cK_\fL (X;E) \times \realn
	 \\
	 U := E^X \times E ,
	 \quad
	 K:= E^X_+ \times E_+ ,
	 \\
	 \braket{c^\ast , (w,s)} := s \quad ((w,s) \in V) ,
	 \quad
	 b := ((- \oM(x))_\xin , 0) ,\\
	 A((w,s)) := \left(w, s u_E - \sum_\xin w(x)  \right)
	 \quad
	 ((w,s)\in V) ,
\end{gather*}
provided that $\mathcal{K}_\fL (X;E)$ is closed.
We prove a stronger fact that $\mathcal{K}_\fL (X;E)$ is weakly$\ast$ closed.
Take $r \in (0,\infty)$ and a net $(\oK_i)_\iin$ in $(\mathcal{K}_\fL (X;E))_r$ weakly$\ast$ converging to some $\oK \in E^X .$
By the definition of $\mathcal{K}_\fL (X;E)$ we may write as $\oK_i = \la_i \oN_i$
for some $\la_i \in [0,\infty )$ and $\oN_i \in \evmsimL (X;E).$
Then from $\| \oK_i (x) \| \leq r$ we have
\[
	\la_i 
	= \| \la_i u_E \| 
	= \left\| \sum_{\xin} \oK_i (x) \right\|
	\leq 
	\sum_\xin \| \oK_i (x) \| \leq \abs{X} r .
\]
Hence we can take subnets $(\la_{i(j)})_\jin$ and $(\oN_{\ikj})_\jin$ converging respectively to some $\la \in [0,\abs{X}r]$ and  $\oN \in \evmsimL (X;E).$
Then we have $\oK = \la \oN \in \mathcal{K}_\fL(X;E) .$
Thus by the Krein-\v{S}mulian theorem $\mathcal{K}_\fL (X;E)$ is weakly$\ast$ closed.

Now let $v_0 := ((u_E)_\xin , \abs{X} +1 ) \in V .$ 
Since the trivial EVM $(\abs{X}^{-1} u_E)_\xin$ is in $\evmsimL (X;E) ,$
we have $v_0 \in C $ and hence 
\begin{align*}
	U & \supset
	A(C) - K
	\\
	& \supset \realn_+ A(v_0)  -K 
	\\
	&=
	\set{(\la u_E - w (x) )_\xin , \la u_E -v) | 
	\la \in \realn_+ , \, w \in E_+^X , \, v \in E_+}
	\\
	&=
	U.
\end{align*}
This implies $-b \in  \interi (A(C)-K) (=U) ,$
where $\interi (\cdot )$ denotes the interior.
Therefore the optimization problem~\eqref{eq:RoUpr}
has no duality gap 
(\cite{Shapiro2001}, Proposition~2.9;
\cite{bonnans2000perturbation}, Theorem~2.187)
and hence \eqref{eq:coneprgrm} coincides with
\begin{equation}
	\sup_{u^\ast } \braket{u^\ast , b}
	\quad
	\textrm{subject to}
	\quad
	u^\ast \in -K^\ast , \, 
	A^\ast (u^\ast) +c^\ast \in C^\ast ,
	\label{eq:RoUd1}
\end{equation}
where 
\begin{gather*}
	K^\ast := \set{u^\ast \in U^\ast | \braket{u^\ast , u} \geq 0 \,(\forall u \in K)}
	\\
	C^\ast := \set{v^\ast \in V^\ast | \braket{v^\ast , v} \geq 0 \,(\forall v \in C)}
\end{gather*}
are dual cones.
Since we have
\begin{gather*}
	K^\ast = (E^\ast_+)^X \times E^\ast_+ ,
	\quad
	C^\ast = \cK_\fL (X;E)^\ast \times \{ 0\} ,
	\\
	\cK_\fL (X;E)^\ast :=
	\set{(\psi_x)_\xin \in (E^\ast)^X | 
	\sum_\xin \braket{\psi_x , \oK (x)} \geq 0
	\,
	(\forall \oK \in \cK_\fL (X;E))
	},
	\\
	A^\ast ((\psi_x)_\xin , \chi)
	=
	((\psi_x - \chi)_\xin , \braket{\chi , u_E})
	\quad
	( ((\psi_x)_\xin , \chi) \in U^\ast = (E^\ast)^X \times E^\ast) ,
\end{gather*}
the dual problem \eqref{eq:RoUd1} is explicitly written as
\[
\begin{aligned}
\sup_{\chi ,(\psi_x)_\xin } \quad & 
\sum_\xin \braket{\psi_x , \oM(x)}
\\
\textrm{subject to} \quad 
&\chi \in E^\ast_+ ,\quad  (\psi_x)_\xin \in (E^\ast_+)^X \\ 
& (\chi - \psi_x)_\xin \in \cK_\fL (X;E)^\ast ,
\quad
\braket{\chi , u_E} =1
\end{aligned}
\]
and hence
\begin{equation}
\begin{aligned}
1 + \Runs (\GM ; \fL) =
\sup_{\chi ,(\psi_x)_\xin } \quad & 
\sum_\xin \braket{\psi_x , \oM(x)}
\\
\textrm{subject to} \quad 
&\chi \in E^\ast_+ ,\quad  (\psi_x)_\xin \in (E^\ast_+)^X \\ 
& (\chi - \psi_x)_\xin \in \cK_\fL (X;E)^\ast ,
\quad
\braket{\chi , u_E} \leq 1 .
\end{aligned}
\label{eq:RoUd2}
\end{equation}

We now show
\begin{equation}
\begin{aligned}
1 + \Runs (\GM ; \fL) =
\sup_{\chi ,(\psi_x)_\xin } \quad & 
\sum_\xin \braket{\psi_x , \oM(x)}
\\
\textrm{subject to} \quad 
&\chi \in E_{\ast +}  ,\quad  (\psi_x)_\xin \in (E_{\ast +})^X \\ 
& (\chi - \psi_x)_\xin \in \cK_\fL (X;E)^\ast ,
\quad
\braket{\chi , u_E} \leq 1 .
\end{aligned}
\label{eq:RoUd3}
\end{equation}
Since the common objective function of \eqref{eq:RoUd2}
and \eqref{eq:RoUd3} is weakly$\ast$ 
(i.e.\ in $\sigma ((E^\ast)^X \times E^\ast , E^X \times E )$) 
continuous,
we have only to prove that the feasible region of \eqref{eq:RoUd3} is weakly$\ast$
dense in that of \eqref{eq:RoUd2}.
Define
\[
	\mathcal{C} 
	:=
	\set{
	((a_x - b_x)_\xin , a+  \sum_\xin b_x - u_E) 
	|
	a \in E_+ , \,
	(a_x)_\xin \in E_+^X ,\,
	 (b_x )_\xin \in \cK_\fL (X;E)
	} ,
\]
which is a convex set in $E^X \times E $ containing $0.$
Then for $((\psi_x)_\xin , \chi) \in (E^\ast)^X \times E^\ast$
we have 
\begin{align*}
	& \sum_\xin \braket{\psi_x , c_x } + \braket{\chi , c_0}
	\geq -1
	\quad
	(\forall ((c_x)_\xin , c_0) \in \mathcal{C})
	\\
	\iff &
	\braket{\chi , a}+ 
	\sum_\xin \braket{\psi_x , a_x} +
	\sum_\xin \braket{\chi - \psi_x , b_x } - \braket{\chi , u_E}
	\geq -1
	\\
	& 
	(\forall a \in E_+ , \, \forall(a_x)_\xin \in E_+^X , \, \forall
	 (b_x )_\xin \in \cK_\fL (X;E) )
	 \\
	 \iff &
	 \chi \in E_+^\ast , \,
	 (\psi_x)_\xin \in (E_+^\ast)^X ,\,
	 (\chi - \psi_x)_\xin \in \cK_\fL (X;E )^\ast , \,
	 \braket{\chi , u_E} \leq 1 .
\end{align*}
Therefore the polar of $\mathcal{C}$ in the pair 
$(E^X \times E , (E^\ast)^X \times E^\ast)$ coincides with the feasible region of
\eqref{eq:RoUd2}.
Similarly the polar of $\mathcal{C}$ in the pair
$(E^X \times E ,  (E_\ast)^X \times E_\ast )$ 
is the feasible region of \eqref{eq:RoUd3}.
Thus, by the bipolar theorem, we have only to prove that 
$\mathcal{C}$ is closed in the weak$\ast$ topology
$\sigma (E^X \times E , (E_\ast)^X \times E_\ast) .$
By the Krein-\v{S}mulian theorem, this reduces to show that
$(\mathcal{C})_r$ is weakly$\ast$ closed for any $r \in (0,\infty).$
Now suppose that
\[
	((a_x - b_x)_\xin , a+  \sum_\xin b_x - u_E) 
	\in (\mathcal{C})_r
\]
with $a \in E_+ , $ $
	(a_x)_\xin \in E_+^X , $ and $
	 (b_x )_\xin \in \cK_\fL (X;E) .$
Then
\begin{gather*}
	\| a\| , \| b_{x } \|
	\leq \left\| 
	a + \sum_{x^\prime \in X} b_{x^\prime}
	\right\|
	\leq
	\left\| 
	a + \sum_{x^\prime \in X} b_{x^\prime }
	-u_E
	\right\|
	+1
	\leq 
	r+1
	\quad (x \in X) ,
	\\
	\| a_x \|
	\leq
	\| a_x - b_x \| + \| b_x \|
	\leq 2r+1 
	\quad (x \in X) .
\end{gather*}
Therefore, by noting the weak$\ast$ closedness of $\mathcal{K}_\fL (X;E) ,$
the weak$\ast$ closedness of $(\mathcal{C})_r$ follows from 
the Banach-Alaoglu theorem as in the proof of Lemma~\ref{lemm:dense}.
Thus we have proved \eqref{eq:RoUd3}.

Now by \eqref{eq:RoUd3} there exists a sequence 
$((\psi_x^k)_\xin , \chi^k)$ $(k\in \natn)$ in the feasible region of \eqref{eq:RoUd3}
satisfying
\[
\sum_\xin \braket{\psi_x^k , \oM(x)}
> 1 + \Runs (\GM ; \fL) - \frac{1}{k} .
\]
Then 
\[
N_k := \sum_\xin \braket{\psi_x^k ,u_E} 
\geq \sum_\xin \braket{\psi_x^k , \oM(x)} >0
\]
for all $k \in \natn .$
Let $\vph_x^k := N_k^{-1} \psi_x^k$
and define
$\E_k := (\vphx^k)_\xin ,$ 
which is an ensemble.
Then we have
\begin{align}
	\Pg (\E_k ; \GM)
	&\geq \sum_\xin \braket{\vphx^k , \oM(x)}
	\notag \\ &=
	\frac{1}{N_k}
	\sum_\xin \braket{\psi_x^k , \oM(x)}
	\notag \\ &>
	\frac{1}{N_k} \left(
	1 + \Runs (\GM ; \fL) - \frac{1}{k} 
	\right) .
	\label{eq:RoU+}
\end{align}
From $(\chi^k - \psi_x^k)_\xin \in \cK_\fL (X;E)^\ast ,$
we have
\[
	\sum_\xin \braket{\chi^k - \psi_x^k , \oN(x) }
	\geq 0
	\quad (\forall \oN \in \evmsimL (X;E))
\]
and hence for any $\oN \in \evmsimL (X;E)$
\[
	1 \geq 
	\braket{\chi^k , u_E}
	=
	\sum_\xin \braket{\chi^k , \oN(x) }
	\geq 
	\sum_\xin \braket{\psi_x^k , \oN(x)}
	= 
	N_k \sum_\xin \braket{\vph_x^k , \oN (x)} .
\]
Therefore we have
\[
	\Pg (\E_k ; \fL)
	=
	\sup_{\oN \in \evmsimL (X;E)}
	\sum_\xin \braket{\vphx^k , \oN(x)}
	\leq 
	\frac{1}{N_k}
\]
By combining this with \eqref{eq:RoU+} we obtain
\[
	\Pg (\E_k ; \GM ) 
	> \Pg (\E_k ; \fL) \left(
	1 + \Runs (\GM ; \fL) - \frac{1}{k} 
	\right) ,
\]
which implies
\[
	1 + \Runs (\GM ; \fL)
	\leq 
	\sup_{k \in \natn}
	\frac{\Pg(\E_k ; \GM)}{\Pg (\E_k ; \fL)}
	\leq
	\sup_{\E \colon \mathrm{ensemble}}
	\frac{\Pg(\E ; \GM)}{\Pg (\E  ; \fL)} .
\]
Thus by Lemma~\ref{lemm:RoUineq}, the statement of Theorem~\ref{thm:RoU} is true in this case.
\end{proof}

To reduce the proof to the finite-outcome case,
we need the following lemma.

\begin{lemm} \label{lemm:RoUlsc}
Let $F$ be a classical space and let $\onon \fL \subset \ME .$
Then the extended real-valued function
\begin{equation}
	\Ch (F\to E) \ni \Ga \mapsto \Runs (\Ga ; \fL) \in [0,\infty]
	\label{eq:RoUmap}
\end{equation}
is lower semicontinuous with respect to the BW-topology, 
i.e.\
for any net $(\Ga_i)_{i\in I}$ in $\Ch (F\to E)$
BW-convergent to $\Ga \in \Ch (F\to E)$ and any $r < \Runs(\Ga ; \fL) ,$
$r < \Runs (\Ga_i ; \fL)$ holds eventually.
\end{lemm}
\begin{proof}
Suppose that there exist 
a net $(\Ga_i)_{i\in I}$ in $\Ch (F\to E)$
BW-convergent to $\Ga \in \Ch (F\to E) $
and
$r < \Runs(\Ga ; \fL) $
such that
$\Runs (\Ga_i ; \fL) \leq r $ frequently.
By taking a subnet we may assume 
$\Runs (\Ga_i ; \fL) \leq r $ for all $\iin .$
Then for every $\iin$ there exists a simulable measurement
$\Psi_i \in \ChwsimL (F\to E)$ such that
$(1+r) \Psi_i \geq \Gamma_i .$
Since $\ChwsimL (F\to E)$ is BW-compact by Corollary~\ref{coro:chlcomp},
there exists a subnet $(\Psi_\ikj)_\jin $
BW-converging to $\Psi \in \ChwsimL (F\to E) .$
Then by the weak$\ast$ closedness of $E_+$
we have $ (1+r)\Psi \geq \Ga $
and hence
$\Runs (\Ga ;\fL) \leq r ,$
which contradicts the assumption.
Therefore \eqref{eq:RoUmap} is lower semicontinuous.
\end{proof}

\noindent
\textit{Proof of Theorem~\ref{thm:RoU}.}
By Lemmas~\ref{lemm:exgain} and \ref{lemm:RoUe3}
we have only to prove \eqref{eq:RoU} when 
$F$ has the Banach predual $F_\ast$ and $\Ga$ is a \wstar-measurement.
Then by Theorem~\ref{thm:finapp} there exists a net
$(\La_\Delta)_{\Delta \in \DF}$ in $\Ch (F \to E)$
such that 
$([\La_\Delta])_{\Delta \in \DF}$ is an increasing net in $\MfinE ,$
$[\La_\Delta] \wto \sup_{\Delta \in \DF} [\La_\Delta] = [\Ga] ,$
and 
$\La_\Delta \bwto \Gamma .$
Thus from Lemmas~\ref{lemm:RoUe2} and \ref{lemm:RoUlsc} we have
\begin{equation}
	\Runs (\La_\Delta ; \fL)  \uparrow \Runs (\Ga ; \fL ) .
	\label{eq:toch1}
\end{equation}
Since $\La_\Delta$ is post-processing equivalent to a finite-outcome measurement,
Lemma~\ref{lemm:RoUfin} implies
\begin{equation}
	1+ \Runs (\La_\Delta ; \fL) = 
	\sup_{\E \colon \mathrm{ensemble}}
	\frac{\Pg(\E ; \La_\Delta)}{\Pg (\E ; \fL )} .
	\label{eq:toch2}
\end{equation}
From \eqref{eq:toch1} and \eqref{eq:toch2} we obtain
\begin{align*}
	1 + \Runs (\Ga ;\fL)
	&= \sup_{\Delta \in \DF}
	(1+ \Runs (\Lambda_\Delta ; \fL))
	\\
	&=
	\sup_{\Delta \in \DF} \sup_{\E \colon \mathrm{ensemble}} 
	\frac{\Pg (\E ; \La_\Delta)}{\Pg (\E ; \fL)}
	\\
	&=
	\sup_{\E \colon \mathrm{ensemble}} 
	\sup_{\Delta \in \DF}
	\frac{\Pg (\E ; \La_\Delta)}{\Pg (\E ; \fL)}
	\\
	&=
	\sup_{\E \colon \mathrm{ensemble}} 
	\frac{\Pg (\E ; \Ga)}{\Pg (\E ; \fL)},
\end{align*}
where the last equality follows from the post-processing monotonicity and 
the weak continuity of the gain functional $\Pg (\E ; \cdot) .$
\qed

As finishing this section, we slightly generalize Theorem~\ref{thm:RoU}
to the partitioned ensembles,
which we will consider again in Section~\ref{sec:incomp}.

\begin{defi}[Partitioned ensemble] \label{def:pens}
\begin{enumerate}[1.]
\item
For a finite set $X \nono ,$
a family $\vecE = (\E_x)_\xin$
of \wstar-families $\E_x = (\vph_{x,y})_{y \in Y_x}$
$(x \in X)$ is called a \textit{partitioned ensemble}
if $\vph_{x,y} \geq 0$
$(x \in X , y \in Y_x)$
and $\vecE$ satisfies the normalization condition
\[
	\sum_{x \in X } \sum_{y \in Y_x} 
	\braket{\vph_{x,y} , u_E} 
	=1 .
\]
Each component $\E_x$ is called a subensemble of $\vecE. $
\item
Let $\vecE $ be a partitioned ensemble.
For each measurement
$\Ga \in \Ch (F\to E) ,$
each measurement $\omega \in \ME ,$
and
each subset $\onon \fL \subset \ME$ we define
\begin{gather*}
	\Pg(\vecE ; \Ga) := \sum_{x\in X} \Pg (\E_x ; \Ga)
	, \\
	\Pg(\vecE ; \omega) := \sum_{x\in X} \Pg (\E_x ; \omega)
	, \\
	\Pg (\vecE ; \fL) := \sum_\xin \Pg (\E_x ; \fL) ,
\end{gather*}
where $\vecE = (\E_x)_\xin .$ \qed
\end{enumerate}
\end{defi}
The operational meaning of the quantity 
$\Pg (\vecE ; \fL)$ in Definition~\ref{def:pens} is as follows.
Consider the situation where Alice prepares the system\rq{}s state according to the 
ensemble $(\vph_{x,y})_{x\in X ; y \in Y_x} $
and Bob can perform only the measurements belonging to $\fL .$
We also assume that before Bob perform a measurement, 
Alice announces the value of $x \in X$ to Bob
so that to Bob the system\rq{}s state corresponds to (up to the normalization factor) 
the subensemble $\E_x .$
Then Bob chooses an appropriate measurement from $\fL ,$ 
perform it, and guesses the original label $y \in Y_x $ based on 
the measurement outcome.
The quantity $\Pg (\vecE ; \fL)$ is then the optimal average probability 
that Bob can correctly guess the label $y \in Y_x .$
A similar interpretation also applies to the quantity $\Pg (\vecE ; \Ga) .$

Now Theorem~\ref{thm:RoU} is generalized to
\begin{coro}\label{coro:RoU}
Let $\fL$ and $\Ga$ be the same as in Theorem~\ref{thm:RoU}.
Then 
\begin{equation}
	1+\Runs (\Ga ;\fL)
	= \sup_{\vecE \colon \mathrm{partitioned \, ensemble}}
	\frac{\Pg(\vecE ; \Ga)}{\Pg(\vecE ; \fL)} .
	\label{eq:RoUp}
\end{equation}
\end{coro}
\begin{proof}
Let $\vecE = (\E_x)_{x\in X}$ be a partitioned ensemble.
Then by Theorem~\ref{thm:RoU}, we have 
\[
 	\Pg (\E_x ; \Ga) \leq (1+ \Runs(\Ga ; \fL)) \Pg (\E_x ; \fL) 
 	\quad
 	(x \in X).
\]
Hence
\begin{align*}
	\Pg (\vecE ; \Ga )
	&= \sum_\xin \Pg (\E_x ; \Ga)
	\\
	&\leq
	 (1+ \Runs(\Ga ; \fL)) \sum_\xin  \Pg (\E_x ; \fL)
	 \\
	 &=
	  (1+ \Runs(\Ga ; \fL))  \Pg (\vecE ; \fL) .
\end{align*}
Therefore by Theorem~\ref{thm:RoU} we obtain
\begin{align*}
	1+ \Runs (\Ga ;\fL)
	&\geq
	\sup_{\vecE \colon \mathrm{partitioned \, ensemble}}
	\frac{\Pg(\vecE ; \Ga)}{\Pg (\vecE ; \fL)}
	\\
	&\geq 
	\sup_{\E \colon \mathrm{ensemble}}
	\frac{\Pg(\E ; \Ga)}{\Pg (\E ; \fL)}
	\\
	&= 
	1+ \Runs (\Ga ; \fL) ,
\end{align*}
where the second inequality follows by restricting $\vecE$
to single-element families.
\end{proof}

\section{Extremal, maximal, and simulation irreducible measurements}
\label{sec:irr}
In this section, we show some basic properties of 
extremal, maximal, and simulation irreducible 
measurements and prove that any measurement is simulable by the 
set of simulation irreducible measurements (Theorem~\ref{thm:irrsim}).

\subsection{Extremal measurement}
\label{subsec:extremal}
As we have seen in Theorem~\ref{thm:cconvME}, 
the measurement space $\ME$ can be regarded as compact convex set 
and hence has sufficiently many extremal points
by the Krein-Milman theorem.
Here one should not confuse the extremality in $\ME$ and the 
\lq\lq{}extremal POVM,\rq\rq{} (e.g.\ \cite{busch2016quantum}, Section~9.3),
which in our terminology corresponds to the extremal point of 
$\Chw (F \to E)$ for some fixed classical space $F .$

In this subsection 
we prove the following theorem that characterizes the set $\de \ME$ of extremal points of $\ME .$
This is a generalization of the corresponding result for classical statistical experiments
(\cite{torgersen1991comparison}, Theorem~7.3.15, (i)$\iff$(vi)).
See also \cite{gutajencova2007} (Corollary~3.8) 
for the related result on the extremality of quantum statistical experiments.

\begin{thm}[Characterization of extremal measurement] \label{thm:extremal}
A measurement $\omega \in \ME$ is an extremal point of $\ME$ if and only if
$\omega$ has a representative $\Ga \in \Chw (F \to E)$
that is an injection.
\end{thm}
For the proof we need the following lemma.
\begin{lemm}\label{lemm:idex}
Let $F$ be an order unit Banach space.
Then the identity channel $\id_{F}$ is an extremal point of 
$\Ch (F \to F) .$
\end{lemm}
\begin{proof}
Take channels $\Psi_1 , \Psi_2 \in \Ch (F \to F)$
and $\lambda \in (0,1)$
such that $\id_F = \lambda \Psi_1 + (1-\lambda ) \Psi_2 . $
Then for any pure state $\psi \in \de S(F)$ we have
\[
	\psi 
	= \id^\ast_F (\psi)
	=
	\lambda \Psi_1^\ast (\psi)
	+
	(1-\lambda ) \Psi_2^\ast (\psi) ,
\]
where the star denotes the Banach dual map.
Since $\Psi_j^\ast (\psi) \in S(F)$ $(j=1,2),$
the extremality of $\psi$ implies $\Psi_1^\ast (\psi) = \Psi_2^\ast (\psi ) = \psi . $
Since $\lin (S(F)) = F^\ast ,$ the Krein-Milman theorem and the weak$\ast$ continuity of the Banach dual maps imply
$\Psi_1^\ast = \Psi_2^\ast = \id_{F^\ast }.$
Therefore $\Psi_1 = \Psi_2 = \id_F ,$
which proves the claim.
\end{proof}
We remark that a proof similar to the above one applies to the more general result that the identity map on an arbitrary Banach space $X$ is an extremal point of the unit ball of the set bounded operators on $X $ \cite{1961234}.

\noindent
\textit{Proof of Theorem~\ref{thm:extremal}.}
Assume $\omega \in \de \ME$ and take a minimally sufficient 
(see Appendix~\ref{app:ms}) representative
$\Ga \in \Chw (F\to E)$ of $\omega .$
We show the injectivity of $\Ga .$ 
Take an element $a \in F$ such that $\Ga (a) = 0 .$
Without loss of generality we may assume $\| a \| \leq 1 .$
Let $e_\pm : = \frac{1}{2}(u_F \pm a) $
and define $\Ga_\pm \colon F \to E$ by 
$\Ga_\pm (b ) := 2 \Ga (b \cdot e_\pm)$
$(b \in F) .$
Then $0 \leq e_\pm \leq u_F$ and 
$\Ga_\pm \in \Chw (F \to E) .$
Define channels  
$\Psi_1 \in \Ch (F \to F\oplus F)$
and
$\Psi_2 \in \Ch (F\oplus F \to F)$
by
\[
	\Psi_1 (b) := b\oplus b ,
	\quad
	\Psi_2 (b\oplus c)
	:=
	b\cdot e_+ + c\cdot e_-
	\quad
	(b,c \in F) . 
\]
Then we have
\begin{gather*}
	 \left( \frac{1}{2} \Ga_+\oplus \frac{1}{2}  \Ga_- \right)
	\circ \Psi_1 (b)
	=
	\Ga (b\cdot e_+) + \Ga (b\cdot e_-)
	=
	\Ga (b) 
	\quad
	(b \in F) ,
	\\
	\Ga \circ \Psi_2 (b \oplus c)
	=
	\Ga (b\cdot e_+) + \Ga (c\cdot e_-)
	=
	\left( \frac{1}{2} \Ga_+\oplus \frac{1}{2}  \Ga_- \right)
	(b \oplus c)
	\quad 
	(b,c \in F),
\end{gather*}
which implies $\Ga \ppeq \frac{1}{2} \Ga_+\oplus \frac{1}{2}  \Ga_-  .$
Thus by the extremality of $\omega = [\Ga] $
it follows that $\Ga \ppeq \Ga_\pm .$
Therefore by Proposition~\ref{prop:postw*} there exist 
channels $\Phi_\pm \in \Chw (F \to F)$ such that
$\Ga_\pm = \Ga \circ \Phi_\pm .$
Then we have
\[
	\Ga = 
	\frac{1}{2} \Ga_+ 
	+ \frac{1}{2} \Ga_-
	=
	\Ga \circ \left( \frac{1}{2} \Phi_+ + \frac{1}{2} \Phi_- \right).
\]
By the minimal sufficiency of $\Ga , $ this implies 
$\frac{1}{2} \Phi_+ + \frac{1}{2} \Phi_- = \id_F$
and hence by Lemma~\ref{lemm:idex} we obtain $\Phi_\pm = \id_F .$
Therefore we have
\[
	\Ga (b) 
	=\Ga \circ \Phi_+ (b)
	=
	\Ga_+ (b)
	=
	2\Ga(b \cdot e_+)
	\quad (b \in F) .
\]
Now suppose that $\| 2e_+ \| >1 .$
Then there exists a non-zero projection $Q \in F$ and $\delta >0$
such that $2 Q \cdot e_+ \geq (1+\delta ) Q .$
Thus 
\[
	\Ga (Q) = \Gamma (2 Q \cdot e_+ )
	\geq 
	(1+\delta) \Ga (Q) ,
\]
which implies $\Ga (Q) =0 .$
Hence by the faithfulness (cf.\ Appendix~\ref{app:ms}) of $\Ga$ we obtain $Q =0,$
which is a contradiction.
Therefore $\| 2e_\pm \| \leq 1 .$
From $\Gamma (u_F - 2e_+) = 0 $
and $2e_+ \leq u_F ,$
the faithfulness of $\Ga$ again yields $2e_+  = u_F .$
Therefore we obtain $a = 0 ,$ proving the injectivity of $\Ga .$

Conversely suppose that $\Ga \in \Chw (F\to E)$ is an injective representative of
$\omega .$
To show the extremality of $\omega ,$
take \wstar-measurements 
$\La_j \in \Chw (G_j \to F)$
$(j=1,2)$
and $\lambda \in (0,1)$ satisfying
$\Ga \ppeq \lambda \La_1 \oplus (1-\lambda ) \La_2 .$
Then there exists channels 
$\Xi \in \Ch (F \to G_1 \oplus G_2)$
and 
$\Theta \in \Ch (G_1 \oplus G_2 \to F)$
such that
\[
	\Ga = (\lambda \La_1 \oplus (1-\lambda ) \La_2 ) \circ \Xi ,
	\quad
	\lambda \La_1 \oplus (1-\lambda ) \La_2 
	=
	\Ga \circ \Theta .
\]
Then we have 
$\lambda \La_1 (b) = \Ga \circ \Theta (b \oplus 0)$
$(b \in G_1) .$
Thus by putting $b = u_{G_1} $ we obtain
\[
	\Ga (\lambda^{-1} \Theta (u_{G_1} \oplus 0))
	=
	\La_1 (u_{G_1})
	=
	u_E
	=
	\Ga (u_F) .
\]
By the injectivity of $\Ga$ this implies 
$\lambda^{-1} \Theta (u_{G_1} \oplus 0) = u_F .$
Hence we may define $\Phi_1 \in \Ch (G_1 \to F)$
by
\[
	\Phi_1 (b) := \lambda^{-1} \Theta (b \oplus 0)
	\quad
	(b \in G_1) .
\]
Similarly the linear map $\Phi_2$ defined by
\[
	\Phi_2 (c) := (1-\lambda)^{-1} \Theta (0 \oplus c)
	\quad
	(c\in G_2)
\]
is a channel in $\Ch (G_2 \to F) .$
Then by the definition of $\Phi_j$ we have
$\La_j = \Ga \circ \Phi_j \pp \Ga $
$(j=1,2) .$
If we write as
\[
	\Xi (a) = 
	\Xi_1 (a)\oplus \Xi_2 (a)
	\quad
	(a \in F) ,
\]
where $\Xi_j \in \Ch (F \to G_j),$
then
\[
	\Ga
	=
	\lambda \La_1 \circ \Xi_1 
	+
	(1-\lambda) \La_2 \circ \Xi_2 
	=
	\Ga \circ (
	\lambda \Phi_1 \circ \Xi_1 
	+ 
	(1-\lambda) \Phi_2 \circ \Xi_2 
	).
\]
By the injectivity of $\Ga$ this implies 
$	\lambda \Phi_1 \circ \Xi_1 
	+ 
	(1-\lambda) \Phi_2 \circ \Xi_2 
	= \id_F .
$
Hence by Lemma~\ref{lemm:idex} we have
$\Phi_1 \circ \Xi_1 = \Phi_2 \circ \Xi_2 = \id_F .$
Therefore we obtain
\[
	\Ga = \Ga \circ \Phi_j \circ \Xi_j
	= \La_j \circ \Xi_j 
	\pp 
	\La_j
	\quad
	(j=1,2) .
\]
This proves $\Ga \ppeq \La_1 \ppeq \La_2 $
and hence $\omega = [\Ga]$ is extremal.
\qed

\subsection{Maximal measurement} \label{subsec:maxmeas}
We next study post-processing 
maximal measurements \cite{Martens1990,Dorofeev1997349,buscemi2005clean}.

\begin{defi} \label{def:maximal}
A measurement $\omega \in \ME$ is said to be \textit{post-processing maximal,}
or just \textit{maximal,}
if $\omega$ is a maximal element of $\ME $ in the post-processing order,
i.e.\ for any $\nu \in \ME ,$
$\omega \pp \nu$ implies $\nu \pp \omega .$
The set of maximal measurements in $\ME$ is denoted by
$\MmaxE .$ \qed
\end{defi}
From Theorem~\ref{thm:compact}, Proposition~\ref{prop:posp},
and Lemma~\ref{lemm:cposet},
application of Zorn\rq{}s lemma immediately gives
\begin{coro}\label{coro:maxexists}
$\ME = {\downarrow \MmaxE} ,$ i.e.\
every measurement in $\ME$ is a post-processing of a
maximal measurement.
\end{coro}
For a finite-outcome POVM $\oM$ on a quantum system, 
$[\GM]$ is maximal if and only if each element of
$\oM$ is rank-$1$
\cite{Martens1990}.
A similar characterization can be shown for continuous-outcome quantum POVMs
\cite{1751-8121-46-2-025303,kuramochi2018incomp}.

We now prove
\begin{prop} \label{prop:maxface}
$\MmaxE$ is a face of $\ME .$
\end{prop}
For the proof of Proposition~\ref{prop:maxface}
we need the following lemma, 
which can be shown in the same way as in \cite{kuramochi2018directedv1}
(Lemma~6).
\begin{lemm} \label{lemm:subalg}
Let $\Ga \in \Chw (F \to E) $ and $\La \in \Chw (G \to E)$
be \wstar-measurements.
Suppose $\Ga \pp \La .$
Then there exists a \wstar-measurement $\tL \in \Chw (\wt{G} \to E)$
such that $\La \ppeq \tL ,$
$F$ is included in $\wt{G}$ as a unital subalgebra of $\wt{G} ,$
and $\Ga$ is the restriction of $\tL$ to the subalgebra $F .$
\end{lemm}
Here for a classical space $G ,$ a subset $F \subset G$ is called a (unital) subalgebra of $G$
if $F$ is linear subspace of $G$ (containing the unit $u_G$) and
$F$ is closed under the multiplication on $G .$

\noindent \textit{Proof of Proposition~\ref{prop:maxface}.}
(Convexity).
Take $\omega_1 , \omega_2 \in \MmaxE $
and $\la \in (0,1) .$
We prove $\la \omega_1 + \ola \omega_2 \in \MmaxE .$
Assume $\la \omega_1 + \ola \omega_2 \pp \nu \in \ME$
and let $\Ga_j \in \Chw(F_j \to E)$
$(j=1,2)$ and $\La \in \Chw (G \to E)$
be representatives of $\omega_j$ and $\nu ,$
respectively.
By Lemma~\ref{lemm:subalg}, we can take $\La$ so that
$F_1 \oplus F_2 $ is a unital subalgebra of $G$
and $\la \Ga_1 \oplus \ola \Ga_2$ is the restriction of $\La$
to $F_1 \oplus F_2  .$
We define projections $\s_1 , \s_2 \in G$
by
\[ \s_1 := u_{F_1} \oplus 0 , \quad \s_2 := 0 \oplus u_{F_2} \] .
Then \wstar-measurements $\La_j \in \Chw (\s_j \cdot G \to E)$
$(j=1,2)$ are well-defined by
\begin{gather*}
	\La_1 (a) := \la^{-1} \La (a) \quad (a \in \s_1 \cdot G) ,
	\\
	\La_2 (b) := \ola^{-1} \La (b) \quad (b \in \s_2 \cdot G),
\end{gather*}
where $\s_j \cdot G := \set{\s_j \cdot c | c \in G}$
which has the order unit $\s_j .$
Indeed we have
\begin{gather*}
	\La_1 (\s_1) 
	=\la^{-1} [\la \Ga_1 \oplus \ola \Ga_2] (u_{F_1}\oplus 0) 
	=\Ga_1 (u_{F_1}) = u_E ,
	\\
	\La_2 (\s_2) 
	=\ola^{-1} [\la \Ga_1 \oplus \ola \Ga_2] (0 \oplus u_{F_2}) 
	=\Ga_2 (u_{F_2}) = u_E .
\end{gather*}
Define channels $\Psi_j \in \Ch (F_j \to \s_j \cdot G)$ by 
\begin{equation*}
	\Psi_1 (a) : = a\oplus 0 \quad (a \in F_1) ,
	\quad
	\Psi_2 (b) := 0 \oplus b \quad (b \in F_2) .
\end{equation*}
Then we have 
\[
	\La_1 \circ \Psi_1 (a) = \la^{-1} \La (a\oplus 0) = \Ga_1 (a) 
	\quad
	(a \in F_1) 
\]
and similarly $\La_2 \circ \Psi_2 = \Ga_2 .$
Thus by the maximality of $[\Ga_1] = \omega_1$ and $[\Ga_2]=\omega_2$
there exist channels $\Phi_j \in \Chw (\s_j \cdot G \to F_j)$
$(j=1,2)$ such that
$\La_j = \Ga_j \circ \Phi_j .$
Then for any $c \in G, $
\begin{align*}
	\La (c) 
	&= \La (\s_1 \cdot c) + \La (\s_2 \cdot c)
	\\
	&=
	\la \La_1 (\s_1 \cdot c) + \ola \La_2 (\s_2 \cdot c)
	\\
	&= 
	\la \Ga_1 \circ \Phi_1 (\s_1 \cdot c)
	+ \ola \Ga_2 \circ \Phi_2 (\s_2 \cdot c)
	\\
	&=
	( \la \Ga_1 \oplus \ola \Ga_2 ) \circ \tPhi (c) ,
\end{align*}
where we defined $\tPhi \in \Ch (G \to F_1 \oplus F_2)$
by 
$
\tPhi (c) := \Phi_1 (\s_1 \cdot c) \oplus \Phi_2 (\s_2 \cdot c)
$
$(c \in G) .$
Therefore this shows $\nu = [\La] \pp \la \omega_1 + \ola \omega_2 ,$
which proves the convexity of $\MmaxE .$

(Extremality).
Take 
$\nu_1 , \nu_2 \in \ME $
and
$\la \in (0,1)$ such that 
$\la \nu_1 + \ola \nu_2 \in \MmaxE .$
If $\nu_1 \pp \nu_1^\prime \in \ME ,$
we have 
\[
	\la \nu_1 + \ola \nu_2 
	\pp
	\la \nu_1^\prime + \ola \nu_2 
\]
by Proposition~\ref{prop:direct}.
Therefore the maximality of $\la \nu_1 + \ola \nu_2 $ implies
\[
	\la \nu_1 + \ola \nu_2 
	=
	\la \nu_1^\prime + \ola \nu_2 
\]
and hence $\nu_1 = \nu_1^\prime ,$
which proves $\nu_1 \in \MmaxE .$
We can show $\nu_2 \in \MmaxE  $ similarly.
Thus $\MmaxE $ is a face.
\qed

\subsection{Simulation irreducible measurement} \label{subsec:sirr}
We now introduce the simulation irreducibility of measurements,
generalizing the finite-dimensional concept in \cite{PhysRevA.97.062102}.

\begin{defi}[Simulation irreducible measurement] \label{def:irr}
A measurement $\omega \in \ME$ is said to be \textit{simulation irreducible}
if $\omega \in \ssimuL$ implies $\omega \in \fL$ for any subset 
$\fL \subset \ME .$
The set of simulation irreducible measurements in $\ME$ is denoted by
$\MirrE .$ \qed
\end{defi}
A simulation irreducible measurement is a measurement that can be 
simulated only by itself.

Now we give equivalent characterizations of the simulation irreducibility.

\begin{prop}[cf. \cite{PhysRevA.97.062102}, Proposition~5]
\label{prop:irrexmax}
For a measurement $\omega \in \ME ,$ the following conditions are equivalent.
\begin{enumerate}[(i)]
\item \label{i:em1}
$\omega \in \MmaxE \cap \de \ME .$
\item \label{i:em2}
For any subset $\fL \subset \ME ,$
$\omega \in \simL$ implies $\omega \in \overline{\fL} ,$
where the closure is with respect to the weak topology.
\item \label{i:em3}
$\omega$ is simulation irreducible.
\end{enumerate}
Specifically, $\MirrE = \MmaxE \cap \de \ME  $ holds.
\end{prop}
\begin{proof}
\eqref{i:em1}$\implies$\eqref{i:em2}. 
Assume \eqref{i:em1} and suppose $\omega \in \simL $ for some subset
$\fL \subset \ME .$
Then there exists a measurement $\nu \in \cconv (\fL)$ such that
$\omega \pp \nu .$
Since $\omega $ is maximal, this implies $\omega = \nu \in \cconv(\fL) .$
By the extremality of $\omega ,$ this implies that $\omega$ is also an 
extremal point of $\cconv (\fL).$
Hence by Milman\rq{}s partial converse to the Krein-Milman theorem 
(\cite{schaefer1999topological}, II.10.5)
we have $\omega \in \ovl{\fL} .$

\eqref{i:em2}$\implies$\eqref{i:em3}.
Assume \eqref{i:em2} and let $\omega \in \ssimuL$ for some 
subset $\fL \subset \ME .$
Then there exists a finite subset $\mathfrak{F} \subset \fL$
such that $\omega$ is a post-processing of a convex combination of 
elements of $\mathfrak{F} .$
Thus $\omega \in \ssimu (\mathfrak{F}) = \simu (\mathfrak{F}) .$
Then by assumption we have 
$\omega \in \ovl{\mathfrak{F}} = \mathfrak{F} \subset \fL ,$
which proves the simulation irreducibility of $\omega .$

\eqref{i:em3}$\implies$\eqref{i:em1}.
Assume \eqref{i:em3}.
Let $\omega \pp \omega^\prime  \in \ME .$
Then $\omega \in \ssimu (\{ \omega^\prime \})$
and the assumption implies $\omega = \omega^\prime .$
Thus $\omega \in \MmaxE .$
To prove the extremality, take  
$\omega_1 , \omega_2 \in \ME$ and $\la \in (0,1)$
such that 
\begin{equation}
	\omega = \la \omega_1 + \ola \omega_2 .
	\label{eq:omegas}
\end{equation}
Then $\omega \in \ssimu (\{\omega_1 , \omega_2 \})$
and the simulation irreducibility of $\omega$ implies either 
$\omega = \omega_1$ or $\omega = \omega_2 .$
In both cases, from \eqref{eq:omegas} we obtain $\omega = \omega_1 =\omega_2 .$
Therefore $\omega \in \de \ME .$ 
\end{proof}
In \cite{PhysRevA.97.062102},
it is shown that a finite-outcome measurement $\GM$ with
$\oM \in \evm (X;E)$ is simulation irreducible
if and only if it is maximal and $\oM$ is extremal in $\evm (X;E) ,$ 
which is a condition different from our extremality in $\ME .$
Indeed, for maximal measurements, we can show these two notions of extremality 
coincide as in the following proposition.
\begin{prop} \label{prop:irrex}
Let $\omega \in \MmaxE$ be a maximal measurement
and let $\Gamma  \in \Chw (F \to E)$ be a minimally sufficient representative of 
$\omega .$
Then $\omega \in \de \ME$ if and only if
$\Ga \in \de \Chw (F \to E) .$
\end{prop}
\begin{proof}
Suppose $\omega \in \de \ME .$
Then by  
Theorem~\ref{thm:extremal} and the uniqueness of the minimally sufficient measurement (Proposition~\ref{prop:msmeas}), $\Ga$ is injective.
To show $\Ga \in \de \Chw (F \to E) ,$
take $\la \in (0,1)$ and $\Ga_1 , \Ga_2 \in \Chw (F \to E) $ 
such that $\Ga = \la \Ga_1 + \ola \Ga_2 .$
Then from Proposition~\ref{prop:direct} we have
\[
	\Ga \pp \la \Ga_1 \oplus \ola \Ga_2
\]
and hence $\omega = [\Ga] \in \ssimu (\{[\Ga_1] , [\Ga_2]  \}) .$
Therefore by Proposition~\ref{prop:irrexmax} this implies $[\Ga] = [\Ga_1] = [\Ga_2] .$
Thus there exist channels $\Psi_j \in \Ch (F\to F)$
$(j=1,2)$ such that $\Ga_j = \Ga \circ \Psi_j .$
Thus 
\[
	\Ga = \la \Ga_1 + \ola \Ga_2 
	=\Ga \circ (\la \Psi_1 + \ola \Psi_2) .
\]
and the injectivity of $\Ga$ implies $\la \Psi_1 + \ola \Psi_2 = \id_F .$
Hence by Lemma~\ref{lemm:idex} we obtain $\Psi_1 = \Psi_2 = \id_F .$
Therefore $\Ga_1 = \Ga_2 = \Ga ,$
which proves $\Ga \in \de \Chw (F \to E) .$

Conversely assume $\Ga \in \de \Chw (F \to E) $
and take $\la \in (0,1)$ and $\omega_1 , \omega_2 \in \ME$
such that $\omega = \la \omega_1 + \ola \omega_2 .$
Let $\La_j \in \Chw (G_j \to E)$ be a representative of $\omega_j$
$(j=1,2).$
Since $\Ga \ppeq \la \La_1 \oplus \ola \La_2 ,$
there exists a channel $\Theta \in \Chw (F \to G_1 \oplus G_2)$
such that 
\[
	\Ga = (\la \La_1 \oplus \ola \La_2) \circ \Theta .
\]
If we write as
\[
\Theta (a) = \Theta_1 (a) \oplus \Theta_2 (a)
\quad
(a \in F) ,
\]
then $\Theta_j \in \Chw (F\to G_j)$
$(j=1,2)$ and we have
\[
	\Ga = 
	\la \La_1 \circ \Theta _1 +
	\ola \La_2 \circ \Theta _2  .
\]
Therefore the extremality of $\Ga$ in $\Chw (F\to E)$ 
implies $\Ga = \La_1 \circ \Theta_1
= \La_2 \circ \Theta_2 $
and hence $\omega \pp \omega_1 , \omega_2.$
Thus the maximality of $\omega$ implies $\omega=\omega_1 = \omega_2 ,$
which proves $\omega \in \de \ME .$ 
\end{proof}

\subsection{Simulability by simulation irreducible measurements}
\label{subsec:irrmain}
We now show that every measurement is simulable by 
the set of simulation irreducible measurements,
generalizing the finite-dimensional results in \cite{Haapasalo2012,PhysRevA.97.062102}.
While the proof for the corresponding finite-dimensional result 
\cite{Haapasalo2012,PhysRevA.97.062102}
is constructive, the proof of the following theorem, 
a part of which is analogous to the common proof of the Krein-Milman theorem,
is non-constructive and based on the well-ordering theorem and Theorem~\ref{thm:simdisc}.
\begin{thm} \label{thm:irrsim}
$\ME = \simu (\MirrE) ,$
i.e.\ every measurement in $\ME$ is simulable by 
the set of simulation irreducible measurements.
\end{thm}
\begin{proof}
By Theorem~\ref{thm:simdisc}, we have only to prove that for any measurement 
$\omega_0 \in \ME$ and an ensemble $\E \in \Ens (E)$
the inequality
\[
	\Pg (\E ; \omega_0) \leq \sup_{\nu \in \MirrE} \Pg (\E ; \nu)
\]
holds, where $\Ens (E)$ is the set of ensembles defined 
in Proposition~\ref{prop:small}.
Well-order $\Ens (E)$ so that $\Ens (E) = \set{ \E_\alpha | 0 \leq \alpha < \gamma }$
and $\E_0 = \E ,$
where the index $\alpha$ runs over all the ordinals smaller than the ordinal $\gamma .$
Define
\[
	F_0 :=
	\set{
	\omega \in \ME |
	\Pg(\E_0 ; \omega) = \sup_{\nu \in \ME} \Pg (\E_0 ; \nu )
	} ,
\]
which is a non-empty compact face of $\ME .$ 
We then inductively define $(F_\alpha)_{0 \leq \alpha < \gamma}$
by
\[
	F_\alpha 
	:=
	\set{\omega \in \bigcap_{0 \leq \beta < \alpha} F_\beta 
	|
	\Pg (\E_\alpha ; \omega)
	=
	\sup_{\nu \in \bigcap_{0 \leq \beta < \alpha} F_\beta }
	\Pg (\E_\alpha ; \nu)
	} 
	\quad (0 < \alpha < \gamma).
\]
Then $(F_\alpha)_{0 \leq \alpha < \gamma}$ is a decreasing transfinite sequence 
of compact faces.
Moreover, if $F_\beta \nono$ for all $0 \leq \beta < \alpha ,$
then, being the intersection of compact sets satisfying 
the finite-intersection property, $\bigcap_{0\leq \beta < \alpha} F_\beta$
is non-empty, and hence so is $F_\alpha .$
Thus by induction $F_\alpha \nono$ for all $0 \leq \alpha < \gamma .$
Therefore $F := \bigcap_{0 \leq \alpha < \gamma} F_\alpha$
is a non-empty compact face.
If $\omega_1 , \omega_2 \in F ,$ then 
\[
	\Pg (\E_\alpha ; \omega_1) 
	=\sup_{\nu \in \bigcap_{0 \leq \beta < \alpha} F_\beta}
	\Pg (\E_\alpha ; \nu)
	=
	\Pg (\E_\alpha ; \omega_2)
\]
for all $0 \leq \alpha <\gamma $
and hence Theorem~\ref{thm:bss} implies $\omega_1 =\omega_2 .$
Therefore $F$ is a singleton $\{ \wt{\nu}  \} .$
Since $F$ is a face in $\ME,$ $\wt{\nu} $ is an extremal point of $\ME .$
To show the maximality of $\wt{\nu} ,$
take a measurement 
$\wt{\nu}^\prime \in \ME$ satisfying $\wt{\nu} \pp \wt{\nu}^\prime .$
Then we have
\[
	\Pg (\E_0 ; \wt{\nu}^\prime)
	\leq 
	\sup_{\nu \in \ME}
	\Pg (\E_0 ; \nu)
	=\Pg (\E_0 ; \wt{\nu})
	\leq \Pg (\E_0 ; \wt{\nu}^\prime) ,
\]
which implies $\wt{\nu}^\prime \in F_0 .$
For $0< \alpha < \gamma ,$
assume $\wt{\nu}^\prime \in \bigcap_{0 \leq \beta < \alpha}F_\beta .$
Then 
\[
	\Pg(\E_\alpha ; \wt{\nu}^\prime)
	\leq
	\sup_{\nu \in \bigcap_{0 \leq \beta < \alpha}F_\beta }
	\Pg (\E_\alpha ; \nu)
	= \Pg (\E_\alpha ; \wt{\nu})
	\leq
	\Pg (\E_\alpha ; \wt{\nu}^\prime)
\]
and hence $\wt{\nu}^\prime \in F_\alpha .$ Therefore by induction 
we have $\wt{\nu}^\prime \in \bigcap_{0 \leq \alpha < \gamma} F_\alpha
= \{ \wt{\nu} \} $
and hence $\wt{\nu} = \wt{\nu}^\prime .$
This proves that $\wt{\nu}$ is maximal.
Thus by Proposition~\ref{prop:irrexmax} 
we have $\wt{\nu} \in \MirrE .$
Furthermore from $\wt{\nu} \in F_0 $ and $\E = \E_0 $ we have
\[
	\Pg (\E ; \omega_0) \leq
	\sup_{\nu \in \ME} \Pg (\E; \nu)
	= \Pg (\E ; \wt{\nu}) 
	\leq \sup_{\nu \in \MirrE } \Pg (\E ; \nu) ,
\]
which proves the claim.
\end{proof}

\section{Incompatibility and robustness of incompatibility} \label{sec:incomp}
In this section, we consider incompatibility of measurements
and generalizes some known results in finite dimensions
\cite{PhysRevLett.122.130402,PhysRevLett.122.130403,PhysRevLett.122.130404}.
The main result in this section is Theorem~\ref{thm:RoI} that characterizes the operational meaning of the robustness of incompatibility.

\subsection{Basic properties of (in)compatible measurements}
\label{subsec:incc}

\begin{defi}[Compatibility and incompatibility of measurements] \label{def:incomp}
Let $X\nono $ be a set which may be finite or infinite. 
A family $(\Ga_x)_\xin$ of measurements is called \textit{compatible},
or \textit{jointly measurable},
if there exists a measurement $\La $ such that 
$\Ga_x \pp \La $ for all $\xin $
and \textit{incompatible} if not.
Such a measurement $\La ,$ if exists, is called a mother measurement.
We can always take $\La$ to be a \wstar-measurement by replacing $\La$ with the \wstar-extension $\ovl{\La}$ if necessary.

A family of measurements $([\Ga_x])_{\xin} $ in $\ME$ is called 
(in)compatible if the family $(\Ga_x)_\xin$ of representatives is (in)compatible.
Note that this definition does not depend on the choices of $\Ga_x .$
We define the sets of compatible and incompatible measurements in $\ME$ by
\begin{gather*}
	\McompXE
	:= \set{(\omega_x)_\xin \in \ME^X
	|
	\exists \nu \in \ME , \,  [ \omega_x \pp \nu \, (\forall \xin) ]
	} ,
	\\
	\MincompXE := \ME^X \setminus \McompXE ,
\end{gather*}
respectively.
\qed
\end{defi}
Before investigating the properties of the sets $\McompXE$ and $\MincompXE ,$
let us show that we can introduce a natural compact convex structure 
on the Cartesian power $\ME^X .$  

\begin{prop} \label{prop:dprod}
Let $X \nono$ and define the convex combination map on $\ME^X$ by
\[
	\braket{\la ; (\omega_x)_\xin , (\nu_x)_\xin}
	:=
	(\la \omega_x + \ola \nu_x)_\xin 
	\quad
	(\la \in [0,1] ; \, \omega_x , \nu_x \in \ME \,(\xin)) .
\]
Then the convex prestructure $(\ME^X , \braket{\cdot ; \cdot , \cdot})$
equipped with the product topology of the weak topology on $\ME$ is 
a compact convex structure.
\end{prop}
\begin{proof}
By Tychonoff\rq{}s theorem, the product topology of the weak topology is 
a compact Hausdorff topology.
Moreover the family of continuous affine functionals
\[
	\ME^X \ni (\omega_{x^\prime})_{x^\prime \in X}
	\mapsto 
	\Pg(\E ; \omega_x) \in \realn
	\quad
	(\xin , \, \E \colon \mathrm{ensemble})
\]
separates points of $\ME^X .$ Therefore 
$(\ME^X , \braket{\cdot ; \cdot , \cdot})$ is a compact convex structure.
\end{proof}
Thus we identify $\ME^X$ with the state space $S(\Ac (\ME^X))$
regularly embedded into $\Ac(\ME^X)^\ast .$
We also define the post-processing partial order $\pp$ on $\ME^X$
by the product order of $\pp$ on $\ME ,$ i.e.\
\[
	 (\omega_x)_\xin \pp (\nu_x)_\xin
	 \, :\defarrow
	 \,
	 [\omega_x \pp \nu_x \quad (\forall \xin) ].
\]
\begin{prop}\label{prop:compcc}
Let $X \nono .$
Then $\McompXE$ is a weakly compact, convex, lower subset of $\ME^X .$
\end{prop}
\begin{proof}
(Compactness).
Let $(\omega_x^i)_\xin $ $(\iin)$ be a net in $\McompXE .$
Then for each $\iin $ we take $\nu_i \in \ME$ such that
$\omega_x^i \pp \nu_i $ $(\xin).$
Since $\ME^X \times \ME$ is compact in the product topology of the weak topology,
there exist subnets $(\omega_x^\ikj)_\xin$ and $\nu_\ikj$
$(\jin)$ and elements $(\omega_x)_\xin \in \ME^X$
and $\nu \in \ME $ such that 
$\omega_x^\ikj \wto \omega_x$ $(\xin)$
and $\nu_\ikj \wto \nu .$
Moreover, since $\ME$ is a pospace, this implies $\omega_x \pp \nu$
$(\xin)$
and hence $(\omega_x)_\xin \in \McompXE ,$
which proves the compactness of $\McompXE .$

(Convexity).
Let $(\omega_x^1)_\xin , (\omega_x^2)_\xin \in \McompXE $
and take measurements $\nu_1 , \nu_2 \in \ME$
such that $\omega_x^j \pp \nu_j$
$(j=1,2; \, \xin) .$
Then for each $\la \in [0,1]$ we have
\[
	\la \omega_x^1 + \ola \omega_x^2 
	\pp
	\la \nu_1 + \ola \nu_2
	\quad (\xin) ,
\]
which implies $(\la \omega_x^1 + \ola \omega_x^2 )_\xin \in \McompXE .$

(Lower set condition).
Suppose $\ME^X \ni (\nu_x)_\xin \pp (\omega_x )_\xin \in \McompXE $
and take a measurement $\nu \in \ME$ satisfying $\omega_x\pp \nu$
$(\xin) .$ 
Then  $\nu_x\pp \nu$ $(\xin) $ and hence $(\nu_x)_\xin$ is also compatible.
\end{proof}

It is common to consider the (in)compatibility of finite family of measurements
(e.g.\ \cite{PhysRevLett.122.130402}).
The following proposition ensures that the (in)compatibility of 
arbitrary finite subfamilies sufficiently characterizes that of an infinite family of measurements.

\begin{prop} \label{prop:compf}
Let $X \nono .$
Then $(\omega_x)_\xin \in \ME^X$ is compatible if and only if
$(\omega_x)_{x \in A}$ is compatible for any finite subset $\varnothing \neq A \subset X .$
\end{prop}
\begin{proof}
Let us denote the set of non-empty finite subsets of $X$ by $\mathbb{F} (X),$
which is directed by the set inclusion $\subset .$ 
Assume that $(\omega_x)_{x \in A}$ is compatible for 
any $A \in \mathbb{F} (X) .$
Then for each $A\in \mathbb{F}(X)$ there exists $\nu_A \in \ME$ such that
$\omega_x \pp \nu_A$ $(x \in A) .$
By the compactness of $\ME ,$ there exists a subnet
$(\nu_{A(i)})_\iin $ of $(\nu_A)_{A \in \mathbb{F}(X)}$ 
weakly converging to some $\nu \in \ME .$
Then since $\omega_x \pp \nu_{A(i)}$ eventually for each $\xin ,$
the pospace property of $\ME$ implies $\omega_x \pp \nu $
for each $\xin ,$
which proves the \lq\lq{}if\rq\rq{} part of the claim.
The \lq\lq{}only if\rq\rq{} part is obvious.
\end{proof}

\subsection{Outperformance in the state discrimination task}
\label{subsec:incompout}
Let $(\omega_x)_\xin \in \ME^X $
and let $\fL : = \set{\omega_x \in \ME | \xin} .$
For each \wstar-family $\E$ and partitioned ensemble $\vecE = (\E_z)_{z\in Z}$
with $\E_z = (\vph_{z,y})_{y \in Y_x } ,$
we define
\begin{gather*}
	\Pg (\E ; (\omega_x)_\xin)
	:= \Pg (\E ; \fL) ,
	\\
	\Pg (\vecE ; (\omega_x)_\xin)
	:=
	\Pg (\vecE ; \fL) .
\end{gather*}
For a family $(\Ga_x)_\xin$ of measurements, we define
\begin{gather*}
	\Pg (\E ; (\Ga_x)_\xin)
	:= \Pg (\E ; ([\Ga_x])_\xin) =
	\Pg (\E ; \fL^\prime) ,
	\\
	\Pg (\vecE ; (\Ga_x)_\xin)
	:= \Pg (\vecE ; ([\Ga_x])_\xin) =
	\Pg (\vecE ; \fL^\prime),
\end{gather*}
where $\fL^\prime := \set{[\Ga_x] \in \ME |  \xin } .$
We also define 
\begin{align*}
	\Pgcomp (\vecE)
	&:=
	\sup_{\text{$Y$: set} ; \, (\nu_y)_{y\in Y } \in \Mcomp^Y (E)} 
	\Pg (\vecE ; (\nu_y)_{y\in Y } )
	\\
	&=
	\sup_{\nu \in \ME}
	\sum_{z\in Z} \Pg (\E_z ; \nu) ,
\end{align*}
where the second equality follows from the post-processing monotonicity of 
the gain functional.

The operational meaning of the quantity $\Pgcomp (\vecE)$ is as follows
\cite{PhysRevLett.122.130402,PhysRevLett.122.130403}.
Suppose that Alice prepares system\rq{}s state according to the the ensemble 
$(\vph_{z,y})_{z\in Z, \, y \in Y_z} .$
Then Bob performs the measurement on the system
and, \textit{after} the measurement, Alice informs Bob of the value of $z\in Z .$
This is in contrast 
Then based on the measurement outcome and the information on the label $z ,$
Bob guesses the label $y \in Y_z.$
The quantity $\Pgcomp (\vecE)$ is then the optimal probability 
that Bob\rq{}s guess coincides with the actual value.
This is in contrast to the operational setting for $\Pg (\E ; \fL)$ in which Bob is informed of the label $x \in X$ before he performs the measurement and can choose a proper measurement from $\fL$ to which incompatible measurements may belong.
We now show that an incompatible family of measurements outperforms 
in this state discrimination task for some partitioned ensemble,
generalizing the result for finite-dimensional quantum systems in \cite{PhysRevLett.122.130402}
(Theorem~2).

\begin{thm} \label{thm:incompdisc}
Let $X \nono$ and let $(\omega_x )_\xin \in \ME^X .$
Then $(\omega_x )_\xin \in \McompXE$ if and only if
\begin{equation}
	\Pg (\vecE ; (\omega_x)_{\xin}) \leq \Pgcomp (\vecE)
	\label{eq:outp}
\end{equation}
holds for all partitioned ensemble $\vecE .$
\end{thm}
For the proof of the theorem, we first consider the set of compatible 
EVMs.
For a family $(Y_x)_\xin $
of finite sets indexed by a finite set $X ,$ 
we define
\begin{gather*}
	\evm (\Yxin ; E) :=
	\prod_\xin \evm (Y_x ; E)
	\subset \prod_\xin E^{Y_x} ,
	\\
	\evmcomp (\Yxin ; E)
	:=
	\set{ (\oM_x)_\xin \in \evm (\Yxin ; E) | \text{$(\Ga^{\oM_x})_\xin$ is compatible}}.
\end{gather*}
A family $(\oM_x)_\xin$ EVMs is called (in)compatible
if the family $(\Ga^{\oM_x})_\xin$ of the associated measurements 
is (in)compatible.

\begin{lemm} \label{lemm:EVMcompc}
Let $X$ and $Y_x$ $(\xin)$ be non-empty finite sets.
Then $\evmcomp (\Yxin ; E)$ is a weakly$\ast$ compact
(i.e.\ $\sigma (\prod_\xin E^{Y_x} , \prod_\xin E_\ast^\Yx)$-compact)
convex subset of $\evm (\Yxin ; E) .$
\end{lemm}
\begin{proof}
(Compactness).
By Lemma~\ref{lemm:gainconti}.\ref{i:gainconti1}, the map
\begin{equation}
	\evm (\Yxin ;E)
	\ni (\oM_x)_\xin \mapsto ([\Ga^{\oM_x}])_\xin \in \ME^X
	\label{eq:compc}
\end{equation}
is continuous with respect to 
$\sigma (\prod_\xin E^{Y_x} , \prod_\xin E_\ast^\Yx)$
and the product topology of the weak topology on $\ME .$
Since $\evmcomp (\Yxin ; E)$ is the inverse image of the compact set
$\McompXE$ under the map \eqref{eq:compc}
and $\evm (\Yxin ; E)$ is compact, 
$\evmcomp (\Yxin ; E)$ is also compact.

(Convexity).
Let $\voM^j = (\oM^j_x)_\xin \in \evmcomp (\Yxin ; E)$
$(j=1,2) .$
Then by Proposition~\ref{prop:direct} for each $\la \in [0,1] ,$
\[
	([\Ga^{\la \oM^1_x + \ola \oM_x^2}])_\xin
	\pp
	(\la [\Ga^{\oM_x^1}] + \ola [\Ga^{\oM_x^2}])_\xin
	\in 
	\McompXE .
\]
Since $\McompXE$ is a lower set by Proposition~\ref{prop:compcc}, this implies 
\[(\la \oM^1_x + \ola \oM_x^2)_\xin \in \evmcomp (\Yxin ; E), \]
which proves the convexity.
\end{proof}

\begin{lemm} \label{lemm:Pgcomp}
Let $\vecE = (\E_x)_\xin$ be a partitioned ensemble with $\E_x = (\vph_{x,y})_{y \in Y_x}.$
Then
\begin{equation}
	\Pgcomp (\vecE)
	=
	\sup_{(\oN_x)_\xin \in \evmcomp (\Yxin ;E)}
	\sum_\xin \sum_{y \in \Yx}
	\braket{\vph_{x,y} , \oN_x(y)}.
	\label{eq:Pgcomp1}
\end{equation}
\end{lemm}
\begin{proof}
For any family $(\oN_x)_\xin \in \evmcomp (\Yxin ;E) $ of compatible EVMs, we have
\[
	\sum_\xin \sum_{y \in \Yx}
	\braket{\vph_{x,y} , \oN_x(y)}
	\leq 
	\sum_\xin \Pg (\E_x ; [\Ga^{\oN_x}])
	\leq 
	\Pgcomp (\vecE) ,
\]
which implies $(\mathrm{LHS})\geq (\mathrm{RHS})$ of \eqref{eq:Pgcomp1}.
Conversely for any $\omega \in \ME$ with the representative 
$\Ga \in \Chw (F \to E) ,$  for each $\xin $ we can take $\oM_x \in \evm (Y_x ; F)$ such that 
\[
	\Pg (\E_x ; \omega) 
	= \sum_{y \in Y_x} \braket{\vph_{x,y} , \Ga(\oM_x (y))}.
\]
Therefore we have
\[
	\Pg (\vecE ; \omega)
	= \sum_\xin \Pg (\E_x ; \omega)
	=\sum_\xin \sum_{y \in Y_x}
	 \braket{\vph_{x,y} , \Ga(\oM_x (y))} .
\]
Since $(\Ga \circ \oM_x)_\xin \in \evmcomp (\Yxin ; E) ,$
this implies $(\mathrm{LHS})\leq (\mathrm{RHS})$ of \eqref{eq:Pgcomp1}.
\end{proof}

\noindent
\textit{Proof of Theorem~\ref{thm:incompdisc}.}
The \lq\lq{}only if\rq\rq{} part of the claim is obvious from the definition of $\Pgcomp (\vecE) .$
To show the \lq\lq{}if\rq\rq{} part, we assume 
$(\omega_x)_\xin \in \MincompXE$ and find a partitioned ensemble
$\vecE$ violating \eqref{eq:outp}.
By Proposition~\ref{prop:compf} there exists a finite subset
$X_0 \subset X$ such that $(\omega_x)_{x \in X_0}\in \Mincomp^{X_0} (E) .$
If we can find a partitioned ensemble $\vecE $ such that 
$\Pg (\vecE ; (\omega_x)_{x \in X_0}) > \Pgcomp (\vecE ) ,$
then 
\[
	\Pg (\vecE ; (\omega_x)_{x \in X})
	\geq 
	\Pg (\vecE ; (\omega_x)_{x \in X_0})
	>
	\Pgcomp (\vecE ) .
\]
Therefore we may assume that $X$ is finite.

We first assume that every $\omega_x$ $(\xin)$ is finite-outcome
and $\omega_x = [\Ga^{\oM_x}]$ for some finite-outcome EVM 
$ \oM_x \in \evm (\Yx ; E) .$
Then since $(\oM_x)_\xin \notin \evmcomp (\Yxin ; E)$ 
and $\evmcomp (\Yxin ; E)$ is convex and weakly$\ast$ compact, 
the Hahn-Banach separation theorem implies that there exist 
weakly$\ast$ continuous linear functionals
$\vph_{x,y} \in E_\ast$
$(x \in X , \, y \in Y_x)$
such that
\[
	\sum_\xin \sum_{y \in \Yx}
	\braket{\vph_{x,y} , \oM_x(y)}
	>
	\sup_{(\oN_x)_\xin \in \evmcomp (\Yxin ;E)}
	\sum_\xin \sum_{y \in \Yx}
	\braket{\vph_{x,y} , \oN_x(y)}.
\]
Similarly as in Proposition~\ref{prop:gaineasy}, we can take 
$(\vph_{x,y})_{\xin , \, y \in \Yx}$ so that
$\vecE := (\E_x)_\xin$ with $\E_x := (\vph_{x,y})_{y\in \Yx}$
is a partitioned ensemble. 
Then
\begin{align*}
	\Pg (\vecE ; (\Ga^{\oM_x})_\xin) 
	&\geq
	\sum_\xin \Pg (\E_x ; \Ga^{\oM_x})
	\\
	&\geq
	\sum_\xin \sum_{y \in \Yx}
	\braket{\vph_{x,y} , \oM_x(y)}
	\\
	&>
	\sup_{(\oN_x)_\xin \in \evmcomp (\Yxin ;E)}
	\sum_\xin \sum_{y \in \Yx}
	\braket{\vph_{x,y} , \oN_x(y)}
	\\
	&=
	\Pgcomp (\vecE) ,
\end{align*}
where the last equality follows from Lemma~\ref{lemm:Pgcomp}.
Therefore $\vecE$ violates \eqref{eq:outp}.

We now consider general $(\omega_x)_\xin \in \MincompXE .$
Then by Theorem~\ref{thm:finapp} there exists a post-processing increasing
net $(\omega_x^i)_\xin$ $(\iin )$ of finite-outcome measurements such that 
$\omega_x^i \wto \sup_\iin \omega_x^i =  \omega_x $ for each $\xin .$
Since $\MincompXE$ is an open subset of $\ME^X $ by 
Proposition~\ref{prop:compcc},
there is some $\iin$ such that $(\omega_x^i)_\xin \in \MincompXE .$
Then from what we have shown in the last paragraph, there exists a
partitioned ensemble $\vecE $ such that 
$\Pg (\vecE ; (\omega_x^i)_\xin) > \Pgcomp (\vecE) .$
Since we have $\Pg (\vecE ; (\omega_x)_\xin) \geq \Pg (\vecE ; (\omega_x^i)_\xin)  $
by the monotonicity of $\Pg (\vecE ; \cdot) ,$ we can readily see that 
$\vecE$ violates \eqref{eq:outp}.
\qed
\begin{coro}\label{coro:incompdisc}
Let $X \nono$ and let $(\Ga_x)_\xin$ be a family of measurements.
Then $(\Ga_x)_\xin$ is compatible if and only if 
$\Pg (\vecE ; (\Ga_x)_\xin) \leq \Pgcomp (\vecE)$ for any partitioned ensemble $\vecE .$
\end{coro}
\begin{proof}
Let $\barG_x$ denote the \wstar-extension of $\Ga_x .$
Then by Proposition~\ref{prop:w*extch} it holds that 
$(\Ga_x)_\xin$ is compatible if and only if $(\ovl{\Ga}_x)_\xin$ is compatible.
Moreover we have $\Pg (\vecE ;(\Ga_x)_\xin )
= \Pg(\vecE ; (\barG_x)_\xin)$ by Lemma~\ref{lemm:exgain}.
Thus the claim immediately follows from Theorem~\ref{thm:incompdisc}.
\end{proof}

\subsection{Robustness of incompatibility} \label{subsec:RoI}

We now define the robustness of incompatibility
\cite{Haapasalo_2015,PhysRevLett.122.130403,PhysRevLett.122.130404,designolle2019incompatibility}.
\begin{defi}[Robustness of incompatibility] \label{def:RoI}
Let $X \nono$ and let $\Ga_x \in \Ch (F_x \to E)$
$(\xin)$ be measurements.
Then we define the \textit{robustness of incompatibility} by
\[
\begin{aligned}
\Rinc ((\Ga_x)_\xin) := \inf_{r , (\La_x)_\xin} \quad &  r
\\
\textrm{subject to} \quad 
& r \in [0,\infty) \\ 
&  (\La_x)_\xin \in \prod_\xin \Ch (F_x \to E) \\
& \left( \frac{\Ga_x + r \La_x}{1+r} \right)_\xin
\text{ is compatible,}
\end{aligned}
\]
which coincides with
\[
\begin{aligned}
\inf_{r , (\Psi_x)_\xin} \quad &  r
\\
\textrm{subject to} \quad 
& r \in [0,\infty) \\ 
&  (\Psi_x)_\xin \in \prod_\xin \Ch (F_x \to E) \\
&  (\Psi_x)_\xin\text{ is compatible} \\
&  \Ga_x \leq (1+r) \Psi_x \quad (\forall \xin) . 
\end{aligned}
\]
Here $\Rinc (\Gxin) := \infty$ if the feasible region is empty
\qed
\end{defi}
The robustness $\Rinc (\Gxin)$ quantifies the minimal amount of  
noise which should be added to the family $\Gxin$ of measurements 
to make it compatible.

We now prove the main result of this section that the robustness of incompatibility coincides with
the maximal relative increase in the state discrimination probability 
of a partitioned ensemble
compared to compatible measurements,
generalizing the result in \cite{PhysRevLett.122.130403,PhysRevLett.122.130404}
for finite-dimensional quantum systems.
\begin{thm} \label{thm:RoI}
In the setting of Definition~\ref{def:RoI}, the equality
\begin{equation}
	1  + \Rinc (\Gxin)
	= \sup_{\vecE \colon \mathrm{partitioned\, ensemble}}
	\frac{\Pg(\vecE ; \Gxin)}{\Pgcomp (\vecE)}
	\label{eq:RoI}
\end{equation}
holds, where the supremum is taken over all the partitioned ensembles.
\end{thm}
While the following proof of Theorem~\ref{thm:RoI} is almost parallel to 
those of Theorem~\ref{thm:RoU}
and the previous work \cite{PhysRevLett.122.130403},
we give it here for completeness.

We first establish some elementary properties of the 
set of compatible measurements with fixed outcome spaces and 
those of the robustness measure.

Let $\Fxin$ be a family of classical spaces.
We regard the product set $\prod_\xin \Ch (F_x \to E)$ as a compact convex set
by considering the direct product topology of the BW-topologies on
$\Ch (F_x \to E)$
and the convex operation
\[
	\la \Gxin + \ola \Lxin
	=(\la \Ga_x + \ola \La_x)_\xin
\]
$(\la \in [0,1];\, \Ga_x , \La_x \in \Ch(F_x \to E) \, (\xin) ) .$
We also denote by
$
	\Chcomp (\Fxin ;E) 
$
the set of compatible measurements in $\prod_\xin \Ch (F_x \to E) .$

\begin{lemm} \label{lemm:RoIe1}
Let $X\nono,$ let $\Fxin$ be a family of classical spaces,
and let $\Gxin \in \prod_\xin \Ch (F_x \to E).$
\begin{enumerate}[1.]
\item \label{i:e1}
The set $\Chcomp (\Fxin ;E)$ is a compact convex subset of 
$\prod_\xin \Ch (F_x\to E) .$
\item \label{i:e2}
If $X$ is finite, then $r : = \Rinc ((\Ga_x)_\xin) < \infty .$ 
\item \label{i:e99}
If $\Rinc (\Gxin) < \infty ,$
then there exists a compatible family $(\Psi_x)_\xin \in \Chcomp (\Fxin ; E)$ 
such that $\Ga_x \leq (1+r) \Psi_x$ $(\forall \xin) .$
\item \label{i:e3}
$\Rinc ((\Ga_x)_{x\in Y}) \leq \Rinc (\Gxin)$ for any subset
$\onon Y \subset X .$
\item \label{i:e4}
$\Rinc (\Gxin) = \sup_{A \in \FX} \Rinc ((\Ga_x)_{x\in A}) ,$
where $\FX$ denotes the set of non-empty finite subsets of $X .$
\end{enumerate}
\end{lemm}
\begin{proof}
\begin{enumerate}[1.]
\item
(Compactness).
Let $(\Ga^i_x)_{x\in X}$ $(i \in I)$ be a net in $\Chcomp (\Fxin ;E)$
such that $\Ga^i_x \bwto \Ga_x \in \Ch (F_x \to E)$
$(\xin) .$
Then by the compatibility and Theorem~\ref{thm:bss}, for each $\iin$ 
there exists a measurement $\nu_i \in \ME$ such that
\[
	\Pg (\E ; \Ga_x^i) \leq \Pg (\E ; \nu_i) 
	\quad
	(\xin)
\]
for any ensemble $\E . $ We take a subnet $(\nu_\ikj)_\jin$
weakly converging to some $\nu \in \ME .$
Then for any $\xin,$ any ensemble $\E = (\vph_z)_{z\in Z},$
and any  EVM $\oM \in \evm (Z ; F_x) ,$ we have
\begin{align*}
	\sum_{z\in Z} \braket{\vph_z , \Ga_x (\oM(z))}
	&=
	\lim_\jin \sum_{z\in Z}
	\braket{\vph_z , \Ga_x^\ikj (\oM (z))}
	\\
	&\leq 
	\limsup_\jin \Pg (\E ; \Ga_x^\ikj)
	\\
	&\leq \limsup_\jin \Pg (\E ; \nu_\ikj)
	\\
	&=
	\Pg (\E ; \nu) .
\end{align*}
By taking the supremum of $\oM ,$ we obtain $\Pg (\E ; \Ga_x ) \leq \Pg (\E ; \nu) .$
Since the ensemble $\E$ is arbitrary, Theorem~\ref{thm:prebss} implies 
$[\Ga_x] \pp \nu .$
Therefore $\Gxin$ is compatible, which proves the compactness of 
$\Chcomp (\Fxin ;E) .$

(Convexity). 
The convexity can be shown analogously as in Lemma~\ref{lemm:EVMcompc} by 
using Propositions~\ref{prop:direct} and \ref{prop:compcc}.
\item
Define $\Psi_x^\prime \in \Ch (F_x \to E)$ $(\xin)$ by
\[
	\Psi_x^\prime (a) :=
	\frac{\Ga_x (a) + (\abs{X}-1) \phi_x (a) u_E }{\abs{X}} 
	\quad (a \in F_x),
\]
where $\phi_x \in S(F_x)$ is a fixed state.
We show that $(\Psi_x^\prime)_\xin$ is compatible, from which 
$\Rinc (\Gxin) \leq  \abs{X} -1 < \infty$ follows.
Define $\Phi \in \Ch (\bigoplus_\xin F_x \to E)$ and 
$\Theta_x \in \Ch (F_x \to \bigoplus_{x^\prime \in X}F_{x^\prime})$ $(\xin)$
by
\begin{gather*}
	\Phi ((a_x)_\xin) := \abs{X}^{-1} \sum_{\xin} \Ga_x (a_x)
	\quad
	((a_x)_\xin \in \bigoplus_\xin F_x) ,
	\\
	\Theta_x (a) 
	:=
	a \oplus \bigoplus_{x^\prime \in X \setminus \{ x\} }
	\phi_x (a) u_{F_{x^\prime}}
	\quad
	(a \in F_x) .
\end{gather*}
Then we have $\Psi_x^\prime = \Phi \circ \Theta_x \pp \Phi $
$(\xin).$
Hence $(\Psi_x^\prime)_\xin$ is compatible.

\item
Since $r= \Rinc (\Gxin) < \infty ,$ there exists a sequence
$(r_n , (\Psi_x^n)_\xin)_{n \in \natn}$ such that $r_n \downarrow r ,$
$\Ga_x \leq (1+r_n) \Psi_x^n ,$
and $(\Psi_x^n)_\xin \in \Chcomp (\Fxin ; E) $
$(\xin ; \, n\in \natn) .$
By the compactness of $\Chcomp (\Fxin ; E)$ there exists a subnet
$(\Psi_x^{n(i)})_\xin $ $(\iin)$ converging to some 
$(\Psi_x)_\xin \in \Chcomp (\Fxin ;E) .$
Then we have $\Ga_x \leq (1+r) \Psi_x$ $(\xin),$ 
which proves the claim.
\item
For simplicity we write as $r_A : = \Rinc ((\Ga_x)_{x \in A})$ for each subset
$A \subset X .$
Without loss of generality we may assume $r_X < \infty .$
Then by the claim~\ref{i:e99} there exists a compatible family
$(\Psi_x)_\xin \in \Chcomp (\Fxin ;E) $
such that $\Ga_x \leq (1+r_X ) \Psi_x$
$(\forall \xin) .$
Then 
$\Ga_x \leq (1+r_X ) \Psi_x$
$(\forall x \in Y) $ 
and,
from the definition of the robustness of incompatibility,
this implies $r_Y \leq r_X .$
\item
From the claim~\ref{i:e3}, the net $(r_A)_{A\in \FX}$ is increasing and 
upper bounded by $r_X .$
Thus we have only to show that 
$s \leq r_A $ holds eventually for any $s < r_X .$
Suppose not.
Then there exist $s < r_X$ and a subnet $(r_{A(i)})_\iin$ such that 
$r_{A(i)} < s$ for all $\iin .$
For each $\iin $ we take a compatible family $(\Psi^i_x)_{x\in A} 
\in \Chcomp ((F_x)_{x\in A} ;E)$ such that
$\Ga_x \leq (1+s) \Psi_x^i $ $(x \in A(i)) .$
We also define $\Psi_x^i \in \Ch (F_x \to E)$ for $x \in X \setminus A(i)$
by 
\[
	\Psi_x^i (a) := \psi_x( a) u_E \quad (a\in F_x)
\]
for some fixed state $\psi_x \in S(F_x) .$
Then since $\Psi_x^i \pp \Lambda$ for any $x\in X \setminus A(i)$
and any measurement $\La ,$ the family $(\Psi_x^i)_\xin$ is compatible.
Then from the compactness of $\Chcomp (\Fxin ;E)$
it follows that there exists a compatible family
$(\Psi_x)_\xin 
\in \Chcomp (\Fxin ;E)
$
to which a subnet of $(\Psi_x^i)_\xin$ $(\iin)$ converges.
Since $\Ga_x \leq (1+s) \Psi_x^i $ eventually for each $\xin ,$ we have $\Ga_x \leq (1+s) \Psi_x$ $(\xin) .$
This implies $r_X \leq s ,$
which contradicts the assumption $s < r_X .$
\qedhere
\end{enumerate}
\end{proof}

\begin{lemm} \label{lemm:RoIe2}
Let $X\nono $ and let $F_x^j$ $(j=1,2 ; \, \xin)$ be classical spaces.
Then for any families $(\Ga^j_x)_\xin \in \prod_\xin \Ch(F_x^j \to E)$
$(j=1,2)$
of measurements, $\Ga^1_x \pp \Ga^2_x $
$(\xin)$ implies $\Rinc ((\Ga_x^1)_\xin) \leq \Rinc ((\Ga_x^2)_\xin) .$
\end{lemm}
\begin{proof}
We write as $r_j := \Rinc ((\Ga_x^j)_\xin) .$
Without loss of generality we may assume $r_2 < \infty .$
Then by Lemma~\ref{lemm:RoIe1}, there exists a compatible family
$(\Psi^2_x)_\xin \in \Chcomp ((F_x^2)_\xin ; E)$ 
such that 
\begin{equation}
	\Ga_x^2 \leq (1+r_2) \Psi_x^2 \quad (\xin) .
	\label{eq:RoIe2}
\end{equation}
By assumption there are channels $\Phi_x \in \Ch (F_x^1 \to F_x^2)$
$(\xin)$ such that $\Ga_x^1 = \Ga_x^2 \circ \Phi_x. $
Then \eqref{eq:RoIe2} implies
\[
	\Ga_x^1 = \Ga_x^2 \circ \Phi_x
	\leq (1+r_2) \Psi_x^2 \circ \Phi_x .
\]
Since $(\Psi_x^2 \circ \Phi_x)_\xin$ is compatible, this implies $r_1 \leq r_2 .$
\end{proof}

\begin{lemm} \label{lemm:RoIe3}
Let $X \nono ,$
let $\Ga_x \in \Ch (F_x \to E)$ $(\xin)$ be measurements, 
and let $\barG_x \in \Chw (F_x^\aast \to E)$ be the \wstar-extension of
$\Ga_x .$ Then 
\[\Rinc (\Gxin) = \Rinc ((\barG_x)_\xin) . \]
\end{lemm}
\begin{proof}
By Lemma~\ref{lemm:RoIe2} we have 
$\Rinc (\Gxin) \leq \Rinc ((\barG_x)_\xin) . $
We prove the converse inequality.
Without loss of generality we may assume $\Rinc (\Gxin) < \infty .$
Then by Lemma~\ref{lemm:RoIe1} we can take a compatible 
family $(\Psi_x)_\xin \in \Chcomp (\Fxin ;E)$ such that
\[
\Ga_x \leq (1+\Rinc (\Gxin)) \Psi_x 
\quad (\xin) .
\]
Let $\barPsi_x \in \Chw (F_x^\aast \to E)$ be the \wstar-extension
of $\Psi_x$ $(\xin).$
Then 
\[(\barPsi_x)_\xin \in \Chcomp ((F_x^\aast)_\xin ; E) . \]
Furthermore, since $E_+^\aast$ is closed in the weak$\ast$ topology $\sigma (E^\aast ,E^\ast),$
we have 
\[ \barG_x \leq (1+\Rinc (\Gxin))  \barPsi_x , \]
which implies 
$\Rinc (\Gxin) \geq \Rinc ((\barG_x)_\xin) .$
\end{proof}

We now show $\mathrm{(LHS)} \geq  \mathrm{(RHS)} $
in \eqref{eq:RoI}.
\begin{lemm}\label{lemm:RoIineq}
In the setting of Theorem~\ref{thm:RoI}, the inequality
\begin{equation}
	1  + \Rinc (\Gxin)
	\geq \sup_{\vecE \colon \mathrm{partitioned\, ensemble}}
	\frac{\Pg(\vecE ; \Gxin)}{\Pgcomp (\vecE)}
	\label{eq:RoIineq}
\end{equation}
holds.
\end{lemm}
\begin{proof}
Without loss of generality, we may assume $\Rinc (\Gxin) < \infty .$
Then by Lemma~\ref{lemm:RoIe1} there exists a compatible family
$(\Psi_x)_\xin \in \Chcomp (\Fxin ;E)$ such that
\[
	\Ga_x \leq (1  + \Rinc (\Gxin)) \Psi_x \quad (\xin).
\]
Take an arbitrary partitioned ensemble $\vecE = (\E_y)_{y\in Y}$
with 
$\E_y = (\vph_{y,z})_{z\in Z_y} .$
Then for any EVMs
$\oM_{x,y} \in \evm (Z_y ;F_x)$
$(x\in X , y\in Y)$
we have
\begin{align*}
	&\sup_\xin \sum_\yin \sum_{z\in Z_y}
	\braket{\vph_{y,z} , \Ga_x (\oM_{x,y} (z))}
	\\
	&\leq
	(1  + \Rinc (\Gxin))
	\sup_\xin
	\sum_\yin \sum_{z\in Z_y}
	\braket{\vph_{y,z} , \Psi_x (\oM_{x,y} (z))}
	\\
	&\leq
	(1  + \Rinc (\Gxin))
	\sup_\xin \Pg (\vecE ; \Psi_x)
	\\
	&\leq
	(1  + \Rinc (\Gxin))
	\Pgcomp (\vecE) .
\end{align*}
By taking the supremum of $\oM_{x,y} ,$ we obtain
\[
	\Pg (\vecE ; \Gxin) \leq (1  + \Rinc (\Gxin))
	\Pgcomp (\vecE),
\]
from which \eqref{eq:RoIineq} follows.
\end{proof}

We now prove the theorem when $X$ is finite and each 
$\Ga_x$ is finite-outcome.

\begin{lemm}[cf.\ \cite{PhysRevLett.122.130403}] \label{lemm:RoIfin}
The statement of Theorem~\ref{thm:RoI} is true when $\abs{X} < \infty$
and $\Ga_x = \Ga^{\oM_x}$ for some finite-outcome EVM
$\oM_x \in \evm (Y_x ; E)$ $(\xin) .$
\end{lemm}
\begin{proof}
By the one-to-one correspondence between $\Ch (\linf (Y_x) \to E)$
and $\evm (Y_x ;E) ,$ the robustness measure $\Rinc ((\Ga^{\oM_x})_\xin)$
can be written as
\[
\begin{aligned}
\Rinc ((\Ga^{\oM_x})_\xin)
= \inf_{r , (\oN_x)_\xin} \quad &  r
\\
\textrm{subject to} \quad 
& r \in [0,\infty), \quad  
(\oN_x)_\xin \in \evmcomp (\Yxin ;E) \\ 
& \oM_x (y) \leq (1+r) \oN_x (y) \quad (\xin , y \in Y_x) .
\end{aligned}
\]
Define
\[
	\cK := \set{(\la \oN_x)_\xin | \la \in [0,\infty), \, 
	( \oN_x)_\xin \in \evmcomp (\Yxin ; E) }.
\]
It can be shown similarly as in Lemma~\ref{lemm:RoUfin} that
$\cK$ is a weakly$\ast$ closed convex cone in $\prod_\xin E^{Y_x} .$
Then we have
\begin{equation}
\begin{aligned}
1+ \Rinc ((\Ga^{\oM_x})_\xin)
= \inf_{s , (\oN_x)_\xin} \quad &  s
\\
\textrm{subject to} \quad 
& s \in \realn, \quad  
(\oN_x)_\xin \in \cK \\
& \sum_{y \in Y_x} \oN_x(y) \leq s u_E \quad (\xin)  \\
& \oM_x (y) \leq   \oN_x (y) \quad (\xin , y \in Y_x) .
\end{aligned}
\label{eq:RoIpr}
\end{equation}
The optimization problem~\eqref{eq:RoIpr} can be written in the 
standard form \eqref{eq:coneprgrm} of the conic programming by putting
\begin{gather*}
	V:= \left( \prod_\xin E^{Y_x} \right) \times \realn
	,\quad
	U:= \left( \prod_\xin E^{Y_x} \right) \times E^X ,
	\\
	C:= \cK \times \realn ,\quad 
	K := 
	\left( \prod_\xin (E_+)^{Y_x} \right) \times( E_+)^X  ,
	\\
	\braket{c^\ast , (w,s)} := s \quad
	((w,s) \in V) ,\\
	b := ((-\oM_x)_\xin , (0)_\xin) \in U ,\\
	A\colon 
	V \ni ( (w_{x,y})_{x\in X , y\in Y_x} , s )
	\mapsto 
	\left((w_{x,y})_{x\in X , y\in Y_x} , \,
	\Bigl(s u_E - \sum_{y\in Y_x} w_{x,y} \Bigr)_\xin
	\right)
	\in U,
\end{gather*}
where
\[
	(w_{x,y})_{x\in X , y\in Y_x}
	:= \left(  \left(   w_{x,y}\right)_{y \in Y_x} \right)_\xin .
\]
The convex cones $C$ and $K$ are weakly$\ast$ closed in $V$ and $U,$ respectively.
Let $v_0 := ( (\abs{Y_x}^{-1}u_X)_{\xin,y\in Y_x} , 2) \in V. $
Since the family $\left((\abs{Y_x}^{-1} u_E)_{y \in Y_x}\right)_{\xin}$
of trivial observables is compatible,
we have $v_0 \in C .$ Moreover
\begin{align*}
	U &\supset A(C) -K
	\\
	&\supset \realn_+ A(v_0) - K
	\\
	&=\{ \, ( (\la \abs{Y_x}^{-1} u_E - w_{x,y})_{x\in X , y\in Y_x} , 
	(\la u_E - w^\prime_x)_\xin   ) \, |\\
	&\quad
	\la \in \realn_+ , \, w_{x,y} , w_x^\prime \in E_+ \, (\xin , y \in Y_x )
	\}
	\\
	&= U ,
\end{align*}
which implies $-b \in \interi (A(C)-K) (= U) .$
Therefore the optimal value of \eqref{eq:RoIpr} coincides with its dual problem
\eqref{eq:RoUd1} with 
\begin{gather*}
	K^\ast = \left( \prod_\xin ( E_+^{\ast})^{Y_x} \right) \times (E_+^\ast)^X ,
	\quad
	C^\ast = \cK^\ast \times \{ 0 \} ,
	\\
	\cK^\ast =
	\set{
	(\omega_{x,y})_{x\in X , y\in Y_x} \in \prod_\xin (E^\ast)^{Y_x} | 
	\sum_{x\in X , y\in Y_x} \braket{\omega_{x,y} , G_{x,y}} \geq 0
	\, 
	(\forall (G_{x,y})_{x\in X , y\in Y_x} \in \cK)
	} ,
	\\
	A^\ast ((\psi_{x,y})_{x\in X , y\in Y_x} , (\chi_x)_\xin) 
	=
	\left( 
	(\psi_{x,y} - \chi_x)_{x\in X , y\in Y_x} , \,
	\sum_\xin \braket{\chi_x , u_E}
	\right)
	\quad
	(\psi_{x,y} , \chi_x \in E^\ast) .
\end{gather*}
Therefore the dual problem can be written as 
\begin{equation}
\begin{aligned}
1+ \Rinc ((\Ga^{\oM_x})_\xin)
= \sup_{(\psi_{x,y})_{x\in X , y\in Y_x} , (\chi_x)_\xin} \quad &  
\sum_{\xin , y\in Y_x} \braket{\psi_{x,y } , \oM_x (y)}
\\
\textrm{subject to} \quad 
& \psi_{x,y} , \chi_x \in E^\ast_+ \quad (\xin , y\in Y_x) \\
& (\chi_x - \psi_{x,y})_{\xin , y\in Y_x} \in \cK^\ast ,
\\
&\sum_\xin \braket{\chi_x , u_E} =1 ,
\end{aligned}
\notag
\end{equation}
which coincides with
\begin{equation}
\begin{aligned}
1+ \Rinc ((\Ga^{\oM_x})_\xin)
= \sup_{(\psi_{x,y})_{x\in X , y\in Y_x} , (\chi_x)_\xin} \quad &  
\sum_{\xin , y\in Y_x} \braket{\psi_{x,y } , \oM_x (y)}
\\
\textrm{subject to} \quad 
& \psi_{x,y} , \chi_x \in E^\ast_+ \quad (\xin , y\in Y_x) \\
& (\chi_x - \psi_{x,y})_{\xin , y\in Y_x} \in \cK^\ast ,
\\
&\sum_\xin \braket{\chi_x , u_E}  \leq 1 .
\end{aligned}
\label{eq:RoId2}
\end{equation}

We next show that the feasible region of \eqref{eq:RoId2}
can be restricted to the weakly$\ast$ functionals, i.e.\
\begin{equation}
\begin{aligned}
1+ \Rinc ((\Ga^{\oM_x})_\xin)
= \sup_{(\psi_{x,y})_{x\in X , y\in Y_x} , (\chi_x)_\xin} \quad &  
\sum_{\xin , y\in Y_x} \braket{\psi_{x,y } , \oM_x (y)}
\\
\textrm{subject to} \quad 
& \psi_{x,y} , \chi_x \in E_{\ast +} \quad (\xin , y\in Y_x) \\
& (\chi_x - \psi_{x,y})_{\xin , y\in Y_x} \in \cK^\ast ,
\\
&\sum_\xin \braket{\chi_x , u_E}  \leq 1 .
\end{aligned}
\label{eq:RoId3}
\end{equation}
For this we have only to show that the feasible region of \eqref{eq:RoId3}
is weakly$\ast$ dense in that of \eqref{eq:RoId2}.
An element $((\psi_{x,y})_{\xin , y\in Y_x} , (\chi_x)_\xin) \in 
\left( \prod_\xin (E^\ast)^{Y_x} \right) \times (E^\ast)^X  $
is in the feasible region of \eqref{eq:RoId2} if and only if
\begin{align*}
	-1  &\leq 
	\sum_{\xin , y\in Y_x}
	\braket{\psi_{x,y} , a_{x,y}} + \sum_\xin \braket{\chi_x , b_x}
	+ \sum_{\xin , y\in Y_x } \braket{\chi_x -\psi_{x,y} , G_{x,y}}
	-\sum_{\xin} \braket{\chi_x , u_E}
	\\
	&=
	\sum_{\xin , y\in Y_x}
	\braket{\psi_{x,y} , a_{x,y} - G_{x,y}} 
	+ \sum_\xin \braket{\chi_x , b_x + \sum_{y \in Y_x} G_{x,y} - u_E}
	\\
	& \quad
	(\forall a_{x,y} , b_x \in E_+ ; \, \forall (G_{x,y})_{x\in X , y\in Y_x} \in \cK) .
\end{align*}
Therefore if we define 
\begin{align*}
\mathcal{L} & :=
\{ \,
\Bigl((a_{x,y} - G_{x,y})_{\xin , y\in Y_x} , 
( b_x + \sum_{y \in Y_x} G_{x,y} - u_E)_\xin \Bigr) | \\
&\quad
 a_{x,y} , b_x \in E_+  \, (x\in X , y\in Y_x)  ; \,  (G_{x,y})_{x\in X , y\in Y_x} \in \cK
\, \} ,
\end{align*}
then $\mathcal{L}$ is a convex subset of 
$\left( \prod_{\xin}E^{Y_x}\right) \times
E^X$
containing the origin and 
the polar of $\mathcal{L}$ in the pair 
$( 
\left( \prod_{\xin}E^{Y_x}\right) \times
E^X
,  \left( \prod_{\xin}(E^\ast)^{Y_x}\right) \times
(E^\ast)^X)$ 
coincides with the feasible region of \eqref{eq:RoId2}.
Similarly the polar of $\mathcal{L}$ in the pair 
$( 
\left( \prod_{\xin}E^{Y_x}\right) \times
E^X
,  \left( \prod_{\xin}(E_\ast)^{Y_x}\right) \times
(E_\ast)^X)$ 
coincides with the feasible region of \eqref{eq:RoId3}.
Thus by the bipolar theorem and Krein-\v{S}mulian theorem, 
it suffices to show that $(\mathcal{L})_r$
is weakly$\ast$ closed for any $r \in (0,\infty) .$ 
Suppose that the element
\[
	\Bigl((a_{x,y} - G_{x,y})_{\xin , y\in Y_x} , 
( b_x + \sum_{y \in Y_x} G_{x,y} - u_E)_\xin \Bigr)
\]
with
\[
	 a_{x,y} , b_x \in E_+  \, (x\in X , y\in Y_x)  ; \,  (G_{x,y})_{x\in X , y\in Y_x} \in \cK
\]
is in $(\mathcal{L})_r .$ 
Then from $b_x , G_{x,y} \geq 0$ we obtain
\begin{gather*}
	\| b_x \| , \|G_{x,y} \|
	\leq \left\| b_x + \sum_{y^\prime \in Y_x} G_{x,y^\prime}\right\|
	\leq \left\| b_x + \sum_{y^\prime \in Y_x} G_{x,y^\prime}   - u_E\right\| +1
	\leq r+1 ,
	\\
	\| a_{x,y} \|
	\leq \| a_{x,y} -G_{x,y} \| + \| G_{x,y} \| \leq 2r+1 
\end{gather*}
$(\xin , y \in Y_x) .$
Thus by using the Banach-Alaoglu theorem, the weak$\ast$ closedness of 
$(\mathcal{L})_r$ follows similarly as in Lemma~\ref{lemm:dense}.
Therefore we have shown \eqref{eq:RoId3}.

Now from \eqref{eq:RoId3} there exists a sequence 
$\left(
(\psi^k_{x,y})_{\xin , y\in Y_x} , 
(\chi_x^k )_\xin
\right)$
$(k\in \natn)$
in the feasible region of \eqref{eq:RoId3} such that
\[
	\sum_{\xin , y\in Y_x} \braket{\psi_{x,y}^k , \oM_x(y)}
	> 1 + \Rinc ((\Ga^{\oM_x})_\xin) - \frac{1}{k} .
\]
Let $N_k := \sum_{\xin , y\in Y_x} \braket{\psi_{x,y}^k , u_E}  ,$
which is $>0$ by the above inequality,
and define a partitioned ensemble $\vecE^k = (\E_x^k)_\xin$ by
\[
	\E_x^k := (\vph_{x,y}^k)_{y \in Y_x},
	\quad
	\vph_{x,y}^k := N_k^{-1} \psi_{x,y}^k .
\]
Then since $(\chi^k_x - \psi_{x,y}^k)_{\xin , y\in Y_x} \in \cK^\ast ,$
for any $(\oN_x)_\xin \in \evmcomp (\Yxin ; E)$ we have
\begin{align*}
	0 & \leq \sum_{\xin , y\in Y_x} \braket{\chi^k_x - \psi_{x,y}^k , \oN_x (y)}
	\\
	&= \sum_\xin \braket{\chi^k_x , u_E}
	- \sum_{\xin , y\in Y_x} \braket{ \psi_{x,y}^k , \oN_x (y)}
	\\
	&\leq 
	1 -  \sum_{\xin , y\in Y_x} \braket{ \psi_{x,y}^k , \oN_x (y)}
\end{align*}
and therefore
\[
	\sum_{\xin , y\in Y_x} \braket{ \vph_{x,y}^k , \oN_x (y)}
	\leq N_k^{-1} .
\]
By taking the supremum of $\oN_x$ we obtain
\[
	\Pgcomp (\vecE^k)
	\leq N_k^{-1} .
\]
Thus
\begin{align*}
	\Pg (\vecE^k ; (\Ga^{\oM_x})_\xin)
	&= \sum_\xin \max_{x^\prime \in X} \Pg(\E_x^k ; \Ga^{\oM_{x^\prime}})
	\\
	&\geq
	\sum_\xin \Pg(\E_x^k ; \Ga^{\oM_x})
	\\
	& \geq 
	N_k^{-1} \sum_{\xin , y\in Y_x} 
	\braket{ \psi_{x,y}^k , \oM_x (y)}
	\\
	&> N_k^{-1} \left(
	1 + \Rinc ((\Ga^{\oM_x})_\xin) - \frac{1}{k}
	\right)
	\\
	& \geq 
	\Pgcomp (\vecE^k)
	\left(
	1 + \Rinc ((\Ga^{\oM_x})_\xin) - \frac{1}{k}
	\right) 
\end{align*}
and hence
\[
	1 + \Rinc ((\Ga^{\oM_x})_\xin)
	\leq \sup_{k \in \natn}
	\frac{\Pg(\vecE^k ;(\Ga^{\oM_x})_\xin )}{\Pgcomp (\vecE^k)}
	\leq
	\sup_{\vecE \colon \mathrm{partitioned \, ensemble}}
	\frac{\Pg(\vecE ;(\Ga^{\oM_x})_\xin )}{\Pgcomp (\vecE )}.
\]
By combining this with Lemma~\ref{lemm:RoIineq} 
we obtain \eqref{eq:RoI}.
\end{proof}

We next consider general measurements.

\begin{lemm} \label{lemm:RoIlsc}
Let $X \nono$ and let $F_x $ $(\xin)$ be classical spaces.
Then the extended real valued function
\begin{equation}
	\prod_\xin \Ch (F_x \to E) 
	\ni 
	\Gxin \mapsto 
	\Rinc (\Gxin)
	\in [0,\infty]
	\label{eq:RoImap}
\end{equation}
is lower semicontinuous with respect to the product topology of the BW
topologies on $\Ch (F_x \to E) .$
\end{lemm}
\begin{proof}
Suppose that \eqref{eq:RoImap} is not lower semicontinuous.
Then there exist a net $(\Ga_x^i)_\xin$ $(\iin)$ in $\prod_\xin \Ch (F_x \to E)$
BW-convergent to some $\Gxin \in \prod_\xin \Ch (F_x \to E) $
and $r  \in [ 0,  \Rinc (\Gxin)  )$ such that
$\Rinc ((\Ga^i_x)_\xin) < r$ for all $\iin .$
Then 
for each $\iin$ there exists a compatible family 
$(\Psi^i_x)_\xin \in \Chcomp (\Fxin ; E)$
such that 
\[\Ga_x^i \leq (1+r)\Psi_x^i  \quad (\xin) . \]
By the compactness of $\Chcomp (\Fxin ;E) ,$
there exists a subnet of $(\Psi^{j(i)}_x)_\xin$ $(\jin)$
BW-converging to some $(\Psi_x)_\xin \in \Chcomp (\Fxin ;E) .$
Then we have $\Ga_x \leq (1+r) \Psi_x$ $(\xin) ,$
which contradicts $r < \Rinc (\Gxin). $
Therefore \eqref{eq:RoImap} is lower semicontinuous.
\end{proof}

\begin{lemm} \label{lemm:Pgvec}
Let $(\Ga_x)_\xin$ be a non-empty family of measurements on $E$ 
and let$\vecE = (\E_y)_\yin$ be a partitioned ensemble.
Then 
\begin{equation}
	\Pg (\vecE ; (\Ga_x)_\xin)
	= \sup_{A \in \FX} \Pg (\vecE ; (\Ga_x)_{x\in A }),
	\notag
\end{equation}
where $\FX$ denotes the set of non-empty finite subsets of $X .$
\end{lemm}
\begin{proof}
The claim is immediate from 
\[
	\Pg (\vecE ; (\Ga_x)_\xin)
	=
	\sum_{y \in Y} \sup_\xin  \Pg (\E_y ; \Ga_x)
\]
and a similar expression for $\Pg (\vecE ; (\Ga_x)_{x \in A})$ $(A \in \FX) .$
\end{proof}

\noindent 
\textit{Proof of Theorem~\ref{thm:RoI}.}
By Lemmas~\ref{lemm:exgain} and \ref{lemm:RoIe3},
we have only to prove \eqref{eq:RoI} when for each $\xin $ $\Ga_x$ is a \wstar-measurement.
Then by Theorem~\ref{thm:finapp} there exists a net 
$(\Ga_x^i)_\xin$ $(\iin = \prod_\xin \mathcal{D} (F_x))$
in $\Ch ((F_x)_\xin ;E)$
satisfying the following conditions:
\begin{enumerate}[(i)]
\item
$\Ga_x^i \bwto \Ga_x$ $(\xin) .$
\item
Each $[\Ga^i_x]$ is finite-outcome $(\iin , \xin).$
\item
$([\Ga_x^i])_\iin$ is an increasing net in $\ME$ weakly converging to
$\sup_\iin [\Ga_x^i] = [\Ga_x]$
for each $\xin .$
\end{enumerate}
From Lemmas~\ref{lemm:RoIe1}, \ref{lemm:RoIe2} and \ref{lemm:RoIlsc} we have
\begin{equation}
	1+ \Rinc ( (\Ga_x^i)_\xin  )
	=
	\sup_{A \in \FX}
	\sup_\iin 
	\left(
	1 + \Rinc (  (\Ga_x^i)_{x \in A} )
	\right) ,
	\label{eq:last1}
\end{equation}
where $\FX$ denotes the set of non-empty finite subsets of $X.$
From Lemma~\ref{lemm:RoIfin} we also have
\begin{equation}
	1 + \Rinc (  (\Ga_x^i)_{x \in A} )
	=\sup_{\vecE \colon \mathrm{partitioned \, ensemble}}
	\frac{\Pg(\vecE ;  (\Ga_x^i)_{x \in A} )}{\Pgcomp (\vecE)} 
	\quad
	(\iin , A \in \FX).
	\label{eq:last2}
\end{equation}
Then from \eqref{eq:last1} and \eqref{eq:last2} we have
\begin{align*}
	1+ \Rinc ( (\Ga_x^i)_\xin  )
	&=
	\sup_{A \in \FX}
	\sup_\iin 
	\sup_{\vecE \colon \mathrm{partitioned \, ensemble}}
	\frac{\Pg(\vecE ;  (\Ga_x^i)_{x \in A} )}{\Pgcomp (\vecE)} 
	\\
	&=
	\sup_{\vecE \colon \mathrm{partitioned \, ensemble}}
	\sup_{A \in \FX}
	\sup_\iin 
	\frac{\Pg(\vecE ;  (\Ga_x^i)_{x \in A} )}{\Pgcomp (\vecE)} 
	\\
	&=
	\sup_{\vecE \colon \mathrm{partitioned \, ensemble}}
	\sup_{A \in \FX}
	\frac{\Pg(\vecE ;  (\Ga_x)_{x \in A} )}{\Pgcomp (\vecE)} 
	\\
	&=
	\sup_{\vecE \colon \mathrm{partitioned \, ensemble}}
	\frac{\Pg(\vecE ;  (\Ga_x)_{x \in X} )}{\Pgcomp (\vecE)} ,
\end{align*}
where in the third equality we used the fact that
for each partitioned ensemble $\vecE =(\E_y)_{y \in Y}$ 
and $A\in \FX ,$ the map 
\[
	\ME^A \ni (\omega_x)_{x\in A}
	\mapsto 
	\Pg(\vecE ; (\omega_x)_{x\in A})
	=
	\sum_{y \in Y} \max_{x\in A} \Pg (\E_y ; \omega_x) 
\]
is weakly continuous and monotonically increasing in the post-processing order.
The fourth equality follows from Lemma~\ref{lemm:Pgvec}.
\qed

\section{Concluding remarks} \label{sec:concl}
In this paper, we have investigated general properties of the measurement space
$\ME$ for a given order unit Banach space $E$ with a predual corresponding to a GPT.
Among these general facts, the compactness of $\ME$ (Theorem~\ref{thm:compact})
and the density of finite-outcome measurements (Theorem~\ref{thm:finapp})
are proved to be essential in the applications to simulability and incompatibility 
of measurements with general outcome spaces in Sections~\ref{sec:sim},
\ref{sec:irr},
and \ref{sec:incomp}.
Our study revealed that the compact convex structure 
naturally arises in the measurement space $\ME ,$
whose physical meaning is fundamentally different from the state space
of a general probabilistic theory.
The compact convex structure of the measurement space $\ME$ is introduced
based on the state discrimination probabilities of a finite-label ensembles.
The general theory developed in this paper applies whenever such quantities are involved and not restricted to the specific examples considered in this paper.

Finally we list some related questions that are left for further research.
\begin{enumerate}[1.]
\item
The measurement space $\ME$ has not only topological and convex structures
but also the post-processing order structure.
As we have shown in Theorem~\ref{thm:vnm}, the post-processing order is a special example of the orders characterized by the independence and the continuity axioms. 
From the mathematical point of view, 
this motivates us to ask when such an ordered 
compact convex set can be regarded as a measurement space $\ME ,$
especially for $E$ corresponding to a quantum or a classical system. 
From Theorem~\ref{thm:infinite}, we can see that such compact convex set has an infinite dimension except when it is a singleton.
\item
We can also ask whether the measurement space $\ME$ characterizes the space $E$
up to weakly$\ast$ isomorphism.
To be specific, the question is formalized as follows:
consider order unit Banach spaces $E_1$ and $E_2$ which respectively have the Banach preduals $E_{1\ast}$ and $E_{2\ast}$
and suppose that there exists a continuous, affine, and order isomorphism $\Psi \colon \mathfrak{M} (E_1) \to \mathfrak{M} (E_2) $ between the measurement spaces.
Then is there a weakly$\ast$ continuous, order bi-preserving, linear isomorphism between $E_1$ and $E_2$?
Note that we can easily see that the converse implication holds, namely an isomorphism between $E_1$ and $E_2$ induces an isomorphism between the measurement spaces $\mathfrak{M} (E_1) $ and $\mathfrak{M} (E_2).$
\item
Recently in \cite{ducuara2019weight,ducuara2019operational,uola2019quantum}
it is shown that the weight of resource is related to the ratio of 
state exclusion probability.
Specifically, in \cite{ducuara2019weight} the resource theory of measurements
based on the state exclusion probability is studied.
For an ensemble $\E = (\vph_x)_{x \in X}$
and a measurement $\Gamma \in \Ch  (F \to E) $
the state exclusion probability is given by
\begin{align}
	P_{\mathrm{ex}} (\E ; \Gamma)
	&:= 
	\sup_{\oM \in \evm (X ; F) }
	\sum_{x, x^\prime \in X \colon x \neq x^\prime}
	\braket{\vph_x , \Gamma (\oM(x^\prime))}
	= \Pg (\wt{\E} ; \Gamma) ,
	\label{eq:exprob}
\end{align}
where 
$\wt{\E} := (\sum_{x^\prime \in X \colon x^\prime \neq x } \vph_{x^\prime})_{x \in X} .$
Since \eqref{eq:exprob} is apparently weakly continuous,
the methods developed in Sections~\ref{sec:sim} and \ref{sec:incomp} will be 
straightforwardly generalized to this case. 
\item
We may also ask whether we can generalize our results for measurements to 
more general class of channels with
non-classical outcome spaces.
If we consider the order induced by the state discrimination probability,
this order is the one induced by statistical morphisms,
much weaker notion than that of channels,
and does not coincide in general with the order induces by the post-processing channels~\cite{doi:10.1063/1.5074187}.
For any of these orders, 
the compactness result (Theorem~\ref{thm:compact})
seems to still hold because the classicality of the outcome spaces in the proof
is used only to guarantee the limit channel has also classical outcome space.
\end{enumerate}

\noindent 
\textit{Acknowledgement.}\\
The author would like to thank Hayata Yamasaki for 
helpful discussions and comments,
Masato Koashi for helpful discussions,
and Erkka Haapasalo for helpful comments on the paper.

\appendix 
\section{Proof of Proposition~\ref{prop:convst}} \label{app:convst}
In this appendix we prove Proposition~\ref{prop:convst}.
\begin{enumerate}[1.]
\item
The first part of the claim is easy to verify.
The affinity of $\Psi$ can be shown in the same way as \cite{Gudder1973}
(Theorem~2.2).
The continuity of $\Psi$ is immediate from the definition.
The injectivity of $\Psi$ follows from that 
$\Ac (S)$ separates points of $S .$
Then since $S$ is a compact Hausdorff space, 
we have only to establish the surjectivity of $\Psi .$
Suppose that there exists a state $\phi \in S (\Ac (S)) \setminus \Psi (S) .$
Since $\Psi (S)$ is a weakly$\ast$ compact convex subset of $\Ac (S)^\ast ,$
by the Hahn-Banach separation theorem we can take $f \in \Ac (S)$
such that
$
\sup_{s \in S} f (s) < \braket{\phi , f} .
$
By replacing $f$ with $f + \| f \| 1_S$ if necessary, 
we may assume $f\geq 0 .$
Then we have
\[
	\| f \|
	= \sup_{s \in S} \abs{ f(s) }
	= \sup_{s \in S} f(s)
	< \braket{\phi , f}
	\leq 
	\| \phi \| \|f \|
	=\| f \| ,
\]
which is a contradiction.
Therefore $\Psi$ is a continuous affine isomorphism.
\item
The first part of the claim is again easy to show.
By the definition of the metric on $S ,$ 
we can easily see that $\Phi$ is an isometry.
The affinity of $\Phi$ can be again shown in the same way as in \cite{Gudder1973}.
Consider the locally convex Hausdorff topology $\sigma (\Ab (S) , S)$ 
on $\Ab (S) ,$
which is the pointwise convergence topology on $\Ab (S) .$
Then the unit ball $(\Ab (S) )_1$ is compact in this topology
and by \cite{Kaijser1978} this implies that
$\Ab (S)$ has the Banach predual 
$\overline{\lin (S)} = \overline{E_\ast},$
where the closure is with respect to the norm topology.
Let $B (\supset S)$ be the base of the positive cone of $\overline{E_\ast} . $
Assume $S \subsetneq B$  and take $\psi \in B \setminus S .$
Since $S$ is norm-complete and convex,
the Hahn-Banach separation theorem implies that there exists 
$g \in \Ab (S) = (\overline{E_\ast})^\ast$
such that $\sup_{s \in S} g(s) < \braket{g ,  \psi} .$
As in the proof of the claim~1, 
this yields a contradiction.
Therefore $S=B$ and hence 
$E_\ast = \lin (S) = \lin (B) = \overline{E_\ast}$
is a Banach predual of $\Ab (S)$ with the base $S$
of the predual positive cone $E_{\ast +} .$
\qed
\end{enumerate}

\section{Proof of Proposition~\ref{prop:nb}} \label{app:nb}
In this section we prove Proposition~\ref{prop:nb}.

If $E$ is classical, 
we may assume $(E , u_E ) = (C(X) , 1_X)$ for some compact Hausdorff space $X.$
If we define $B_0 \colon E \times E \to E$ by the pointwise multiplication 
$B_0(f,g) (x) := f(x) g(x) ,$ then we can easily see that  the conditions \eqref{i:nb1} and \eqref{i:nb2} 
hold.

Conversely assume that there exists a bilinear map $B\colon E \times E \to E$
satisfying \eqref{i:nb1} and \eqref{i:nb2}.
Then as in \cite{barnum2006cloning} (Lemma~3), we can show 
\begin{equation}
	\phi \circ B(a,b) = \phi (a) \phi(b)
	\label{eq:Bab}
\end{equation}
for any pure state $\phi \in \de S(E) .$
We show that 
\[X : = \set{\phi \in S(E) | 
\phi\circ B(a,b) = \phi (a) \phi (b) \, (\forall a, b \in E)}
\]
is a Hausdorff topological space in the relative topology of the weak$\ast$ 
topology on $S(E) .$
Take a net $(\phi_i )_{i \in I} $ in $X$ weakly$\ast$ converging to 
$\phi \in S(E) .$
Then
$
\phi \circ B(a,b) 
= \lim_{i \in I} \phi_i \circ B(a,b) 
=\lim_{i \in I } \phi_i (a) \phi_i (b)
=\phi (a) \phi (b) 
$
for any $a,b \in E .$
Therefore, being a closed subset of $S(E) ,$
$X$ is weakly$\ast$ compact.
We define a linear map $\Psi \colon E \to C(X)$
by
$\Psi (a) (\phi) := \braket{\phi , a}$
$(a\in E , \phi \in X) .$
Then $\Psi$ is unital and positive.
Furthermore, by the Krein-Milman theorem, 
for any $a \in E $
\[
\| a\| = \sup_{\phi \in S(E)} \abs{\braket{\phi , a} }
=\sup_{\phi \in \de S(E)}  \abs{\braket{\phi , a} }
=\sup_{\phi \in X} \abs{\braket{\phi , a} }
= \| \Psi (a) \| 
\]
and 
\begin{align*}
	a \geq 0 &\iff 
	\braket{\phi , a} \geq 0 \quad (\forall \phi \in S(E))
	\\
	&\iff 
	\braket{\phi , a} \geq 0 \quad (\forall \phi \in \de S(E))
	\\
	&\iff 
	\braket{\phi , a } \geq 0 \quad (\forall \phi \in X)
	\\
	&\iff
	\Psi (a) \geq 0 .
\end{align*}
where we used $\de S(E) \subset X.$
Thus to show that $\Psi$ is an isomorphism between
the order unit Banach spaces $E$ and $C(X) ,$
it suffices to prove that $\Psi$ is a surjection.
Since $\Psi (a) (\phi) \Psi (b) (\phi) = \braket{\phi ,B(a,b)} = \Psi (B(a,b)) (\phi)$
$(a,b \in E ; \phi \in X)$
the image $\Psi (E)$ is a norm-complete subalgebra of $C(X)$
containing the unit $1_X = \Psi (u_E) .$
Moreover $\Psi (E)$ separates points of $X$ since $E$ separates $S(E) ,$
a fortiori $X (\subset S(E)) .$
Therefore the Stone-Weierstrass theorem implies $\Psi (E) = C(X) ,$
which proves the classicality of $E .$

Let $B^\prime \colon E \times E \to E$ be another bilinear map satisfying (i) and (ii).
Then we can similarly show $\braket{\phi , B^\prime (a,b)} = \phi (a) \phi (b)$
$(\phi \in \de S(E) ; a,b\in E) .$
This implies $\braket{\phi , B(a,b)} = \braket{\phi , B^\prime (a,b)}$
$(\phi \in \de S(E) ; a,b\in E) $
and hence the Krein-Milman theorem implies 
$B(a,b) = B^\prime (a,b) ,$
which proves the uniqueness.
The commutativity and the associativity of $B$ follows again from \eqref{eq:Bab} and 
the Krein-Milman theorem. \qed

\section{Proof of Proposition~\ref{prop:EVM}} \label{app:1}
In this appendix, we prove Proposition~\ref{prop:EVM}.
Throughout this appendix, we fix the system order unit Banach space $E$
and its Banach predual $E_\ast $ corresponding to the system.

An EVM $(X, \Sigma , \oM)$ on $E$ is called regular 
(\cite{busch2016quantum}, Section~4.10)
if $X$ is a compact Hausdorff space,
$\Sigma$ is the Borel $\sigma$-algebra $\mathcal{B} (X)$ of $X ,$
and $\mu^\oM_\psi$ is a regular signed measure for any $\psi \in E_\ast .$
The following Riesz-Markov-Kakutani-type representation theorem 
can be shown similarly as in \cite{busch2016quantum} (Theorem~4.4).

\begin{prop} \label{prop:RMK}
Let $X$ be a compact Hausdorff space.
Then for each channel $\Psi \in \Ch (C(X) \to E)$ there exists a unique
regular EVM $(X , \mathcal{B}(X) , \oM)$ such that 
\[
	\Psi (f)
	=
	\int_X f(x) d\oM (x) \quad (f \in C(X)) .
\]
\end{prop}
\noindent 
\textit{Proof of Proposition~\ref{prop:EVM}.}
Since $F$ is classical, we may assume $F=C(X)$ for some compact Hausdorff space
$X .$
Then by Proposition~\ref{prop:RMK}
there exists a unique regular EVM $(X , \mathcal{B}(X) , \oM)$ on $E$ such that
\[
	\Gamma (f) 
	= \int_X f(x) d\oM (x) 
	\quad
	(f \in F = C(X) ) .
\]
We show $\Gamma \ppeq \Gamma^\oM .$
Let $\gamma^\oM \in \Ch (B(X , \mathcal{B}(X)) \to E)$
be the measurement associated with $\oM .$
Then $\Gamma $ is the restriction of $\gamma^\oM$ to the subalgebra
$C(X) \subset B(X , \mathcal{B}(X)) $
and hence $\Gamma \pp \gamma^\oM \pp \Gamma^\oM .$
 
We now prove
$\gamma^\oM \pp \ovl{\Gamma} ,$
where $\ovl{\Gamma} \in \Chw (C(X)^\aast \to E)$ is the \wstar-extension of 
$\Gamma .$ 
By the ordinary Riesz-Markov-Kakutani representation theorem,
the Banach dual space $C(X)^\ast$ is identified with the set $\mathbf{M} (X)$
of signed regular measures on $X$ with the bilinear form
\[
	\braket{\nu , f}
	=
	\int_X f(x) d\nu (x)
	\quad
	(f \in C(X) , \nu \in \mathbf{M} (X)) .
\]
We define a linear map $\Phi \colon \mathbf{M} (X) \to B(X , \mathcal{B} (X))^\ast$
by 
\[
	\braket{\Phi (\nu) , f}
	:=
	\int_X f(x) d\nu (x)
	\quad 
	(f \in B(X , \mathcal{B} (X)) , \nu \in \mathbf{M} (X)).
\]
Then $\Phi$ is positive and sends a state (i.e.\ a probability measure) 
in $\MX$ to a state in $B(X , \BX)^\ast .$
Therefore the dual map
$\Phi^\ast \colon B(X , \BX )^\aast \to C(X)^\aast (= \mathbf{M} (X)^\ast)$
is a \wstar-channel.
Define a channel $\Psi \in \Ch (B(X ,\BX) \to C(X)^\aast)$
by the restriction of $\Phi^\ast$ to $B(X , \BX) .$
Then for any $f\in B(X , \BX)$ and $\psi \in E_\ast$
\begin{align*}
	\braket{\psi , \ovl{\Gamma} \circ \Psi (f)}
	=
	\braket{\Gamma^\ast (\psi ) , \Psi (f) }
	=
	\braket{\mu_\psi^\oM , \Psi (f) }
	=
	\int_X f(x ) d\mu_\psi^\oM (x)
	=
	\braket{\psi , \gamma^\oM (f)},
\end{align*}
which implies $\gamma^\oM = \ovl{\Gamma} \circ \Psi \pp \ovl{\Gamma } .$

Since $\Gamma \ppeq \ovl{\Gamma}$ by Proposition~\ref{prop:w*extch},
this implies $\gamma^\oM \pp \Gamma $
and hence again by Proposition~\ref{prop:w*extch} we obtain
$\Gamma^\oM \pp \Gamma ,$
which completes the proof.
\qed

\section{Measurement space and types of statistical experiments} \label{app:se}

In this appendix, we discuss the relation between the general theory of \wstar-measurements
developed in this paper and the theory of (classical) statistical 
experiments~\cite{lecam1986asymptotic,torgersen1991comparison}.
It will be shown that there is a one-to-one correspondence between 
the statistical experiments with a given parameter set and 
the \wstar-measurements with the discrete classical space corresponding to the parameter set.
Conversely the class $ \Meas$ of \wstar-measurements for a given input space $E$
is shown to be regarded as a face-like subclass of the \lq\lq{}larger\rq\rq{} class of statistical experiments
with the parameter set $S_\ast (E) .$
The former statement indicates that our results on general \wstar-measurements are 
more general than the corresponding results for statistical experiments~\cite{lecam1986asymptotic,torgersen1991comparison},
while, according to the latter one, we can define a quantity or relation 
known in the general statistical experiments to \wstar-measurements 
by restricting the quantity or relation defined in the \lq\lq{}large\rq\rq{} class of statistical experiments to the class of \wstar-measurements.
We remark that these correspondences are also valid in the setup of 
quantum statistical experiments and post-processing completely positive channels
(\cite{kuramochi2018directedv1}, Section~2.2).

A (classical) statistical experiment is a parameterized family of probability measures.
Formal definition is as follows.
\begin{defi}[Statistical experiment] \label{def:se}
A triple $\bfE = (E, \Theta , \phthin)$ is called 
a (classical) statistical experiment if $E$ is a classical space with a Banach predual,
$\Theta \neq \varnothing$ is a set, and $\phthin \in S_\ast (E)^{\Theta}$ is a family of weakly$\ast$ continuous states indexed by $\Theta . $
$E$ and $\Theta$ are called the outcome (or sample) space and the parameter set of 
$\bfE  ,$ respectively.
For each set $\Theta \neq \varnothing$ the class of statistical experiments with the parameter set $\Theta$ is denoted by $\Exper (\Theta),$ which is a proper class.
\qed
\end{defi}
As in the case of \wstar-measurements or channels, we can define the post-processing (or randomization) order and equivalence relations for statistical experiments:

\begin{defi}[Post-processing relation for statistical experiments] \label{def:sepp}
For any statistical experiments $\bfE = (E , \Theta , \phthin)$ and $\bfF = (F , \Theta , \psthin)$ with the same parameter set $\Theta \nono ,$ we define the following binary relations $\pp$ and $\ppeq .$
\begin{enumerate}[(i)]
\item
$\bfE \pp \bfF$ ($\bfE$ is a post-processing of $\bfF$)
$:\defarrow$
there exists a channel $\Psi \in \Ch (E \to F)$ such that $\phth = \psth \circ \Psi$
for all $\theta \in \Theta .$
\item
$\bfE \ppeq \bfF$ ($\bfE$ is post-processing equivalent to $\bfF$)
$:\defarrow$ 
$\bfE \pp \bfF$ and $\bfF \pp \bfE .$
\end{enumerate}
The relations $\pp$ and $\ppeq$ are binary preorder and equivalence relations on $\Exper (\Theta) ,$ respectively. \qed
\end{defi}

An operational meaning of a statistical experiment $\bfE = (E , \Theta , \phthin)$ is that
the system is governed by the parameter $\theta \in \Theta$
and the system\rq{}s state is prepared to $\phth$ when $\theta$ prevails.
If $\bfE \pp \bfF$ (respectively, $\bfE \ppeq \bfF$),
then the information about $\theta$ when we can access $\bfE$
is at least as much as (respectively, the same as) the information 
when we can access to $\bfF .$

The class of statistical experiments $\Exper (\Theta )$ equipped with the post-processing relations can be identified with a class of \wstar-measurements in the following way.

\begin{prop} \label{prop:sem}
Let $\Theta \nono$ be a set.
For each statistical experiment $\bfE = (E , \Theta , \phthin)$ we define 
a channel $\Ga_\bfE \in \Chw (E \to \linf (\Theta))$ by
\[
	\Ga_\bfE (a) := \sum_\thin \phth (a) \delta_\theta
	\quad (a\in E),
\]
where $\linf (\Theta)$ denotes the classical space of bounded real-valued functions on $\Theta,$ $\delta_\theta := 1_{\{ \theta\}} ,$ and the summation is convergent in the weak$\ast$ topology, or equivalently the pointwise convergence topology, on $\linf (\Theta) .$
Then following assertions hold.
\begin{enumerate}[1.]
\item
The class-to-class map
\begin{equation}
	\Exper (\Theta) \ni \bfE \mapsto \Ga_\bfE \in \meas (\linf (\Theta))
	\label{eq:emmap}
\end{equation}
is bijective.
\item
For any statistical experiments $\bfE = (E , \Theta , \phthin) , \bfF =  (F , \Theta , \psthin)\in \Exper (\Theta) ,$
$\bfE \pp \bfF$ if and only if $\Ga_\bfE \pp \Ga_\bfF .$
\end{enumerate}
\end{prop}
\begin{proof}
\begin{enumerate}[1.]
\item
Take statistical experiments $\bfE = (E , \Theta , \phthin) , \bfF = (F , \Theta , \psthin) \in \Exper (\Theta ) $ and suppose $\Ga_\bfE = \Ga_\bfF .$
Then $E = F$ and 
\[
	\phth (a) = \Ga_\bfE (a) (\theta) = \Ga_\bfF (a) (\theta) = \psth (a)
	\quad (a \in E ; \thin) ,
\]
which implies $\bfE = \bfF .$
Thus \eqref{eq:emmap} is injective.
If $\Ga \in \Chw (E \to \linf(\Theta)) $ is a \wstar-measurement, 
then for each $\thin ,$ $\phth (a) := \Ga (a) (\theta)$ $(a \in E)$ is a weakly$\ast$ continuous state and 
\[
	\Ga (a) = \sum_\thin \phth (a) \delta_\theta =\Ga_\bfE (a) 
	\quad
	(a\in E),
\]
where $\bfE = (E , \Theta , \phthin) \in \Exper (\Theta) .$
Therefore \eqref{eq:emmap} is surjective.
\item
To establish the \lq\lq{}only if\rq\rq{} part of the claim,
suppose $\bfE  \pp \bfF $
and take a channel $\Psi \in \Ch (E \to F)$ such that $\phth = \psth \circ \Psi$ $(\thin) .$
Then 
\[
	\Ga_\bfE (a) 
	= \sum_\thin \phth (a) \delta_\theta 
	= \sum_\thin \psth \circ \Psi (a) \delta_\theta 
	= \Ga_\bfF \circ \Psi (a) 
	\quad
	(a \in E)
\]
which implies $\Ga_\bfE \pp \Ga_\bfF .$
Conversely if $\Ga_\bfE = \Ga_\bfF \circ \Phi$ for some $\Phi \in \Ch (E\to F),$
then by using the injectivity of \eqref{eq:emmap}, we have 
$\phth = \psth \circ \Phi$ $(\thin) ,$ which proves the \lq\lq{}if\rq\rq{} part of the claim.
\qedhere
\end{enumerate}
\end{proof}
By Proposition~\ref{prop:sem} 
we can define the convex combination $\la \bfE \oplus \ola \bfF$ $(\la \in [0,1])$ of two statistical experiments $\bfE , \bfF \in \Exper (\Theta)$ by 
$\Ga_{\la \bfE \oplus \ola \bfF} := \la \Ga_\bfE \oplus \ola \Ga_\bfF .$
If $\bfE = (E , \Theta , \phthin)$ and $\bfF = (F , \Theta , \psthin) ,$
the convex combination is given by
\[
	\la \bfE \oplus \ola \bfF 
	= (E \oplus F  , \Theta , (\la \phth \oplus \ola \psth)_\thin) .
\]
Furthermore, we can define the set $\condi (\Theta)$ of post-processing equivalence classes of statistical experiments by $\condi (\Theta) := \mathfrak{M}(\linf (\Theta)),$
where for each statistical experiment $\bfE \in \Exper (\Theta)$ the corresponding equivalence class is defined by $[\bfE] := [\Ga_\bfE] \in \mathfrak{M} (\linf (\Theta)) .$
In \cite{lecam1986asymptotic,torgersen1991comparison} the equivalence class $[\bfE] \in \condi (\Theta)$ is called the type of $\bfE .$

The weak topology on $\condi(\Theta) = \mathfrak{M} (\linf (\Theta))$ in our sense coincides with the weak topology on $\condi (\Theta)$ in the sense of \cite{lecam1986asymptotic,torgersen1991comparison},
which can be seen from Theorem~7.4.15 of \cite{torgersen1991comparison}.

We next show that the class of \wstar-measurements can be regarded as a special class of 
statistical experiments.
\begin{prop} \label{prop:msmap}
Let $E$ be an order unit Banach space with a Banach predual $E_\ast . $
Define 
\begin{equation}
	\Meas \ni \Ga \mapsto \bfE_\Ga \in \Exper (S_\ast (E))
	\label{eq:msmap}
\end{equation}
by $\bfE_\Ga := (F , S_\ast (E) , (\phi \circ \Ga )_{\phi \in S_\ast (E) })$ for 
$\Ga \in \Chw (F \to E) .$
Then the following assertions hold.
\begin{enumerate}[1.]
\item
The map \eqref{eq:msmap} is injective and affine in the following sense:
\begin{equation}
	\bfE_{\la \Ga \oplus \ola \La} = 
	\la \bfE_\Ga \oplus \ola \bfE_\La
	\quad
	(\la \in [0,1] ; \Ga , \La \in \Meas) .
	\label{eq:exaff}
\end{equation}
\item
A statistical experiment $\bfE = (F , S_\ast (E) , (\xi_\phi)_{\phi \in S_\ast (E)}) \in \Exper (S_\ast (E))$ is in the image of \eqref{eq:msmap} if and only if the map
\begin{equation}
	S_\ast (E) \ni \phi \mapsto \xi_\phi \in S_\ast (F)
	\label{eq:SSaff}
\end{equation}
is affine. $\bfE$ is called affine if the map \eqref{eq:SSaff} is affine.
\item
The image of \eqref{eq:msmap} is a face of $\Exper (S_\ast (E))$ in the following sense:
for any $\la \in (0,1) ,$ and any statistical experiments 
$\bfE = (F , S_\ast (E) , (\xi_\phi)_{\phi \in S_\ast (E)}) $
and 
$\bfF = (G , S_\ast (E) , (\eta_\phi)_{\phi \in S_\ast (E)}) , $
if $\la \bfE \oplus \ola \bfF $ is in the image of \eqref{eq:msmap},
then so are $\bfE$ and $\bfF .$
\item
For any \wstar-measurements $\Ga  \in \Chw (F \to E)$ and $\La  \in \Chw (G \to E),$ $\Ga \pp \La$ if and only if $\bfE_\Ga \pp \bfE_\La .$
\end{enumerate}
\end{prop}
\begin{proof}
\begin{enumerate}[1.]
\item
For \wstar-measurements $\Ga , \La \in \Meas,$ suppose $\bfE_\Ga = \bfE_\La .$
Then $\Ga$ and $\La$ have the same outcome classical space $F$ with a Banach predual and $\phi \circ \Ga = \phi \circ \La$ for all $\phi \in S_\ast (E) .$
Since $S_\ast (E)$ generates $E_\ast , $ this implies $\Ga = \La .$

For any $\la \in [0,1]$ and any \wstar-measurements 
$ \Ga \in \Chw (F\to E)$ and $\La \in \Chw (G \to E)$ we have
\begin{align*}
	\bfE_{\la \Ga \oplus \ola \La}
	&=
	(F \oplus G , S_\ast (E) , (\phi\circ (\la \Ga \oplus \ola \La))_{\phi \in S_\ast (E)})
	\\
	&=
	(F \oplus G , S_\ast (E) , (\la\phi\circ   \Ga \oplus  \ola \phi \circ  \La)_{\phi \in S_\ast (E)})
	\\
	&= 
	\la \bfE_\Ga \oplus \ola \bfE_\La ,
\end{align*}
which proves \eqref{eq:exaff}.
\item
If $\bfE = \bfE_\Ga$ for some \wstar-measurement $\Ga \in \Meas ,$
we can readily see that $\bfE$ is affine.
Conversely suppose that $\bfE$ is affine.
Then the map \eqref{eq:SSaff} is uniquely extended to a bounded linear map
$\Ga_\ast \colon E_\ast \to F_\ast .$
If we define $\Ga \colon F \to E$ by the dual map of $\Ga_\ast ,$
it is easy to show that $\Ga$ is a \wstar-channel and $\bfE = \bfE_\Ga .$
\item
By the claim~2, the assumption implies that the map 
\[
	S_\ast (E) \ni \phi \mapsto 
	\la \xi_\phi \oplus \ola \eta_\phi \in S_\ast (F\oplus G)
\]
is affine.
Then we can easily see that the maps
\[
	S_\ast (E) \ni \phi \mapsto \xi_\phi \in S_\ast (F) ,\quad
	S_\ast (E) \ni \phi \mapsto \eta_\phi \in S_\ast (G)
\]
are also affine, and therefore, again by the claim~2, $\bfE$ and $\bfF$ are in the image of \eqref{eq:msmap}.
\item
The claim readily follows from the definitions of the post-processing orders on $\Meas$
and $\Exper (S_\ast (E))$ and from the injectivity of \eqref{eq:msmap}.
\qedhere
\end{enumerate}
\end{proof}
The affine injection \eqref{eq:msmap} induces the following affine injection for the sets of equivalence classes:
\begin{equation}
	\ME \ni [\Ga] \mapsto [\bfE_\Ga] \in \condi (S_\ast (E)) .
	\label{eq:MEE}
\end{equation}
The image of \eqref{eq:MEE} is a face of $\condi (S_\ast (E)) .$
As for the weak topology, we have
\begin{prop} \label{prop:compactimage}
The map \eqref{eq:MEE} is continuous with respect to the weak topologies on $\ME$ and on $\condi (S_\ast (E)) ,$ respectively,
and hence the image of \eqref{eq:MEE} is a compact face of $\condi (S_\ast (E)) .$
\end{prop}
\begin{proof}
Let $(\vphx)_\xin \in (\linf (S_\ast (E))_\ast)^X $ be an ensemble.
Then each $\vphx$ corresponds to $q_x \in \ell^1 (S_\ast (E))$ such that
\[
	\braket{ \vphx , f} = \sum_{\psi \in S_\ast (E)} f(\psi ) q_x (\psi) 
	\quad
	(f \in \linf (S_\ast (E)))
\]
and $q_x(\psi) \geq 0$ $(\psi \in S_\ast (E)) ,$ where $\ell^1 (\Omega)$ denotes the set of summable real functions on a set $\Omega$ equipped with the $\ell^1$-norm
$\| q \|_1 := \sum_{\omega \in \Omega} | q(\omega) | .$
Then for any \wstar-measurement $\Ga \in \Chw (F \to E)$ and any EVM $\oM   \in \evm (X; F) ,$ we have
\begin{align*}
	\sum_{\xin} \braket{\vphx , \Ga_{\bfE_\Ga} ( \oM (x) )}
	&=
	\sum_\xin \sum_{\psi \in S_\ast (E)}
	q_x (\psi) \braket{\psi , \Ga (\oM (x))}
	\\ 
	&=
	\sum_\xin 
	\left\langle 
	\sum_{\psi \in S_\ast (E)} q_x (\psi) \psi , \, \Ga (\oM(x)) 
	\right\rangle .
\end{align*}
Note that $\sum_{\psi \in S_\ast (E)} q_x (\psi) \psi$ makes sense since 
the summation is at most countable and absolutely convergent with respect to the norm on $E_\ast .$
This implies 
\[
	\Pg (\vphxin ; \Ga_{\bfE_\Ga})
	=
	\Pg (\E ; \Ga ) ,
\]
where 
\[
	\E := \left( \sum_{\psi \in S_\ast (E)} q_x (\psi) \psi  \right)_\xin 
\]
is an ensemble on $E .$
Therefore 
\[
	\ME \ni [\Ga] \mapsto \Pg ((\vphx)_\xin ; \Ga_{\bfE_\Ga}) \in \realn
\]
is weakly continuous on $\ME$ for any ensemble $(\vphx)_\xin ,$
which implies the continuity of
\eqref{eq:MEE}.
\end{proof}
By Proposition~\ref{prop:compactimage}, the measurement space $\ME$ can be regarded as a compact face of the set $\condi (S_\ast (E)) .$
The above proof also shows that the restriction of any gain functional on $\condi (S_\ast (E))$ restricted to (the image of) $\ME$ is a gain functional on $\ME .$
We note that this does not imply that the theory of measurements reduces to that of statistical experiments since in general we cannot obtain all the information about a mathematical structure from another larger structure into which the structure in consideration is embedded.

Conversely, as we have seen in Proposition~\ref{prop:sem}, the statistical experiment
is a special kind of \wstar-measurements. 
Moreover, in the case of statistical experiments, the notions of maximal and simulation irreducible measurements are trivial.
Indeed for $\mathfrak{M}( \linf (\Theta)) = \condi (\Theta) ,$ the maximum element $[\id_{\linf (\Theta)}]$ is the unique maximal, and hence simulation irreducible, measurement and the results in Sections~\ref{subsec:maxmeas}, \ref{subsec:sirr}, and \ref{subsec:irrmain} are trivial and not interesting in this case.

\section{Proof of Theorem~\ref{thm:vnm}} \label{app:vnm}
In this section we prove Theorem~\ref{thm:vnm} in the line of \cite{DUBRA2004118}.

The proof of the following lemma is the same as in 
\cite{DUBRA2004118} and omitted.

\begin{lemm}[\cite{DUBRA2004118}, Lemmas~1 and 2]
\label{lemm:vnm2}
Let $S$ be a compact convex structure and let $\preceq$ be a preorder on $S$
satisfying the independence and continuity axioms of Theorem~\ref{thm:vnm}.
Then the following assertions hold
\begin{enumerate}[1.]
\item
For any $\omega , \nu, \mu \in S$ and any $\la \in (0,1] ,$
the cancellation law
\[
	\la \omega + \ola \mu \preceq \la \nu + \ola \mu
	\implies
	\omega \preceq \nu 
\]
holds.
\item
Define
\begin{equation}
	\Cprec
	:= \set{
	\la (\nu - \omega) \in \Ac (S)^\ast | 
	\la \in (0,\infty) ; \omega , \nu \in S ; \omega \preceq \nu
	}. 
	\label{eq:Cprec}
\end{equation}
Then $\Cprec$ is a convex cone in $\Ac (S)^\ast .$
Moreover for any $\omega , \nu \in S ,$
\[
	\omega \preceq \nu \iff
	\nu -\omega \in \Cprec 
\]
holds.
\end{enumerate}
\end{lemm}

\begin{lemm}[\cite{DUBRA2004118}, Claim~1] \label{lemm:vnm3}
Let $S$ be a compact convex structure and let $\preceq$ be a preorder on $S$
satisfying the independence and continuity axioms of Theorem~\ref{thm:vnm}.
Then $\Cprec$ defined by \eqref{eq:Cprec} is weakly$\ast$ closed.
\end{lemm}
\begin{proof}
Since $\Cprec$ is a convex set from Lemma~\ref{lemm:vnm2},
by the Krein-\v{S}mulian theorem it suffices to show that 
$(\Cprec)_r$ is weakly$\ast$ closed for any $r \in (0,\infty) .$
We take an arbitrary net $(\psi_i)_\iin$ in $(\Cprec)_r $ weakly$\ast$ convergent to 
$\psi \in \Ac(S)^\ast$ and prove $\psi \in \Cprec .$ 
From Proposition~II.1.14 of \cite{alfsen1971compact}, 
by noting that $\braket{\psi_i  , 1_S} = 0 ,$
for each $\iin$ we can write as
\[
	\psi_i =\frac{ \| \psi_i \|}{2} ( \nu_i - \omega_i)
\]
for some $\omega_i , \nu_i \in S ,$ where we take as $\omega_i = \nu_i$
when $\psi_i =0 .$
Then by Lemma~\ref{lemm:vnm2}, $\omega_i \preceq \nu_i$
holds for all $\iin .$
Since $\| \psi_i \| \leq r$ for all $\iin , $
we can take a subnet $(\psi_\ikj)_\jin ,$
a real number $\la \in [0,r] ,$
and states $\omega , \nu \in S$
such that
\[
	\| \psi_{i(j)} \| \to \la ,
	\quad
	\omega_\ikj \to \omega ,
	\quad
	\nu_\ikj \to \nu .
\]
Then \[\psi = \frac{\la}{2} (\nu - \omega ) .\]
Furthermore from the continuity axiom we have $\omega \preceq \nu .$
Therefore $\psi \in \Cprec ,$
which completes the proof.
\end{proof}

\noindent
\textit{Proof of Theorem~\ref{thm:vnm}.}
The implication \eqref{i:ordA}$\implies$\eqref{i:vnm} is trivial.
We assume \eqref{i:vnm} and prove \eqref{i:ordA}.
Define
\[
	\mathcal{U}
	:=
	\set{f \in \Ac (S) | \braket{\psi , f} \geq 0 \, (\forall \psi \in \Cprec)} ,
\]
where $\Cprec$ is defined by \eqref{eq:Cprec}.
Then $\mathcal{U}$ is the dual cone of $\Cprec $ in the pair
$(\Ac (S)  , \Ac (S)^\ast) .$
Since $\Cprec$ is a weakly$\ast$ closed convex cone by Lemmas~\ref{lemm:vnm2}
and \ref{lemm:vnm3}, 
the bipolar theorem implies that 
\[
	\Cprec
	=\set{ \psi \in \Ac (S)^\ast |
	\braket{\psi , f} \geq 0 \, (\forall f \in \mathcal{U})
	} .
\]
Therefore from Lemma~\ref{lemm:vnm2}, for any $\omega , \nu \in S$ we have
\begin{align*}
	\omega \preceq \nu 
	& \iff
	\nu - \omega \in \Cprec
	\\
	&\iff
	\braket{\nu - \omega , f} \geq 0 \quad (\forall f \in \mathcal{U})
	\\
	& \iff
	\omega \preceq_\mathcal{U} \nu .
\end{align*}
Hence $\preceq$ coincides with $\preceq_\mathcal{U},$
which proves \eqref{i:ordA}.

Now we establish the remaining uniqueness part of the claim.
Take subsets $A, B \subset \Ac (S)$ such that $\ordA$ and $\ordB$
coincide.
Then by definition any $f \in A$ is monotonically increasing in $\ordA $
and hence in $\ordB .$
Thus by Theorem~\ref{thm:monoGen} we have 
$ f \in \ovl{\cone} (B \cup \{ \pm 1_S\})$ and therefore
$\ovl{\cone} (A \cup \{ \pm 1_S\}) \subset \ccone (B \cup \{ \pm 1_S\}) $ holds.
The converse inclusion can be shown similarly and hence we obtain 
$\ccone (A \cup \{ \pm 1_S\}) = \ccone (B \cup \{ \pm 1_S\}) .$
Conversely suppose that 
$\ccone (A \cup \{ \pm 1_S\}) = \ccone (B \cup \{ \pm 1_S\}) .$
Then since we can easily see that the orders 
$\preceq_{\ccone (A \cup \{ \pm 1_S\})}$
and 
$\preceq_{\ccone (B \cup \{ \pm 1_S\})}$
respectively coincide with $\ordA$ and $\ordB , $
the orders $\ordA$ and $\ordB$ coincide.
\qed

\section{Minimal sufficiency} \label{app:ms}
In this appendix, we summarize the facts on minimally sufficient \wstar-measurements 
(\cite{kuramochi2017minimal}; \cite{torgersen1991comparison}, Section~7.3)
needed in Section~\ref{sec:irr}.

A \wstar-measurement $\Gamma \in \Chw (F\to E)$ is called \textit{minimally sufficient}
if for any $\Psi \in \Chw (F \to F) ,$
$\Gamma \circ \Psi = \Gamma$ implies $\Psi = \id_F ,$
where $\id_S$ denotes the identity map on a set $S .$
It can be shown that every minimally sufficient \wstar-measurement 
$\Gamma \in \Chw (F\to E)$
is faithful, i.e.\
$\Gamma (a) = 0$ implies $a = 0$ for $a \in F_+ .$
Since a classical space $F$ with a predual is isomorphic to the
self-adjoint part of an abelian \Wstar-algebra,
the following proposition readily follows from \cite{kuramochi2017minimal}
(Corollary~2).
\begin{prop} \label{prop:msmeas}
Let $\Gamma \in \Chw (F\to E)$ be a \wstar-measurement.
Then there exists a minimally sufficient \wstar-measurement
$\Gamma_0 \in \Chw (F_0 \to E)$
post-processing equivalent to $\Gamma .$
Furthermore such a minimally sufficient \wstar-measurement is unique up to isomorphism
of the outcome space, 
i.e.\ if $\Gamma_1 \in \Chw (F_1 \to E)$ is a minimally sufficient \wstar-measurement
and $\Gamma \ppeq \Gamma_1,$
then there exists a weakly$\ast$ continuous isomorphism $\Phi \colon F_0 \to F_1$ such that
$\Gamma_0 = \Gamma_1 \circ \Phi .$
\end{prop}
Let us see how to construct such a minimally sufficient \wstar-measurement $\Gamma_0$
when $\Gamma$ is faithful.
Define $\F \subset \Chw (F\to F)$ 
and $F_0 \subset F$ by
\begin{gather*}
	\F :=
	\set{ \Psi \in \Chw (F \to F) | \Gamma \circ \Psi = \Gamma } ,
	\\
	F_0 :=
	\set{a \in F | \Psi (a)= a \, (\forall \Psi \in \F)} .
\end{gather*}
Then $F_0$ is a weakly$\ast$ closed unital subalgebra of $F$
and by the mean ergodic theorem~\cite{ISI:A1979HD47500016}
there exists weakly$\ast$ continuous conditional expectation (norm-$1$ projection)
$\condi $ from $F$ onto $F_0$ such that
$\condi \circ \Psi = \Psi \circ \condi = \condi$
$(\Psi \in \F)$
and 
$\Gamma \circ \condi = \Gamma .$
Then it can be shown that the restriction $\Gamma_0$
of $\Gamma $ to the subalgebra $F_0$ is 
a minimally sufficient \wstar-measurement and post-processing
equivalent to $\Gamma .$

The following statements can also be shown similarly as 
in the case of the quantum theory
\cite{Martens1990,10.1063/1.4934235,kuramochi2017minimal}.
For a finite-outcome EVM $\oM \in \evm (X;E) ,$
the associated \wstar-measurement 
$\GM \in \Ch ( \linf (X) \to E)$ 
is minimally sufficient if and only if
$\oM$ is pairwise linearly independent,
i.e.\ $(\oM (x ) , \oM (x^\prime))$ is linearly independent 
for any $x , x^\prime \in X$ with
$x \neq x^\prime .$
Every finite-outcome EVM $\oM$ is post-processing equivalent to 
a pairwise linearly independent EVM $\oM_0$ and such $\oM_0$
is unique up to the bijective permutation of outcome sets.

\end{document}